\documentclass[12pt,twoside,a4paper]{article}
\usepackage[english]{babel}
\usepackage[utf8]{inputenc}
\usepackage{amsmath}
\usepackage{amsfonts}
\usepackage[pagebackref]{hyperref}
\usepackage{amsthm}
\usepackage{amssymb}
\usepackage{bm}
\usepackage{tikz-cd}
\usepackage{xr}
\usepackage[nameinlink]{cleveref}

\usepackage{enumitem}
\usepackage{authblk}
\usepackage[title]{appendix}
\usepackage{chngcntr}
\usepackage{apptools}
\usepackage{mathtools}
\usepackage{appendix}

\usepackage{bookmark}
\usepackage{appendix}
\definecolor{crimsonglory}{rgb}{0.75, 0.0, 0.2}
\definecolor{darkpowderblue}{rgb}{0.0, 0.2, 0.6}
\hypersetup{
	colorlinks   = true, 
	urlcolor     = darkpowderblue, 
	linkcolor    = darkpowderblue, 
	citecolor   = crimsonglory 
}

\newtheorem*{remark}{Remark}
\newtheorem*{quest}{Question}
\newtheorem*{remarks}{Remarks}

\newtheorem{prop}{Proposition}

\newtheorem*{question}{Question}

\newtheorem{theorem}{Theorem}
\newtheorem{lemma}{Lemma}
\newtheorem{definition}{Definition}
\newtheorem{cor}{Corollary}
\newtheorem*{thm*}{Theorem}
\newtheorem{conj}{Conjecture}

\newtheorem*{claim}{Claim}
\usepackage{theoremref}

\DeclareMathOperator{\id}{id}
\DeclareMathOperator{\homm}{Hom}

\DeclareMathOperator{\End}{End}
\DeclareMathOperator{\GL}{GL}

\DeclareMathOperator{\diag}{diag}
\DeclareMathOperator{\rank}{rank}
\DeclareMathOperator{\pr}{pr}
\DeclareMathOperator{\tr}{tr}

\DeclareMathOperator{\spec}{Spec}
\DeclareMathOperator{\gal}{Gal}
\DeclareMathOperator{\aut}{Aut}

\DeclareMathOperator{\inv}{inv}
\DeclareMathOperator{\ram}{Ram}
\DeclareMathOperator{\Span}{Span}
\DeclareMathOperator{\LT}{LT}
\DeclareMathOperator{\fin}{fin}
\DeclareMathOperator{\atyp}{atyp}
\DeclareMathOperator{\disc}{disc}

\newcommand{\im}[1]{Im(#1)}
\newcommand{\Sum}[2]{\displaystyle\sum_{#1}^{#2}}
\newcommand{\Bigcup}[2]{\displaystyle\cup_{#1}^{#2}}
\newcommand{\Bigsum}[2]{\displaystyle\bigoplus_{#1}^{#2}}

\newcommand{\Z}{\mathbb{Z}}
\newcommand{\N}{\mathbb{N}}
\newcommand{\C}{\mathbb{C}}
\newcommand{\Q}{\mathbb{Q}}
\newcommand{\R}{\mathbb{R}}

\newcommand{\fgj}{\hat{\gamma}^{*}_j}
\newcommand{\drfj}[1]{\hat{\gamma}^{*}_{{#1},B_{dR}}}

\newcommand{\et}{\acute{e}t}
\newcommand{\proet}{pro\acute{e}t}
\newcommand{\HL}{HL}

\newcommand{\V}{\mathbb{V}}
\usepackage[OT2,T1]{fontenc}
\DeclareSymbolFont{cyrletters}{OT2}{wncyr}{m}{n}
\DeclareMathSymbol{\Sha}{\mathalpha}{cyrletters}{"58}

\usepackage{tocloft}

\numberwithin{theorem}{section}
\numberwithin{prop}{section}
\numberwithin{cor}{section}
\numberwithin{lemma}{section}
\numberwithin{conj}{section}
\numberwithin{definition}{section}

\usepackage{url}
\makeatletter
\newcommand{\address}[1]{\gdef\@address{#1}}
\newcommand{\email}[1]{\gdef\@email{\url{#1}}}
\newcommand{\@endstuff}{\par\vspace{\baselineskip}\noindent\small
	\begin{tabular}{@{}l}\scshape\@address\\\textit{E-mail address:} \@email\end{tabular}}
\AtEndDocument{\@endstuff}
\makeatother
\title{Unlikely intersections in the Torelli locus and the G-functions method}
\author{Georgios Papas}
\date{}
\address{Einstein Institute of Mathematics, Edmond J. Safra Campus,\\	
	Hebrew University of Jerusalem, 
	Givat Ram, Jerusalem, 9190401, Israel}
\email{georgios.papas@mail.huji.ac.il}

\begin{document}
	\maketitle
	
	
	\begin{abstract} Consider a smooth irreducible Hodge generic curve $S$ defined over $\bar{\Q}$ in the Torelli locus $T_g\subset \mathcal{A}_g$. We establish Zilber-Pink-type  statements for such curves depending on their intersection with the boundary of the Baily-Borel compactification of $\mathcal{A}_g$. For example, when our curve intersects the $0$-dimensional stratum of this boundary and $g$ is odd, we show that there are only finitely many points in the curve for which the corresponding Jacobian variety is non-simple.
		
These results follow as a special case of height bounds for exceptional points in $1$-parameter variations of geometric Hodge structures via Andr\'e's G-functions method, which we extend here to the setting of such variations of odd weight.
	\end{abstract}

	\section{Introduction}

\subsection{Motivation}

\subsubsection*{A short review of the Zilber-Pink Conjecture}

 The primary motivation for this paper is the Zilber-Pink conjecture, in its modern form a combination of conjectures made by B.Zilber in \cite{zilber} and R.Pink in \cite{pink}, which aims to generalize conjectures in the field of Unlikely intersections such as the Andr\'e-Oort Conjecture, now a Theorem by \cite{pst} and \cite{gaoandreoort}, in the case of Mixed Shimura varieties, by describing the expected behavior of all atypical intersections of a given variety $X$ with special sets in the ambient space $Z$, in greatest generality a Mixed Shimura variety in the classical case. 

To formulate this conjecture formally we start with the following definition, following the exposition in \cite{pilaicm}.
\begin{definition}Let $Z$ be a mixed Shimura variety and let $\mathcal{S}_Z$ be its set of special subvarieties. 
	
	Consider $X$ a subvariety of $Z$. A component $A$ of $X\cap Y$, where $Y\in\mathcal{S}_Z$, is called an atypical subvariety of $X$ if \begin{center}
		$\dim A> \dim X+\dim Y-\dim Z$.
	\end{center}
	
	We also define $\atyp(X):=\underset{Y\in\mathcal{S}_Z}{\cup} A$ to be the union of all atypical subvarieties of the variety $X$.\end{definition}

A priori $\atyp(X)$ is a countable union of varieties. The Zilber-Pink conjecture predicts that this union is in fact finite. 

\begin{conj}[Zilber-Pink conjecture]\label{zilberpink} Let $Z$ be as in the previous definition and $X\subset Z$ a subvariety of $Z$. Then $\atyp(X)$ is a finite union of varieties. Equivalently, the variety $X$ contains finitely many maximal atypical subvarieties. 
\end{conj}

While the Zilber-Pink conjecture remains largely open some partial results are known to be true, particularly when $X$ is a curve. In fact when $X$ is a curve and the ambient space $Z$ is a torus, the Zilber-Pink conjecture has been established by work of E.Bombieri, P.Habegger, D.Masser, U.Zannier, and G.Maurin in a series of papers \cite{bmz1,bmz2,bhmz,maurin}.

To that end, as a consequence of our exposition, we establish Zilber-Pink-type statements for Hodge generic curves in the Torelli locus. 
\begin{theorem}\label{zpnonsimplejacs}Let $S\subset T_g$ be a Hodge generic smooth irreducible curve in the Torelli locus $T_g$ defined over $\bar{\Q}$. Assume that $g$ is odd with $g\geq 3$ and that the curve intersects the $0$-dimensional stratum of the boundary of the Baily-Borel compactification of $\mathcal{A}_g$. For $s\in T_g$ write $J_s$ for the Jacobian associated it. 
	
	Then the set $\{s\in S(\C):J_s \text{ is non-simple}\}$ is finite.
\end{theorem}

\subsubsection*{The Zilber-Pink Conjecture for variations of Hodge structures}

In \cite{klingler2017hodge} B.Klingler proposed some, even more, far reaching conjectures that are natural analogues of the Zilber-Pink conjecture in the setting of variations of mixed Hodge structures. Klingler's conjectures  in fact imply the ``classical'' Zilber-Pink conjecture, which so far had been only formulated for mixed Shimura varieties. 

The ultimate motivation behind these conjectures is to study the Hodge locus of a variation of mixed Hodge structures $\V$ defined over a smooth quasi-projective complex algebraic variety $S$. In short, the Hodge locus $\HL(S,\V)$ of the variation $(S,\V)$ is the points in $S$ whose Mumford-Tate group is strictly smaller than the generic one, or in other words points for which the fiber $\V_s$ has more Hodge classes than the generic such fiber.

It is known by work of Cattani-Deligne-Kaplan \cite{cdk}, in the case of variations of pure polarized Hodge structures, and Brosnan-Pearlstein-Schnell \cite{bpsc}, in the mixed admissible case, that the Hodge locus is in fact a countable union of closed irreducible algebraic subvarieties of $S$. 

As in the case of Shimura varieties one can naturally define a set of special subvarieties of $S$, or more precisely special subvarieties of $(S,\V)$, since the ``special nature'' of these subvarieties is tied to the information encoded by the variation $\V$. In analogy with the Zilber-Pink conjecture one wants to focus on those special subvarieties that provide ``atypical intersections''. With that in mind, using Hodge-theoretic language, G. Baldi, B. Klingler, and E. Ullmo, in a strengthening of the aforementioned conjectures of Klingler, define the notion of an ``atypical special subvariety'' for a given variation of polarizable integral Hodge structures $(S,\V)$ and define the atypical part $\HL(S,\V)_{\atyp}$ of the Hodge locus of the variation to be the union of all such atypical special subvarieties. They then state the following strengthened version of Conjecture $1.9$ of \cite{klingler2017hodge}:

\begin{conj}[Conjecture $2.5$, \cite{bkullmo}]\label{klinglerzilberpink} Let $S$ be an irreducible smooth quasi-projective complex variety endowed with a variation of polarized integral Hodge structures $\V$.
	Then the atypical part $\HL(S,\V)_{\atyp}$ of the Hodge locus of the variation is a finite union of atypical special subvarieties of $(S,\V)$.
\end{conj}

We note that in the case where $S$ is a curve the aforementioned conjecture reduces to an honest fineteness statement, since then $\HL(S,\V)_{\atyp}$ will be just a finite set of points according to the above conjecture.

For a more detailed introduction to problems of unlikely intersections we direct the interested reader to our primary sources, namely \cite{around,pilabook,zannier}.
	
\subsubsection*{Height bounds}

In a recent series of works, see \cite{daworr,daworr2,daworr3,daworr4,papaszp,davidg}, initiated by ideas of C.Daw and M.Orr, the height bounds on atypical points have played a pivotal role in establishing cases of the Zilber-Pink Conjecture for curves in moduli spaces of abelian varieties. 

Height bounds of this type were first established by Y. Andr\'e, using what is now referred to as the ''G-functions method''.
 \begin{theorem}\label{andreoriginal}[Y. Andr\'e,\cite{andre1989g}, Theorem $1$, Chapter $X$]Let $f:X\rightarrow S$, where $S=S'\backslash \{s_0\}$ with $S'$ a smooth connected curve defined over a number field $K$ and $s_0\in S'(K)$, be a one-parameter family of abelian varieties.
	
	 Assume that the generic fiber $X_{\eta}$ is a simple abelian variety of odd dimension $g>1$, and that the scheme $X$ has completely multiplicative reduction at the point $s_0$. 
	 
	 Then there exist constants $c_1$, $c_2$ such that for all points in the set \begin{center}
	 	$\{ s\in S(\bar{\Q}): \End X_s\not\hookrightarrow M_g(\Q)\} $
	 \end{center} $h(s)\leq c_1\cdot [K(s):K]^{c_2}$, where $h$ is a Weil height on $S'$.
\end{theorem}

The importance of these height bounds, and the G-functions method in particular, becomes evident given the fact that, other than the result of P. Habegger and J. Pila \cite{habeggerpila1} who get the height bounds needed via a different method, they have been the only way so far to establish a so called ''Large Galois orbits Hypothesis'' that is needed in the context of the Pila-Zannier method to establish Zilber-Pink-type statements.

The main technical result of this paper is a height bound analogous to that in \Cref{andreoriginal} but now for geometric variations of Hodge structures, of arbitrary odd weight, defined over some curve. With a view towards \Cref{klinglerzilberpink} of Baldi-Klingler-Ullmo, the question we attempt to answer is thus the following:

\begin{question}Consider a geometric Hodge structure $\V$ of weight $n\geq 1$ coming from some smooth projective morphism $f:X\rightarrow S$ defined over some number field $\bar{\Q}$, with $S$ a smooth irreducible curve. 
	
	Is it it true that the $\bar{\Q}$-points of the subset of the atypical locus $\HL(S,\V)_{\atyp}$ of $\V$ who owe their atypicality to having fibers $\V_s$ with ``atypically large'' algebras of Hodge endomorphisms have height uniformly bounded by their degree of definition as in \Cref{andreoriginal}?	
	\end{question}

While this is clearly implied by \Cref{klinglerzilberpink}, a positive answer to this question can be viewed in itself as a ``first indication'' of the validity of \Cref{klinglerzilberpink} for geometric variations defined over curves.
	\subsection{Our main results}

The main technical result of this paper, see \Cref{maintheorem}, establishes analogues of Andr\'e's height bound in the setting of certain pure polarized Hodge structures of arbitrary weight. The motivation behind this pursuit of ours was two-fold as can be seen by this introduction. On the one hand, establishing such height bounds seems to the author, as is evident in the work of Daw and Orr, to be the most serious obstacle in establishing a significant amount of cases of the Zilber-Pink conjecture in $\mathcal{A}_g$. Furthermore, such height bounds are, with the notable exception of the height bounds used by P. Habegger and J. Pila in \cite{habeggerpila1}\footnote{See in particular inequality $(13)$ in the proof of Lemma $4.2$ in \cite{habeggerpila1}.}, the only tool successfully used in establishing Large Galois orbits hypotheses needed for Zilber-Pink type problems in the setting of Shimura varieties. On the other hand, \Cref{klinglerzilberpink} of Baldi-Klingler-Ullmo seems at the moment to be the most general version of a conjecture of Zilber-Pink type.

\textbf{Our setting:} Let $K$ be a number field and let $S'$ be a smooth geometrically irreducible curve over $K$, not necessarily projective, and fix $s_0\in S'(K)$ a closed point. Let us consider $S$ to be the curve $S'\backslash\{s_0\}$, $X$ a smooth variety over $K$, and let $f:X\rightarrow S$ be a smooth projective morphism that is also defined over $K$ and assume that the dimension of the fibers of $f$ is $n$. 

For each $i\in\{0,\ldots, 2n\}$ the morphism $f$ defines variations of Hodge structures on the analytification $S^{an}$ of $S$, namely the variations given by $R^{i}f^{an}_{*}\Q_{X^{an}_{\C}}\otimes \mathcal{O}_{S^{an}_\C}$. We focus on the variation with $i=n$ and set $\V:= R^{n}f^{an}_{*}\Q_{X^{an}_{\C}}$. We furthermore assume that there exists a smooth $K$-scheme $X'$ and a projective morphism $f':X'\rightarrow S'$ such that:\begin{enumerate}
	\item $f'$ is an extension of $f$, and 
	
	\item $Y=(f')^{-1} (s_0)$ is a union of transversally crossing smooth divisors $Y_i$ entering the fiber with multiplicity $1$.
\end{enumerate} 

Let $\Delta\subset S'^{an}_{\C}$ be a small disk centered at $s_0$ such that $\Delta^{*}\subset S^{an}_{\C}$. From work of Katz\footnote{See in particular the monodromy Theorem in part $VII$ along with the discussion in part $VI$ of \cite{katzregularity}.} it is known that the residue at $s_0$ of the Gauss-Manin connection of the relative de Rham complex with logarithmic poles along $Y$ is nilpotent if we have $(2)$ above. From this it follows, by \cite{steenbrink} Theorem $2.21$, that the local monodromy around $s_0$ acts unipotently on the limit Hodge structure $H^n_{\Q-lim}$. By the theory of the limit Hodge structure we then get the weight monodromy filtration $W_{\bullet}$. We let $h:=\dim_{\Q} W_{0}$.\\

We have the following definition.
\begin{definition}\label{stronglyexceptional}Let $\mathcal{E}\subset S(\C)$ be the set of points for which the decomposition $\V_s=V_1^{m_1}\oplus\cdots\oplus V_r^{m_r}$ of $\V_s$ into simple polarized sub-$\Q$-HS and the associated algebra $D_s:=M_{m_1}(D_1)\oplus \cdots \oplus M_{m_r}(D_r)$ of Hodge endomorphisms are such that:\begin{enumerate}
		
		\item $h> \frac{\dim_\Q V_j}{[Z(D_j):\Q]}$ for some $j$, or
		
		\item there exists at least one $D_i$ that is of type IV in Albert's classification and $h\geq \min\{  \frac{\dim_\Q V_i}{[Z(D_i):\Q] } : i \text{ such that }  D_i=\End_{HS}(V_i)  \text{ is of type IV } \}$,
		\end{enumerate}and furthermore there exists a subset $J\subset \{1,\ldots,r\}$ and integers $0\leq m'_i\leq m_i$ such that 
	\begin{enumerate}
		\item $h(\Sum{i\in J}{} ([Z(D_i):\Q]-1)m'_i)\geq (\Sum{i\in J}{}m'_i \frac{(\dim_\Q V_i)}{2}   )\geq h$, or 
		
		\item $h \geq (\Sum{i\in J}{}m'_i \frac{(\dim_\Q V_i)}{2}  )+1$.
	\end{enumerate}
\end{definition}

\begin{remark}The first two of the above condition reflect the extent to which we can create relations at archimedean places among periods, see \Cref{section:archim} for more on this issue. The last two conditions reflect the restrictions we need in order to create relations among periods at finite places, see \Cref{section:nonarchi} for more on this.
	
	We return to examples that satisfy these conditions in \Cref{section:examples}.
\end{remark}

Our main result, using the above notation, is the following theorem.

\begin{theorem}\label{maintheorem}Let $S'$, $s_0$, and $f:X\rightarrow S$ be as above and all defined over a number field $K$. We assume that the dimension $n$ of the fibers is odd and that the Hodge conjecture holds for the endomorphisms of Hodge structures of the fibers of $f$.
	
	For the variation whose sheaf of flat sections is given by $\V:=R^nf^{an}_{*}\Q_{X^{an}_{\C}}$ we assume the following hold true:\begin{enumerate}
		\item the generic special Mumford-Tate group of the variation is $Sp(\mu,\Q)$, where $\mu=\dim_{\Q} \V_z$ for any $z\in S^{an}$, and 
		\item $h\geq 2$.
	\end{enumerate}
	
	Let $\Sigma:= \mathcal{E}\cap S(\bar{\Q})$.	Then, there exist constants $C_{1}$, $C_2>0$ such that for all $s\in\Sigma$ we have\begin{center}
		$h(s)\leq C_1 [K(s):K]^{C_2}$,
	\end{center}where $h$ is a Weil height on $S'$.\end{theorem}

\begin{remark}
	We note that CM-points of the variation will be in the set $\Sigma$ of this \Cref{maintheorem}. We can also create concrete examples of possible algebras of Hodge endomorphisms for which the conditions that guarantee $s\in\Sigma$ above can be checked fairly easily. We return to this issue in \Cref{section:examples}.
\end{remark}

\subsection{Applications to the case n=1}

\Cref{maintheorem}, as well as some tools needed for its proof, has immediate applications to the case $n=1$. Perhaps the most striking, at least to the author, is the following:

\begin{theorem}\label{application1}Let $f:X\rightarrow S$, where $S=S'\backslash \{s_0\}$ is a curve as before, be a family of smooth irreducible projective curves of genus $g\geq 2$ defined over some number field $K$. Assume that the induced morphism $S\rightarrow \mathcal{M}_g$ is Hodge generic. We furthermore assume that the associated family of Jacobians $F:\mathcal{A}\rightarrow S$ over $S$ is such that it degenerates completely multiplicatively over $s_0$. 
	
Consider the set
	\begin{center}
			$\mathcal{E}_g:=\begin{cases}\{s\in S(\bar{\Q}): \End^{0}_{\C}(\mathcal{A}_s)\neq \Q \}&g\text{ is odd}\\
		 \mathcal{E}\cup	\{s\in S(\bar{\Q}):\End^{0}_{\C}(\mathcal{A}_s) \not\hookrightarrow M_{g} (\Q)  \}&\text{otherwise}
		\end{cases}$
	\end{center}where $\mathcal{E}$ is as defined in \Cref{stronglyexceptional} in the second case.
	Then there exist positive constants $c_1$ and $c_2$ such that for all points $s\in \mathcal{E}_g$ we have $h(s)\leq c_1[K(s):K]^{c_2}$.
\end{theorem}

In other words, we may lift the condition $D_s\not\hookrightarrow M_g(\Q)$ in the case of Andr\'e's original theorem \Cref{andreoriginal}, i.e. when $g$ is odd, when the family of abelian varieties is the family of Jacobians of some family of curves that is furthermore Hodge generic in the above sense.

In \Cref{section:applications} we also discuss how one can derive so called Large Galois orbits for points in these families for which the corresponding Hodge structure satisfies the conditions defining the sets $\mathcal{E}_g$ above, depending on the parity of $g$.

We close off our exposition by establishing some Zilber-Pink type statements in this context. In the context of the above case, i.e. when $g=h$ and $n=1$, we obtain the following:
\begin{theorem}\label{zpgish}Let $S\hookrightarrow \mathcal{M}_g$, with $g\geq 3$, be a Hodge generic smooth irreducible curve defined over $\bar{\Q}$. Assume that the compactification $\bar{S}$ of $S$ intersects the boundary $\bar{\mathcal{M}}_g\backslash \mathcal{M}_g$ at a point $s_0$ and let $S':=S\cup \{s_0\}$. 
	
	Let $f:X\rightarrow S$ be the associated $1$-parameter family of smooth irreducible projective genus $g$ curves. Let $h$ be the toric rank of the connected component of the fiber of the N\'eron model of the generic fiber $\mathcal{A}_{\eta}$ over $S'$ at $s_0$, i.e. $h:=t_{\rank}\tilde{A}'_{s_0}$ where $\tilde{\mathcal{A}}'_{s_0}$ is the above connected component of the identity.

 Assume that $h=g$ and define $\mathcal{E}_g$ to be the set of points in $\mathcal{M}_g(\C)$ defined by the same conditions as the sets of the same name in \Cref{application1}.
	
	Let $\mathcal{S}$ be the subset of $\mathcal{E}_g$ whose points $s$ are such that either the corresponding Jacobian $\mathcal{A}_s$ is isogenous to a non-simple abelian variety, or $s$ is such that $D_s$ is  simple of type $I$ or $II$ in Albert's classification.
	
	Then the set $S(\C)\cap \mathcal{S}$ is finite.
\end{theorem}

In particular, in the case where $g\geq 3$ is odd we recover \Cref{zpnonsimplejacs}.

In the case $g=2$ our result is a restatement of Theorems $1.1$ of \cite{daworr} and $1.4$ of \cite{daworr2} combined. We note that all that is missing for the whole conjecture to hold in this case are  archimedean relations for points whose corresponding Jacobian splits into a product of isogenous elliptic curves. 

For more general examples of Zilber-Pink type statements we point the interested reader to \Cref{section:zpstatements}.
	\subsection{A brief summary of related results}

The G-functions method was introduced, as noted earlier, by Y. Andr\'e in \cite{andre1989g}, who managed to apply this method to the study of $1$-parameter totally degenerating families of abelian varieties. Without access to relations at finite places, Y. Andr\'e \cite{andre1989g}  used ideas centered around Gabber's Lemma to show that if the algebra of a point is ``large enough'' then the point cannot be $v$-adically close to the degeneration for any finite place $v$.

The author's main inspiration to undertake this current work came from results of C. Daw and M. Orr who in a series of papers, see \cite{daworr,daworr2,daworr3,daworr4} employed this method to give unconditional results of Zilber-Pink type in moduli spaces. This attempt of Daw and Orr appears, at least to the author, to have born the only systematic way of tackling the Large Galois orbits hypotheses needed to prove the Zilber-Pink conjecture within the general setting of the Pila-Zannier method.

In the first three of their aforementioned papers Daw and Orr have studied $1$-parameter families in $\mathcal{A}_g$ with $g\geq 2$ with the same restriction, i.e. that of totally multiplicative reduction at a point, getting rid of Andr\'e's condition that $g$ is odd and $g>1$ and replacing it with $g>1$. There they construct relations among values of G-functions at archimedean places for points with large endomorphism algebras. This largeness here is reflected in the condition $\End(A_s)\not\hookrightarrow M_g(\Q)$ that also appears in the original work of Andr\'e as we saw above. They then use these height bounds to establish cases of the Zilber-Pink conjecture for curves in $\mathcal{A}_g$.

 In their latest paper in this series, dealing with curves in $Y(1)^n$, under the same assumption of completely multiplicative reduction at a point, Daw and Orr took the height bounds in question one step further. By managing to define relations at finite places, they were able to show finiteness of intersection of a generic curve, again under the same assumption about its intersection with the boundary $\bar{Y}(1)^n\backslash Y(1)^n$, with special subvarieties of the form $Z(\Phi_N(x_{i_1},x_{i_2}))\cap Z(\Phi_M(x_{i_3},x_{i_4}))$. This result is known to imply the Zilber-Pink conjecture for this curve by previous work of \cite{habeggerpila1}.

Also in the context of Shimura varieties, \cite{papaszp}, the author applies some of the ideas present here to go beyond the case of completely multiplicative reduction in $\mathcal{A}_g$, by relying solely on the archimedean relations and a variation of the aforementioned argument of Andr\'e's to deal with finite places.

Our primary goal here was to extend this circle of ideas beyond the abelian case to general geometric variations of Hodge structures and study points in those with large endomorphism algebras. Following a first version of this text appearing in the author's PhD thesis, D. Urbanik, in a breakthrough paper, employed ideas from relative p-adic Hodge theory to give a geometric interpretation of the values of the G-functions that appear as periods in our setting at finite places as well! It is this new input that allows us to start considering relations among values of G-functions at finite places.

Apart from the ``p-adic realization'' of the values of G-functions D. Urbanik was also able to develop, in the same paper, using input from p-adic Hodge theory, relations among the values of G-functions at finite places when the Hodge structure at that point has a summand that has complex multiplication.  He was also able to develop a variant of the Pila-Zannier method to deduce finiteness of such points in the weight $1$ case. 
	\subsection{Organization of the paper}

Our exposition is made up off independent parts dealing with the ingredients needed to obtain the height bounds in question. We have tried to make each part as independent as possible with the hope that variations to our assumptions about the variations of Hodge structures we deal with are possible and can be made with smaller effort in the future.

We start in \Cref{section:reviewgfunctions} with a short review of G-functions and their connection to periods. The method that Andr\'e uses to obtain his height bounds hinges on two results from the theory of G-functions. First is the fact that among the relative $n$-periods associated to the morphism $f:X\rightarrow S$ there are some that are G-functions. Namely they will, roughly speaking, be the ones that can be written as $\int_{\gamma}\omega$ where $\gamma_z\in\im{(2\pi i N^{*})^{n}}$ for $z\in\Delta^{*}$, where $\Delta^{*}$ and $N^{*}$ are the aforementioned punctured disc and nilpotent endomorphism. The second main result we will need is a result that can be described as a ``Hasse principle'' for the values of G-functions. This is what will ultimately allow us to extract height bounds. 

We then proceed in \Cref{section:hodgereview}, where we review some standard facts about the structure of the algebra of Hodge endomorphisms of a Hodge structure. After this, in \Cref{section:notations} we fix some general notation with the hope of making the exposition easier to follow.

In \Cref{section:derhamendo} we address some technical issues that appear later on in our exposition. Namely, we consider the isomorphism between algebraic de Rham and singular cohomology for a smooth projective variety $Y/k$, where $k$ is a subfield of $\bar{\Q}$:\begin{center}
	$P^n:H^n_{DR}(Y/k)\otimes_k\C\rightarrow H^n(Y^{an},\Q)\otimes_\Q \C$.
\end{center}The singular cohomology is endowed with a Hodge structure and we consider its algebra of Hodge endomorphisms $D$. Later on we will want to create splittings of both de Rham and singular cohomology with respect to actions of $D$ on these vector spaces. To do that we will need to have an action of $D$ on $H^n_{DR}(Y/k)$, which a priori we do not. We show that assuming the absolute Hodge conjecture we may base change $Y$ by a finite extension $L$ of $k$ to obtain such an action that will be compatible with the action of $D$ on $H^n(Y^{an},\Q)$ via the isomorphism $P^n$. We also show that this field extension may be chosen so that its degree is bounded from above only in terms of the dimension $\dim_{\Q} H^n(Y^{an},\Q)$. We believe these results are known to experts in the field, however being unable to find a reference for these arguments we include them here for the sake of completeness.

Our next goal, realized in \Cref{section:trivialrelations}, is to describe the trivial relations among those relative $n$-periods associated to the morphism $f$ which are G-functions. This amounts to describing the polynomials defining the $\bar{\Q}[x]$-Zariski closure of a certain $h\times \mu$ matrix, where $x\in K(S')$ is a rational function that vanishes only at the point $s_0$. This is achieved by a monodromy argument using Andr\'e's Theorem of the Fixed part. 

The next part of our exposition, starting in \Cref{section:pseudocmrelations}, consists of creating relations among the values of the G-functions in question at certain exceptional points that are ``non-trivial'' and ''global''. That means that these do not come from specializing the trivial relations we described earlier and hold at all places for which ''they make sense'', i.e. for all places $v$ for which the value $\xi=x(s)$ in question is in the $v$-adic radius of convergence of the G-functions. This is essentially done in two steps, first for archimedean such places and then for non-archimedean such places. 

After this we put all the aforementioned ideas together in \Cref{section:proofoftheresult} to establish our height bounds. Following that, we see an application of our main theorem to the concrete setting $n=1$, i.e. that of $1$-parameter families of genus $g$ smooth projective curves that degenerate at some point. 

In \Cref{section:applications} we give also some applications of our height bounds to Zilber-Pink-type statements for curves in the moduli space $\mathcal{M}_g$. As an ``in-between-step'' we also show how, following ideas of C. Daw and M. Orr, one can create lower bounds for the Galois orbits of very general such atypical points on curves via results of Masser and W\"ustholz. We end the exposition with a few examples of the points that our method can handle, which obviously gives better results for larger values of the number $h$ of columns of G-functions in the relative period matrix. 

We have also included an appendix on polarizations. The main result we need about polarizations in the text is a description of the relations they define among the $n$-periods. This description in the case where the weight of the Hodge structures is $1$ already appears in \cite{andre1989g}. The description in the case of arbitrary odd weight is not different at all. We include it in this appendix for the sake of completeness.\\

\textbf{Acknowledgements:} This text is an updated version of the first half of the author's PhD thesis. The author thanks his PhD advisor Jacob Tsimerman for countless helpful discussions and suggestions. The author also thanks Chris Daw for answering questions about his work, for pointing out some errors, and for providing many useful comments and suggestions on previous versions of this text that significantly improved the exposition. The author also thanks David Urbanik for answering question about his work and for helpful comments and for pointing out some errors on a previous version of this text. The author also thanks Yves Andr\'e for answering questions about his work. While working on this text the author was partially supported by Michael Temkin's ERC Consolidator Grant 770922 - BirNonArchGeom.

	\part{Background material}
	
	\section{A short review of G-functions in Arithmetic Geometry}\label{section:reviewgfunctions}

G-functions were first introduced by Siegel in \cite{siegel}. We start with a short review of G-functions and we list some of their main properties.\\

\textbf{Notation:} We fix some notation that appears throughout the text. For $v$ a place of a number field $K$ we let $i_v:K\rightarrow \C_v$ denote the inclusion of $K$ into $\C_v$, which denotes a completion of an algebraic closure of $K_v$. For $y\in K[[x]]$ we let $i_v(y)$ denote the element of $\C_v[[x]]$ given via $i_v$ acting coefficient-wise on $y$.\\

\begin{definition}
	Let $K$ be a number field and let $y=\Sum{n=0}{\infty}a_n x^n\in K[[x]]$. Then $y$ is called a \textbf{G-series at the origin} if the following are true: 
	\begin{enumerate} 	
		\item $\forall v\in \Sigma_{K,\infty}$ we have that $i_v(y)\in \C_v[[x]]$ defines an analytic function around $0$,
		\item there exists a sequence $(d_n)_{n\in\N}$ of natural numbers such that 
		\begin{itemize}
			\item $d_n a_m\in \mathcal{O}_K$ for all $m\leq n$,
			\item there exists $C>0$ such that $d_n\leq C^n$ for all $n\in \N$,
		\end{itemize}	
		\item $y$ satisfies a linear homogeneous differential equation with coefficients in $K(x)$.
	\end{enumerate}
\end{definition}

Examples of G-series at the origin are elements of $\bar{\Q}(x)$ without a pole at $0$, the expansion of $log(1+x)$ at $0$, and any element of $\bar{\Q}[[x]]$ which is algebraic over $\bar{\Q}(x)$.

We note, see \cite{dwork}, that we can naturally define ``\textit{G-series at} $\zeta$'', for any $\zeta \in \C$. We also remark that the number field $K$ can be replaced by $\bar{\Q}$ without problems thanks to the third condition, which implies that the $a_i$ are all in some finite extension of $\Q$. Finally, we note that the set of G-series at $\zeta$ forms a ring.

\begin{definition}A \textbf{G-function} is a multivalued locally analytic function $y$ on $\C\backslash {S}$, with $|S|<\infty$, such that 
	for some $\zeta\in\C\backslash {S}$, $y$ can be represented by a G-series at $\zeta$.\end{definition}

Thanks to the Theorem of Bombieri-Andr\'e and the Theorem of Chudnovsky we know that the global nature of a G-function is in fact very much dependent on the fact that it can be locally written as a G-series. That is why, essentially following \cite{andre1989g}, we identify the two notions, especially since we will be only interested at power series centered at the origin.\\


For more on G-functions we direct the interested reader to the excellent introductory text \cite{dwork} and the more advanced \cite{andre1989g}.

	\subsection{A Hasse Principle for G-functions}\label{section:hasseprinciple}

The main tool we will need from the theory of G-functions is a theorem of Andr\'e, that generalizes work of Bombieri in \cite{bombg}, which plays the role of a ``Hasse Principle'' for G-functions. First we need some definitions and to fix some notation.\\

 For the rest of this section consider $y_0,\ldots,y_{m-1}$ to be G-functions with coefficients in some number field $K$. We also define $Y:= (y_0,\ldots,y_{m-1})\in K[[x]]^{m}$, we fix some homogeneous polynomial $p\in K[t_1,\ldots,t_m]$, and a $\xi\in K$.
 
 \begin{definition}\label{localradiusden}For each $j$, we let $R_v(y_j)$ be the radius of convergence of the Taylor series $i_v(y_j)\in \C_v[[x]]$. We also define \begin{center}
 		$R_v(Y):=\min\{R_v(i_v(y_j)):0\leq j\leq m-1 \}$
 	\end{center} and call it the \textbf{local radius of convergence of $Y$ at the place $v$}.
 \end{definition}

\begin{definition}
	1. We say that a relation $p(y_0(\xi),\ldots,y_{m-1}(\xi))=0$ \textbf{holds $v$-adically for some place $v$ of $K$} if $|\xi|_v< R_v(Y)$ and \begin{center}
		$i_v(p)(i_v(y_0)(i_v(\xi)),\ldots,i_v(y_{m-1})(i_v(\xi)))=0$.
	\end{center}
	
	2. A relation like that is called \textbf{non-trivial} if it does not come by specialization at $\xi$ from a homogeneous relation of the same degree with coefficients in $K[x]$ among the $y_i$. Respectively, we call it \textbf{strongly non-trivial} if it does not occur as a factor of a specialization at $\xi$ of a homogeneous irreducible relation among the $y_i$ of possibly higher degree.\\
	
	3. A relation $p(y_0(\xi),\ldots,y_{m-1}(\xi))=0$ is called \textbf{global} if it holds $v$-adically for all places $v$ of $K$ for which $|\xi|_v<\min \{1, R_v(Y) \}$.\end{definition}

\begin{theorem}[Hasse Principle for G-functions,\cite{andre1989g}, Ch VII, \S5.2]\label{hasse}Assume that $Y\in\bar{\Q}[[x]]^m$ satisfies the differential system $\frac{d}{dx}Y=\Gamma Y$ where $\Gamma\in M_m(\bar{\Q}(x))$ and that $\sigma(Y)<\infty$. Let $\Sha_{\delta}(Y)$, resp. $\Sha_\delta'(Y)$, denote the set of ordinary points or apparent singularities\footnote{For the notion of a real or apparent singularity of a differential system see page $406$ of \cite{ince}. We note that points that are neither real or apparent singularities are called \textbf{ordinary points} of a differential system.} $\xi\in\bar{\Q}^*$ where there is some non-trivial, resp. strongly non-trivial, and global homogeneous relation of degree $\delta$.
	
	Then, \begin{center}
		$h(\Sha_\delta(Y))\leq c_1(Y) \delta^{3(m-1)}(\log \delta +1)$, and
	\end{center}\begin{center}
		$h(\Sha_\delta'(Y))\leq c_2(Y) \delta^{m}(\log \delta +1)$.
	\end{center}In particular, any subset of $\Sha_\delta(Y)$ with bounded degree over $\Q$ is finite.
\end{theorem}

\begin{remark}
	The quantity $\sigma(Y)$ is called the size of $Y$. G-functions have finite size. \footnote{For this fact and the definition of the notion of ``size'' of a power series see \cite{andre1989g} Chapter I. }
\end{remark}

	\subsection{Periods and G-functions}\label{section:periodsandgfunctions}

Our primary interest in the theory of G-functions stems from the connection between G-functions and relative periods. We give a brief review of the results in \cite{andre1989g} that highlight this connection together with some basic facts and definitions that we will use later on. We also present a short review of results of D. Urbanik, who, in a recent breakthrough \cite{davidg}, managed to establish a $p$-adic realization of the G-functions that we are interested in.\\

Let $S'$ be a smooth geometrically irreducible curve over some number field $K$, $S=S'\backslash\{s_0\}$, where $s_0\in S'(k)$ is some closed point, and let $x\in K(S')$ be a uniformizing parameter of $S'$ at $s_0$.

We also consider $f:X\rightarrow S$ a smooth projective morphism and we let $n=dim X-1$. Fixing an embedding $i_v:K\rightarrow \C$, we then have the following isomorphism of $\mathcal{O}_{S_{\C}^{an}}$-modules\begin{center}
	$P^{\bullet}_{X/S}:
	H^{\bullet}_{DR}(X/S)\otimes_{\mathcal{O}_S}\mathcal{O}_{S_{\C}^{an}}\rightarrow R^{\bullet}f^{an}_{*}\Q_{X^{an}_\C}\otimes_{\Q_{S^{an}_\C}}\mathcal{O}_{S^{an}_\C}$.\end{center}
In what follows we will be focusing on the isomorphism $P^{n}_{X/S}$, which from now on we will simply denote by $P_{X/S}$. We also let $\mu=\dim_\Q H^{n}(X_z^{an},\Q)$ where $z\in S(\C)$.

This isomorphism is the relative version of Grothendieck's isomorphism between algebraic de Rham and Betti cohomology and it can be locally represented by a matrix. Namely, if we choose a basis $\omega_i$ of $H^{n}_{DR}(X/S)$ over some affine open subset $U\subset S$ and a frame $\gamma_{j}$ of $R_nf_{*}^{an}\Q_{X^{an}_\C}$ over some open analytic simply connected subset $V$ of the analytification $U^{an}_{\C}$, $P_{X/S}$ is represented by a matrix with entries of the form $\int_{\gamma}^{ }\omega_i$.

\begin{definition}We define the \textbf{relative n-period matrix} (over $V$) to be the $\mu\times\mu$ matrix 
	\begin{center}
		$\bigg(\frac{1}{(2\pi i)^n} \int_{\gamma}^{ }\omega_i  \bigg)$.
	\end{center}Its entries will be called the \textbf{relative n-periods}.
\end{definition}

A result we will need in what follows guarantees the existence of G-functions among the relative n-periods under the hypothesis that the morphism $f$ extends over all of $S'$. Namely, let us assume $f$ extends to a projective morphism $f':X'\rightarrow S'$ with $X'$ a smooth $K$-scheme, such that $Y:=f'^{-1}(s_0)$ is a union of smooth transversally crossing divisors $Y_i$ entering the fiber with multiplicity $1$. 

Under these assumptions we know, see \cite{peterssteenbrink} Corollary $11.19$, that the local monodromy is unipotent. Let $\Delta_v$ be a small disk embedded in $S'^{an}_{\C}$ and centered around $s_0$ such that for the punctured disk we have that $\Delta^{*}_v\subset S^{an}$. We let $2\pi iN^{*}$ be the logarithm of the local monodromy acting on the sheaf $R_n(f^{an}_{\C})_{*}(\Q_{X^{an}_{\C}})|_{\Delta^{*}}$.

\begin{definition}\label{m0nperiods}We denote by $M_0 R_n(f^{an}_\C)_{*}(\Q_{X^{an}_{\C}})|_{\Delta^{*}}$ the image of the map $(2\pi iN^{*})^n$. We call $\boldsymbol{M_0}$\textbf{-}$\boldsymbol{n}$\textbf{-period} any relative $n$-period over a cycle $\gamma$ in $M_0 R_n(f^{an}_\C)_{*}(\Q_{X^{an}_{\C}})|_{\Delta^{*}}$. 
\end{definition}

In this context we have the following theorem of Andr\'e:

\begin{theorem}\label{andresexistence}[\cite{andre1989g}, Theorem $2$ p.185]\label{existence} Let $\omega_i$ be a basis of sections of $H^n_{DR}(X/S)$ over some dense open subset of S, such that for any section $\gamma$ of $M_0 R_n(f^{an}_\C)_{*}(\Q_{X^{an}_{\C}})|_{\Delta^{*}}$ the Taylor expansion in $x$ of the relative $M_0$-$n$-periods $\frac{1}{(2\pi i)^n}\int_{\gamma}^{ }\omega_i$ are globally bounded G-functions.\end{theorem}

Let us explicate on this a bit. Andr\'e's Theorem shows the existence of power series $y_{i,j}\in \bar{K}[[x]]$ such that, with respect to the above fixed place $v\in \Sigma_{K,\infty}$, we have that $i_v(y_{i,j})(x(P))=\frac{1}{(2\pi i)^n}( \int_{\gamma_j} \omega_i)_{P}$ for $P$ in some open analytic disc $\Delta_{v,R}\subset (S')^{an}$. The cycles $\gamma_j$ that appear here are a basis of $M_0 R_n(f^{an}_\C)_{*}(\Q_{X^{an}_{\C}})|_{\Delta^{*}}$. 

In other words, upon assuming that the above simply connected open analytic subset $V$ is such that $V\subset \Delta_{v,R}$, we may order the frame $\gamma_j$ above starting with elements that constitute a basis of $M_0 R_n(f^{an}_\C)_{*}(\Q_{X^{an}_{\C}})|_{\Delta^{*}}$.  In this case the aforementioned G-functions appear naturally as the entries of the first columns of the relative period matrix.

It turns out that we may without loss of generality\footnote{We may have to a priori tensor everything by a finite extension of $K$.} assume that the entries of these G-functions are in fact in $K$, and not just in $\bar{K}$. See also \Cref{changeofplace} and the remarks in its section for more on this.
	\subsubsection{p-adic realization of G-functions}

The proof of \Cref{andresexistence} is constructive and one can see\footnote{See \Cref{changeofplace} ahead.} that the G-functions $y_{i,j}(x)$ it produces are in fact ``independent'' of the chosen archimedean place $v\in \Sigma_{K,\infty}$ chosen above. Meaning that if we consider a different archimedean place $w$ then $i_w(y_{i,j})(x)$ naturally appear as the entries of the first $h$ columns of the relative period matrix, where $h:=\dim_{\Q}M_0 R_n(f^{an}_\C)_{*}(\Q_{X^{an}_{\C}})|_{\Delta^{*}}$, over any simply connected open analytic subset of a punctured archimedean disk centered at $s_0$ that is contained in $S^{an}$. Note here that the analytification is taken with respect to the ``new place'' $w$.

It has been a long standing problem whether this also holds for non-archimedean places $w$ of $K$. Namely, given $w\in \Sigma_{K,f}$, can we interpret the power series $i_w(y_{i,j})(x)\in \C_w[[x]]$ as relative periods in some open adic disk centered $\Delta_{w,R}\subset (S')^{ad}$ centered at $s_0$?

This problem was answered affirmatively in recent work of D. Urbanik \cite{davidg}. Here we reinterpret the main result of his exposition in our setting. 

\begin{theorem}\label{padicrealization}[\cite{davidg} Theorem $4.3$] Let $y_{i,j}(x)\in K[[x]]$ be the G-functions of \Cref{andresexistence} and let $\iota:K\rightarrow \C_p$ be a fixed embedding corresponding to some finite place of the field $K$. Let also $R\leq \min\{1,R_v(\vec{y})\}$
	
	Then for all $1\leq j\leq h$ and for all closed points $s\in \Delta^{ad}_{\iota,R}\subset (S')^{ad}_{\iota}$, with $s\neq s_0$, in the component that contains $s_0$ of the preimage of $x$ of a punctured adic open disc defined by $|T|<R$, there exists a a functional $\fgj: H^n(X^{ad}_{s, \proet}, \hat{\Z}_p(n))\rightarrow \Z_p(n)$ such that\begin{center}
		$\frac{1}{t^n}(\drfj{j}\circ \rho^{-1})((\omega_{i})_s)=\iota_v(y_{i,j}(x(s)))$.
	\end{center}
\end{theorem}

In \Cref{padicrealization} $\rho$ stands for the $p$-adic Hodge isomorphism \begin{center}
	$H^n(X^{ad}_{s,\proet}, \Z_p(n))\otimes_{\Z_p}B_{dR}\rightarrow H^n_{DR}(X_s/K(s))\otimes_{K(s)} B_{dR}$, 
\end{center}and $\drfj{j}$ stand for the functional obtained from $\fgj$ obtained by extending scalars to $B_{dR}$. For more on this see $\S$ $1.2$ and $\S$ $4$ of \cite{davidg}.

\begin{remark}We note that Urbanik's result holds in greater generality, in that the degree $N$ of the local systems, i.e. $R^{N}f^{an}\Q$, he studies is not necessarily equal to the dimension of the fibers of the morphism $f$ as in our case.
	
	The distinct advantage of the variations of Hodge structures that we consider is that the number $h$ of columns in the relative period matrix whose entries are G-functions is known by \Cref{andresexistence}! Namely it is the dimension of a local sub-system of the system underlying our variation and this dimension is described in terms of the weight filtration induced from the local monodromy on fibers (archimedeanly) close enough to $s_0$.	
\end{remark}

	\section{Endomorphism algebras of Hodge Structures}\label{section:hodgereview}

One of the central notions we will employ in what follows are the endomorphism algebras of polarized $\Q$-Hodge structures of pure weight. We present here a quick review of the main facts we will need later on about the structure of these algebras, as well as a few standard definitions and notation on Hodge-theoretic notions that we will use.\\

\subsubsection*{Mumford-Tate groups}

Given a $\Q$-Hodge structure, or $\Q$-HS for short, of pure weight, we also get a group homomorphism $\tilde{\varphi}:\mathbb{S}\rightarrow GL(V)_{\R}$ of $\R$-algebraic groups, where $\mathbb{S}$ is the Deligne torus. Let $\mathbb{U}_1$ be the $\R$-subtorus of $\mathbb{S}$ with $\mathbb{U}_1(\R)=\{z\in\C^*:|z|=1\}$ and let $\varphi:=\tilde{\varphi}|_{\mathbb{U}_1}$. 

\begin{definition}
	Let $V$ be a pure weight $\Q$-HS and $\tilde{\varphi}$ and $\varphi$ be as above. The \textbf{Mumford-Tate group} of $V$, denoted by $G_{mt}(V)$, is defined as the $\Q$-Zariski closure of $\tilde{\varphi}(\mathbb{S}(\R))$. The \textbf{special Mumford-Tate group} of $V$, denoted by $G_{smt}(V)$, is defined as the $\Q$-Zariski closure of $\varphi(\mathbb{U}_1(\R))$.
\end{definition}

We review a key property of the special Mumford-Tate group. First we will need some notation, following the exposition in \cite{ggkbook}. For a given pure $\Q$-HS $(V,\phi)$ of weight $n$, as in the previous definition, we can get via standard linear-algebraic constructions, a countable family of $\Q$-HS defined as\begin{center}
	$T(V)^{a,b}:=V^{\otimes k}\otimes (V^{*})^{\otimes b}$,
\end{center}where $a,b\in \Z_{\geq0}$. In each of these $\Q$-HS we have Hodge classes, as long as $2|n(a-b)$. One then defines $Hg^{a,b}(V)$ to be the  set of Hodge classes in $T(V)^{a,b}$.

The special Mumford-Tate group is then characterized by the following
\begin{lemma}\label{keypropertyofmt}Let $(V,\phi)$ be a pure polarized $\Q$-HS with polarization given by the bilinear form $Q$. Then $G_{mt}(V)$ is the subgroup of $\aut(V,Q)$ that fixes $Hg^{a,b}(V)$ for all $a,b\in \Z_{\geq0}$.
\end{lemma}

Now let us consider a quasi-projective complex variety $S$ and $(\mathbb{V},\mathcal{F}^{\bullet})$ a variation of polarized $\Q$-HS on $S^{an}$. A natural question to ask here is how the special Mumford-Tate groups $G_{smt}(\mathbb{V}_t)$ vary for $t\in S^{an}$.

First of all, we let $V_0:= \Gamma(\tilde{S}^{an}, \pi^{*} \mathbb{V})$, where $\pi:\tilde{S}^{an}\rightarrow S^{an}$ is a universal covering of $S^{an}$. Since $\pi^{*}\mathbb{V}$ is a constant local system, for $\tilde{t}\in \tilde{S}^{an}$ we have a natural identification $\mathbb{V}_{\tilde{t}}\xrightarrow{\simeq} V_0$. Under this identification, we get natural a natural inclusion $G_{smt}(\mathbb{V}_{\tilde{t}})\subset \GL(V_0)$.

\begin{lemma}\label{lemmaexclocus}There exists a subset $\tilde{S}^{exc}\subset \tilde{S}^{an}$, which is a countable union of proper irreducible analytic subvarieties of $\tilde{S}^{an}$, such that for all $x,y\in  \tilde{S}^{an}\backslash{\tilde{S}^{exc}}$ we have that $G_{smt}(\mathbb{V}_x)=G_{smt}(\mathbb{V}_y)$. Furthermore, for $z\in \tilde{S}^{exc}$ and $y\in  \tilde{S}^{an}\backslash\tilde{S}^{exc}$ the group $G_{smt}(\mathbb{V}_z)$ is a proper subgroup of $G_{smt}(\mathbb{V}_y)$.
\end{lemma}
	
	For a proof see the discussion in $\S 6$ of \cite{moonenintro}, on which this part is heavily based, and Chapter $III$ of \cite{ggkbook}.\\
	
	Following this we have the following natural definitions.
	
	\begin{definition}\label{defhodgegener}Let $S$ be a quasi-projective complex variety as above and $(\mathbb{V},\mathcal{F}^{\bullet})$ be a variation of polarized $\Q$-HS of pure weight on $S^{an}$. Let $\tilde{S}^{exc}$ be the set described in \Cref{lemmaexclocus}.\\
		
		1. The set $S^{gen}:=\pi(\tilde{S}^{an}\backslash{\tilde{S}^{exc}})$ is called the \textbf{Hodge generic locus} of the variation and its points are called the \textbf{Hodge generic points} of the variation.\\
		
		2. The \textbf{special Mumford-Tate group of the variation} is defined to be the special Mumford-Tate group of any of the generic points of the variation.
	\end{definition}
	
	\begin{remark}We note that analogous results hold for the (non-special) Mumford-Tate group. We focus more on special Mumford-Tate groups since they are the notion that we will use out off the two.
	\end{remark}

\subsubsection*{Irreducible Hodge Structures: Albert's Classification}

It is well known that the category of polarizable $\Q$-Hodge structures is semi-simple. This implies that for a polarizable $\Q$-HS $V$, its endomorphism algebra $D:=\End(V)^{G_{smt}(V)}$ is a semi-simple $\Q$-algebra. If, furthermore, the polarizable $\Q$-HS $V$ that we consider is simple, then $D$ is a simple division $\Q$-algebra equipped with a positive involution, naturally constructed from the polarization.

Such algebras are classified by Albert's classification.

\begin{theorem}[Albert's Classification,\cite{mumfordabelian}]\label{albert} Let $D$ be a simple $\Q$-algebra with a positive (anti-)involution $\iota$, denoted $a\mapsto a^{\dagger}$. Let $F=Z(D)$, be the center of $D$, $F_0=\{a\in F: a=a^{\dagger}\}$, $e_0=[F_0:\Q]$, $e=[F:\Q]$, and $d^2=[D:F]$. Then $D$ is of one of the following four types:\\
	
\textbf{Type I:} $D=F=F_0$ is a totally real field, so that $e=e_0$, $d=1$, and $\iota$ is the identity.\\
	
\textbf{Type II:} $D$ is a quaternion algebra over the totally real field $F=F_0$ that also splits at all archimedean places of $F$. If $a\mapsto a^{*}=tr_{D/F}(a)-a$ denotes the standard involution of this quaternion algebra, then there exists $x\in D$ with $x=-x^{*}$ such that $a^{\dagger}=xa^{*}x^{-1}$ for all $a\in D$. Finally, in this case $e=e_0$ and $d=2$.\\

\textbf{Type III:} $D$ is a totally definite\footnote{We remind the reader that a quaternion algebra $B$ over a number field $F$ is called totally definite if for all archimedean places $v\in\Sigma_{F,\infty}$ we have that the algebra $B$ is ramified at $v$. This requires that $F$ is totally real so that $B\otimes_F F_v\simeq \mathbb{H}$, with $\mathbb{H}$ the standard quaternion algebra over $\R$, for all $v\in\Sigma_{F,\infty}$.} quaternion algebra over the totally real field $F=F_0$. In this case $\iota$ is the standard involution of this quaternion algebra and as before $e=e_0$ and $d=2$.\\

\textbf{Type IV:} $D$ is a division algebra of rank $d^2$ over the field $F$, which is a CM-field with totally real subfield $F_0$, i.e. $e=2e_0$. Finally, the involution $\iota$ corresponds, under a suitable isomorphism $D\otimes_{\Q}\R\overset{\simeq}{\rightarrow}M_d(\C)\times\ldots\times M_d(\C)$, with the involution $(A_1,\ldots,A_{e_0})\mapsto(^t\bar{A}_1,\ldots,^t\bar{A}_{e_0})$

Furthermore, in this case we have that for $\sigma$ a generator of $\gal(F/F_0)$ the following must hold:\begin{enumerate}
	\item if $v\in\Sigma_{F,f}$ is such that $\sigma(v)=v$ we have that $\inv_v(D)=0$, and 
	
	\item for all $v\in \Sigma_{F,f}$ we must have that $\inv_v(D)+\inv_{\sigma(v)}(D)=0$.
\end{enumerate}\end{theorem}

\subsubsection*{The general case}

Let $(V,\phi)$ be a polarized $\Q$-HS of weight $n$. Then, combining the semi-simplicity of the category of polarized $\Q$-HS and \Cref{albert} we get a good description of the endomorphism algebra $D=\End(V)^{G_{smt}(V)}$. 

Indeed, we know that there exist simple polarized weight $n$ sub-$\Q$-Hodge structures $(V_i,\varphi_i)$ with $1\leq i\leq r$, such that $V_i\not\simeq V_j$ for all $i\neq j$ and we have a decomposition\begin{equation}\label{eq:decompirrhs}
	V=V_1^{m_1}\oplus\ldots\oplus V_r^{m_r}.
\end{equation}Denoting by $D_i:=\End(V_i)^{G_{smt}(V_i)}$ the corresponding endomorphism algebras and by $F_i:=Z(D_i)$ their respective centers, we then have a decomposition 
\begin{equation}\label{eq:decompalgebras}
D=M_{m_1}(D_1)\times \ldots\times M_{m_r}(D_r).
\end{equation}Finally, this implies that the  center $F$ of $D$ is such that\begin{equation} \label{eq:centersemi}
F=F_1\times\ldots\times F_r,
\end{equation}where each $F_i$ is diagonally embedded into $M_{m_i}(D_i)$, and the maximal commutative semi-simple sub-algebra $E$ of $D$ may be written as\begin{equation}\label{eq:maxcommss}
E=F_1^{m_1}\times\ldots \times F_r^{m_r}.
\end{equation}

For a proof of Albert's classification see \cite{mumfordabelian}, \S 21. For more on Mumford-Tate groups we direct the interested reader to our sources for this section, which are mainly \cite{moonennotes} and \cite{ggkbook}.	
	
	\section{The main setting-notational conventions}\label{section:notations}

Before delving into the technical parts of our argument, we devote this section to describing the general setting that we will be working on in more detail. We give the definitions of the main objects and introduce the notation that we will, unless otherwise stated, keep uniform throughout our exposition.

\subsection{G-admissible variations of Hodge Structures}

Let $S'$ be a smooth geometrically irreducible curve over some number field $K\subset \bar{\Q}$ and let $s_0\in S\subset S'(K)$ be a fixed closed $K$-point of $S'$. We let $S=S'\backslash\{s_0\}$ and let $\eta$ be the generic point of $S$. 

Let us consider $f:X\rightarrow S$ a smooth projective morphism and let $n=\dim X-1$. Assume $f$ extends to a projective morphism $f':X'\rightarrow S'$ with $X'$ a smooth $K$-scheme and that $Y=f'^{-1}(s_0)$ is a simple normal crossings divisor. We record the definition for the convenience of the reader.
\begin{definition}An effective divisor $Y=\Sigma Y_i$ on $X$, where $\dim X=n$, is a \textbf{simple normal crossings divisor} if $Y$ is reduced, each component $Y_i$ is smooth, and the components $Y_i$ intersect transversally.\end{definition}

The map $f$ defines a variation of polarized $\Q$-HS of weight $n$ over $S^{an}_{\C}$ given by $R^nf^{an}_{*}\Q_{X^{an}_\C}$. We denote by $G_{mt,p}$, respectively by $G_{smt,p}$, the Mumford-Tate group, or respectively the special Mumford-Tate group, associated to the $\Q$-HS associated to the point $p\in S(\C)$. We also let $G_{mt,\eta}$, respectively $G_{smt,\eta}$, be the generic Mumford-Tate group, or respectively the generic special Mumford-Tate group, of the variation. For each $p\in S(\C)$ we also let $V_p=H^n(X^{an}_p,\Q)$ be the fiber of the local system $R^nf^{an}_{*}\Q_{X^{an}_\C}$ and let $\mu=\dim_\Q V_p$.

Consider $z\in S(\C)$ to be a  Hodge generic point for the above variation of $\Q$-HS. The main invariant of the variation we will be interested in is the $\Q$-algebra\begin{center}
	$D:=\End(V_z)^{G_{smt,z}}=\End(V_z)^{G_{smt,\eta}}$.
\end{center}
Similarly, for $s\in S(\C)$ we let \begin{center}
	$D_s:= \End(V_s)^{G_{smt,s}}$.
\end{center}

\begin{definition}Let $X$, $S$, $s\in S(\C)$, $D_s$, and $D$ be as above. 
	We call $D_s$ the \textbf{algebra of Hodge endomorphisms at $s$}.
	\end{definition}

\begin{definition}
	A variation of Hodge structures such as above, meaning a weight $n$ geometric variation of $\Q$-HS parameterized by $S$ whose degeneration at some $s_0\in S'$ is as above, with $S=S'\backslash\{s_0\}$, with all of the above defined over some number field $K$, will be called \textbf{G-admissible}.
\end{definition}	
	\subsection{Good covers}\label{section:goodcovers}

Consider a pair of morphism $f':X'\rightarrow S'$ and $f:X\rightarrow S$ underlying a G-admissible variation of $\Q$-HS, as above, and let $h:=\rank_{\Q} M_0 R_n(f^{an}_\C)_{*}(\Q_{X^{an}_{\C}})|_{\Delta^{*}}$, where this is defined as in \Cref{section:periodsandgfunctions}. 

In recent work, see \cite{daworr4}, C. Daw and M. Orr, establish the existence of what we \footnote{See also the discussion in $\S$ $6.1$ of \cite{davidg} on which our exposition is heavily based.} call ``good covers of the curve $S'$''. Namely they show that, up to extending the base field $K$ by a finite extension, there exists a finite morphism $c:C'_{S}\rightarrow S'$ of the curve $S'$ by a smooth geometrically irreducible curve $C'_{S}/K$ that comes equipped with a non-constant rational function $x\in K(C_S)$ such that the following conditions are satisfied:
\begin{enumerate}
	\item the zeros $\{s_1,\ldots,s_{t_0}\}\subset C'_S(\bar{K})$ of $x$ are all simple,
	
	\item $c(s_t)=s_0$ for all of the above zeroes, and 
	
	\item the map $x:\tilde{C}'_S\rightarrow \mathbb{P}^1_{K}$, induced from $x$, is a finite Galois covering, where $\tilde{C}'_S$ is some fixed completion of $C'_{S}$.
\end{enumerate}

\begin{definition}\label{goodcovers} Any such curve $C'_S$ will be called a \textbf{good cover} of the pair $(S',s_0)$, or simply a good cover of the curve $S'$.
\end{definition}

By pulling back the morphisms $f$ and $f'$ we get families $F':X'_{C'_S}\rightarrow C'_S$ and $F:X_{C_S}\rightarrow C_S$, where $C_S:=C'_S\backslash\{s_t:1\leq t\leq t_0\}$. From now on we let $\mathcal{X'}:= X'_{C'_S}$ and $\mathcal{X}:=X_{C_S}$. For this pair of morphisms we furthermore have that now the curve $C'_S$, the target of the morphism $F'$, is equipped with a rational function $x$ that has only simple zeroes. We also know that the zeroes of this rational function correspond to singular values of the morphism $F'$.


Most importantly\footnote{See \cite{davidg} $\S$ $6.1$.}, for all such $s_t$, for all finite extensions $L/K$, and for all places $v\in\Sigma_{L}$ we can find an analytic, either in the complex-analytic for archimedean such places or analytic in the adic sense for finite such places, component of the analytic subset $|x(s)|_v\leq \min\{1,R_v(\mathcal{G})\}$ of the analytification ${C'}_{S}^{an}$, either in the complex analytic or adic sense accordingly, that contains the point $s_t$. In particular in the complex analytic sense we can always find $\Delta_t\subset (C'_S)^{an}$ which will be the homeomorphic image of a disk centered at $s_t$.\\

For $1\leq t\leq t_0$ define  $M_{0}^{(t)}:=M_0 R_n(F^{an}_\C)_{*}(\Q_{{X_{C_S}}^{an}_{\C}})|_{\Delta_{t}^{*}}$. We will then need the following lemma.

\begin{lemma}\label{goodcoverslemma}Let $f$, $f'$, $F$, and $F'$ be as above. Let $h_t:=\rank_{\Q} M_{0}^{(t)}$ for $1\leq t\leq t_0$. 
	
	Then $h_t=h$ for all $1\leq t\leq t_0$.
\end{lemma}

\begin{proof}This practically follows from the proof of Lemma $5.1$ of \cite{daworr4}. In the notation of loc. cit. let $C=\tilde{S}'$ be a smooth geometrically irreducible projective model of our smooth (not necessarily projective) curve $S'$. Then in the notation of the proof of Lemma $5.1$ in loc. cit. one has that $C_1=C$, since $C$ in now smooth, and $C_3=C_2$, since again $C_2$ will be smooth in this case as well. 
	
	Let now $C_4=:\tilde{C}'_S$ be as in the proof of the above Lemma of Daw and Orr and fix $1\leq t\leq t_0$ as above. Note that the morphism $\nu:C_4\rightarrow C$ they construct will be \'etale when restricted to a small enough punctured archimedean disk $\Delta_t^{*}$ centered at $s_t$. Indeed, let $\ram(\nu)\subset C_4$ be the finite set of ramification points of the non-constant morphism $\nu$ and let $C_5:=C'_{S}\backslash (\ram(\nu) \cup \{x^{-1}(0)\})$.  Now consider the restriction $\phi:=\nu|_{C_5}:C_5\rightarrow S_0$, where $S_0$ is the complement of a finite subset of $S$ and take its analytification $\phi^{an}$ with respect to a fixed archimedean place of $K$. Note that by construction $\phi$ will be \'etale. By the properties of good covers highlighted above, we can then find $\Delta^{*}_t\subset {C'}^{an}_S\cap C_5^{an}$ which will be the homeomorphic image of a small punctured disc centered at $s_t$, such that $\phi^{an}(\Delta^{*}_t)\subset\Delta^{*}_0$, where $\Delta^{*}_0\subset S_0^{an}$ is again the homeomorphic image of a punctured analytic disk centered at $s_0$. 
	
	We let $\psi:=\phi|_{\Delta_t^{*}}$ and note that it will map the loop $\gamma_t$ generating $\pi_1(\Delta_t^{*})$ to some power $\gamma_0^d$ of the loop $\gamma_0$ generating $\pi_1(\Delta_0^{*})$. 
	
	Let $\rho_S$ and $\rho_C$ denote the local monodromy representations on the local systems $R_n(f^{an}_\C)_{*}(\Q_{{X}^{an}_{\C}})|_{\Delta_{0}^{*}}$ and $R_n(F^{an}_\C)_{*}(\Q_{{X_{C_S}}^{an}_{\C}})|_{\Delta_{t}^{*}}$ respectively. By \Cref{andresexistence} we know that $h$ is equal to the weight $0$ part of the filtration of the Mixed Hodge structure associated to the degeneration at $s_0$, and similarly for $h_t$. 
	
	By our assumptions, i.e. the assumption that at $s_0$ and therefore also at $s_t$ we have singularities coming from divisors with simple normal crossings, the local monodromy acts unipotently in both cases. In particular, we can write $\rho_S(\gamma_0)=I+N_0$ and $\rho_C(\gamma_t)=I+N_t$, where $N_0$ and $N_t$ are both nilpotent. Since $\gamma_t\mapsto \gamma^d_0$ we get that \begin{equation}
		N_0=N_t\cdot A
	\end{equation}where $A$ will be invertible. 

The equality of $h_t$ and $h_0$ then follows by \Cref{m0nperiods}, from which we get that $h_t=\rank(N_t)$ and $h_0=\rank(N_0)$, but since $A$ above is invertible, we get that $\rank(N_0)=\rank(N_t)$.\end{proof}

Now consider, for $1\leq t\leq t_0$, the pairs $(C_t',C_S)$, where $C'_t:=C_{S'}\backslash \{s_i:i\neq t\}$ and $C_S:=C'_{t}\backslash \{s_t\}$, the latter being independent of $t$, together with the morphisms $F'_t:\mathcal{X}_t':=\mathcal{X}'_{C'_t}\rightarrow C'_{t}$ and $F:\mathcal{X}:=\mathcal{X}_{C}\rightarrow C_S$.

Then for each of these pairs of curves and pairs of morphisms \Cref{andresexistence} and \Cref{padicrealization} apply and give us, for each $1\leq t\leq t_0$, a family of G-functions that we will denote by $\mathcal{G}^{(t)}$.

\begin{definition}\label{defnfamilygfuns}Let $f':X'\rightarrow S'$, $s_0$ be the data defining a G-admissible variation of $\Q$-HS. Then we call the family \begin{equation}\label{eq:familyofgfuns}
		\mathcal{G}:=\mathcal{G}^{(1)}\sqcup\ldots\sqcup \mathcal{G}^{(t_0)}
	\end{equation}of G-functions constructed above, the \textbf{family of G-functions associated to our G-admissible variation}.
\end{definition}

Most importantly for us, we know from \Cref{goodcoverslemma} that each of these sub-families $\mathcal{G}^{(t)}$ has the same cardinality. Furthermore\footnote{See also \Cref{changeofplace}.}, for any fixed archimedean embedding $K\hookrightarrow \C$ the elements of $\mathcal{G}^{(t)}$ can be naturally identified, after fixing suitable bases and local frames, with the first $h$-columns of the relative period matrix in a region on $C_S^{an}$ that is close enough to the ''degeneration'' $s_t$.

Finally, we note that throughout our exposition we drop the mention of $S$ in the notation of the good cover, writing simply $C$ and $C'$ for the above curves, since we will, at almost all times, consider a fixed $f':X'\rightarrow S'$.\\

\textbf{Notation:} We fix some notation that appears throughout the text. By $\Sigma_K$, $\Sigma_{K,f}$, $\Sigma_{K,\infty}$ we denote the set of all places of a number field $K$, respectively finite or infinite places of $K$. As mentioned earlier, for $v\in \Sigma_K$ we let $i_v:K\rightarrow \C_v$ denote the inclusion of $K$ into $\C_v$. For $y\in K[[x]]$ we let $i_v(y)$ denote the element of $\C_v[[x]]$ given via $i_v$ acting coefficient-wise on $y$.

For a scheme $Y$ defined over a field $k$ we let $\bar{Y}:= Y\times_{\spec k}\spec \bar{k}$ and $Y_{L}:=Y \times_{\spec k}\spec L$ for any extension $L/k$.

	\section{Hodge Endomorphisms and De Rham Cohomology}\label{section:derhamendo}

Let $K$ be a number field and $f:X\rightarrow S$  a smooth projective $K$-morphism of $K$-varieties, with $S$ a curve as above. Let us also consider a point $s\in S(L)$ for some finite extension $L/K$ and set $Y:=X_s$ which is a smooth projective variety defined over $L$. 

In what follows we will need the existence of a natural action of the algebra of Hodge endomorphisms of $H^n(Y,\Q)$ on both sides of the comparison isomorphism \begin{equation}
	P^n:H^n_{DR}(\bar{Y}/\bar{L})\otimes_{\bar{L}} \C \rightarrow H^n(\bar{Y}_{\C}^{an},\Q)\otimes_{\Q} \C,
\end{equation}such that these actions commute with this  isomorphism. 

In the case of abelian varieties this is automatic from the fact that the algebra of Hodge endomorphisms is naturally realized as the algebra of endomorphisms of the abelian variety. This in turn acts naturally on both sides of the comparison isomorphism and the actions commute with the isomorphism itself. In a general variety $Y$ we cannot hope for such a description without assuming the validity of the absolute Hodge Conjecture.

It is the author's belief that the results in this section are known to experts in the field. Since we were not able to find an exact reference of the results we needed we have dedicated this section to providing proofs for these results.

\subsection{Existence of the action}

For the rest of this subsection we fix a number field $L$ and a smooth projective $n$-dimensional variety $Y$ defined over $L$.

\begin{prop}\label{propendodr}Let $Y$ be a smooth projective variety over the number field $L$ of dimension $n$. Let $V:=H^n(\bar{Y}^{an}_{\C}, \Q)$ and $D:=\End_{HS}(V)$ the algebra of Hodge endomorphisms. Then,\textbf{ assuming the absolute Hodge Conjecture}, there exists a finite Galois extension $\hat{L}$ of $L$ such that there exists an injective homomorphism of algebras \begin{center}
		$i: D\hookrightarrow  \End_{\hat{L}}(H^n_{DR}(Y/L)\otimes_{L} \hat{L})$.
	\end{center}

Moreover, we have that $P^n(i(d) v)= d\cdot P^n(v)$ for all $d\in D$ and all $v\in H^n_{DR}(\bar{Y}/\bar{L} )\otimes_{\bar{L}} \C$. In other words, the action of the algebra $D$, that is induced by $i$, on de Rham cohomology coincides with the usual action of $D$ on the Betti cohomology as endomorphisms of the Hodge structure under the comparison isomorphism $P^n$.
\end{prop}

\begin{proof}
	We start with some, well known, observations. First of all, the natural isomorphism \begin{center}
		$\alpha_0:\End_{\Q}(V)\cong V\otimes_{\Q} V^{*}$
	\end{center} is an isomorphism of $\Q$-HS. In particular, via $\alpha_0$ the elements of $D$ correspond to Hodge classes\footnote{See Lemma $11.41$ of \cite{voisinhodge1}. }.
		
	It is also known that the isomorphism $\alpha : H^n(\bar{Y}^{an}_{\C},\Q)^{*}\rightarrow H^n(\bar{Y}^{an}_{\C},\Q)(n)$, given by Poincar\'e duality, is an isomorphism of $\Q$-HS. As a consequence we get that the induced isomorphism 
	\begin{center}
		$\alpha_1: V\otimes_{\Q}  V^{*} \overset{\cong}{\rightarrow } (V\otimes_{\Q} V)(n)$ 
\end{center}is also an isomorphism of $\Q$-HS. Moreover, it is known that the injection \begin{center}
		$\alpha_2:(H^n(\bar{Y}^{an}_{\C},\Q)\otimes_{\Q} H^n(\bar{Y}^{an}_{\C},\Q) )(n)\hookrightarrow H^{2n}(\bar{Y}^{an}_{\C}\times \bar{Y}^{an}_{\C} , \Q  )(n)$,
	\end{center}given by the K\"unneth formula is also an injective homomorphism of $\Q$-HS.\\

\textbf{Step 1: Reduction to the algebraic closure:} Let us start by fixing a basis $\beta$ of $D$ over $\Q$. By the above remarks, for each $d\in \beta$ we get a Hodge class $\phi_d:= \alpha_2\circ\alpha_1\circ \alpha_0(d)\in H^{2n}(\bar{Y}^{an}_{\C}\times \bar{Y}^{an}_{\C} , \Q  )(n)$.
	
Now, assuming the \textbf{absolute Hodge Conjecture}, from Corollary $11.3.16$ of \cite{charlesschnell} such a class $\phi_d$ has to be defined over the algebraically closed field $\bar{L}$, i.e. $\phi_d=P^{2n}_{Y\times Y} (\tilde{\phi}_d)$ where $\tilde{\phi}_d\in H^{2n}_{DR}(\bar{Y}\times_{\bar{L}} \bar{Y}/ \bar{L})(n)$.\\

	\textbf{Step 2: Reduction to finite extension:} Let us set $Z:= Y\times_L Y$. We have an $L$-vector space $H^{2n}(Z/L)$ and we also have an isomorphism \begin{equation}\label{eq:derhambase}
		H^{2n}_{DR}(Z/L)\otimes_L F\cong H^{2n}_{DR} (Z_F/F  )
	\end{equation}for every extension $F/L$. In particular \eqref{eq:derhambase} holds for $F=\bar{L}$. 

Consider $\delta:=\{\delta_1,\ldots,\delta_m  \}$ the image of an $L$-basis of $H^{2n}_{DR}(Z/L)$ in $H^{2n}_{DR}(\bar{Z}/\bar{L})$ under the above isomorphism. For $\tilde{\phi}_d$ as above we may write
	\begin{equation}
		\tilde{\phi}_d=a_1(d)\delta_1+\cdots +a_m(d) \delta_m.
	\end{equation}Given these coefficients, we set $L_d$ to be the field $L(a_1(d),\cdots,a_m(d))$, which is a finite extension of $L$. Finally, we let $\hat{L}$ be the Galois closure of the compositum of the $L_d$ for all $d\in\beta$.

We observe that for any Galois extensions $F_1, F_2$ of the field $L$ with $L\subset F_1\subset F_2$ the diagram\begin{center}
		$\begin{tikzcd}
			(H^{2n}_{DR}(Z/L)\otimes_L F_1)\otimes_{F_1} F_2  \arrow[r] \arrow[d] & H^{2n}_{DR} (Z_{F_1}/F_1)\otimes_{F_1}F_2 \arrow[d] \\
			H^{2n}_{DR} (Z/L)\otimes_{L}F_2 \arrow[r]           & H^{2n}_{DR} (Z_{F_2}/F_2 )
		\end{tikzcd}$
	\end{center}is a commutative diagram of $\gal(F_2/L)$-modules. As a consequence of this we may and do view from now on each $\tilde{\phi}_d$ as an element of $H^{2n}_{DR}(Z_{\hat{L}}/\hat{L})$.\\

	\textbf{Step 3: Back to endomorphisms:} So far we have found classes $\tilde{\phi}_d\in H^{2n}_{DR}(Z_{\hat{L}}/\hat{L})$. We want to show that these naturally correspond to endomorphisms of $H^n_{DR}(Y_{\hat{L}}/\hat{L})$ and that this correspondence behaves well with respect to the comparison isomorphism of Grothendieck.

	To that end, we start by noting that Grothendieck's comparison isomorphism between algebraic de Rham cohomology and Betti cohomology is compatible with the isomorphisms given by both Poincar\'e duality and the K\"unneth formula. We note that both of these, i.e. Poincar\'e duality and the K\"unneth formula, are defined for both cohomology theories in question. In fact for de Rham cohomology they are defined over $L$.
	
With that in mind we define $\alpha_{i,DR}$ mirroring the homomorphisms $\alpha_i$ we had earlier.
	
Therefore, for every $d\in \beta$, viewing the class $\tilde{\phi}_d$ as an element of $ H^{2n}_{DR}(Z_\C/\C)$, due to the aforementioned compatibility, we get an element $\tilde{d}\in \End_\C (H^n_{DR}(\bar{Y}_{\C}/\C))$ which is such that 
	\begin{enumerate}
		\item it maps to $d$ via the comparison isomorphism, and 
		
		\item it maps to $\tilde{\phi}_d$ via the injective map $\alpha_{2,DR}\circ \alpha_{1,DR}\circ\alpha_{0,DR}$.
	\end{enumerate}
	
	Property $(1)$ above tells us that $P^n(\tilde{d} (v))= d(P^n(v))$ for all $v\in H^n_{DR}(\bar{Y}_\C/\C)$. Thus proving the ``moreover'' part of the proposition. 
	
 Since $Y$ is defined over the field $L$ the same is true for the $\alpha_{i,DR}$. In particular since their composition $\alpha_{DR}:=\alpha_{2,DR}\circ\alpha_{1,DR}\circ \alpha_{0,DR}$ is an injective homomorphism 
	\begin{center}
		$\alpha_{DR}: \End(H^n_{DR} (Y_{\hat{L}}/ \hat{L})    )\hookrightarrow H^{2n}_{DR} (Y_{\hat{L}} \times_{\hat{L}} Y_{\hat{L}} / \hat{L}   )$,
	\end{center}we get that in fact $\tilde{d}\in \End(H^n_{DR} (Y_{\hat{L}}/ \hat{L})) $.
	
	Since $d$ was a random element in a $\Q$-basis of $D$ we get an injective homomorphism \begin{equation} 
		i:D\hookrightarrow \End(H^n_{DR} (Y_{\hat{L}}/ \hat{L}) ) \cong \End(H^n_{DR} (Y_{L}/ L)  \otimes_L \hat{L}).\end{equation}\end{proof}
	\subsubsection{Bounds on the degree extension}

Later on we want to have some control on the degree of the Galois extension $\hat{L}/L$ constructed in the proof of \Cref{propendodr}. In particular, we want an upper bound on the degree $[\hat{L}:L]$ that will be independent of the smooth projective variety $Y/L$ and the field $L$ itself. We want this bound to only depend on the dimension of $Y$ and its $n$-th Betti number. In making an analogy with the case of abelian varieties, we want upper bounds akin to those achieved in \cite{silverberg}.

\begin{prop}\label{propdegreebound}
	Assume the \textbf{absolute Hodge Conjecture} is true. Let $Y$ be a smooth $n$-dimensional projective variety defined over the number field $L$. Then the field extension $\hat{L}/L$ constructed in \Cref{propendodr} may be chosen so that for its degree we have
	\begin{center}
		$[\hat{L}:L]\leq  ( (6.31) m^2)^{m^2} $,
	\end{center}where $m=\dim_\Q  H^n(\bar{Y}^{an}_\C,\Q  )$ is the $n$-th Betti number.
\end{prop}

\begin{proof}
Let $\beta $ be a $\Q$-basis of $D$. From the proof of \Cref{propendodr} we have an injective homomorphism of $\Q$-algebras $D\hookrightarrow \End_{\hat{L}}(H^{n}_{DR} (Y_{\hat{L}}/ \hat{L}))$, given in the basis elements by $d\rightarrow \tilde{d}$ in the notation of the proof of \Cref{propendodr}. 
	
By base change we have a natural action of the finite Galois group $\gal(\hat{L}/L)$ on de Rham cohomology $H^n_{DR}(Y_{\hat{L}}/\hat{L})$, as an $L$-vector space. This induces a natural action of the same group on $\End_{\hat{L}}(H^{n}_{DR} (Y_{\hat{L}}/ \hat{L}))$, viewed as an $L$-vector space again. We start by proving the following claim.\\
	
\textbf{Claim:} The above action of the Galois group induces an action on the embedding of $D$ in $\End_{\hat{L}}(H^{n}_{DR} (Y_{\hat{L}}/ \hat{L}))$. In other words for all $\sigma \in \gal(\hat{L}/L)$ we have that $\sigma(D)=D$. \\

\begin{proof}[Proof of the claim]
Assuming the absolute Hodge Conjecture, by our earlier construction, for every element $d$ of the basis $\beta$ we get an element $\tilde{d}=i(d)\in \End_{\hat{L}} (H^n_{DR} (Y_{\hat{L}}/\hat{L}))$. By the previous proof, via Poincar\'e duality and the K\"unneth formula, we get classes $\tilde{\phi}_d\in H^{2n}_{DR} (Y_{\hat{L}}\times Y_{\hat{L}}/\hat{L})$ that map to Hodge classes $\phi_d\in H^{2n}(\bar{Y}^{an}_\C\times \bar{Y}^{an}_\C,\Q)(n)$.  As we did in our earlier proof we let $Z:=Y\times_{L} Y$. In the above construction we implicitly consider a fixed embedding $\sigma_0:\hat{L}\hookrightarrow \C$. 

By our assumption that the absolute Hodge Conjecture holds true, we get that for all embeddings $\sigma :\hat{L}\hookrightarrow \C$ the class $\tilde{\phi}_d\in  H^{2n}_{DR}(\sigma (Z_{\hat{L}})/\C)$ is Hodge. Here $\sigma (Z_{\hat{L}})$ denotes the complex variety obtained from $Z_{\hat{L}}$ when we base change via the embedding $\sigma$ to $\C$.

From the embedding $\sigma_0:\hat{L}\hookrightarrow \C$ that we fixed earlier we get an embedding $i_0:L\hookrightarrow \C$. Any embedding $\sigma:\hat{L}\hookrightarrow\C$ that is such that $\sigma|_L=i_0$ will correspond to an element of the Galois group $\gal(\hat{L}/L)$ via the bijective map $\gal(\hat{L}/L)\rightarrow \{\sigma:\hat{L}\hookrightarrow \C : \sigma|_L=i_0\} $ given by $\tau\mapsto \sigma_0\circ\tau$. For notational brevity we suppress $\sigma_0$ from our notation from now on and identify $\tau\in\gal(\hat{L}/L)$ with $\sigma_0\circ\tau$, in other words we identify the elements of $\gal(\hat{L}/L)$ with the corresponding embedding $\hat{L}\hookrightarrow \C$. With this notational convention we may and will write from now on $Y_\C$, or $Z_\C$ respectively, for the complex variety we would otherwise denote by $\sigma_0Y_{\hat{L}}$, or $\sigma_0Z_{\hat{L}}$ respectively.

For the above $\sigma$, since $Y$ and hence also $Z$ are defined over the field $L$, by the above remarks $H^{2n}_{DR}(\sigma Z_{\hat{L}}/\C)$ may be identified with $H^{2n}_{DR}(Z_{\C}/\C)$. Via this identification $\tilde{\phi}_d$ will get mapped to $\sigma^{*}(\tilde{\phi}_d)\in H^{2n}_{DR}(Z_{\hat{L}}/\hat{L})$. Here $\sigma^{*}: H^{2n}_{DR}(Z_{\hat{L}}/\hat{L})\rightarrow H^{2n}_{DR}(Z_{\hat{L}}/\hat{L})$ denotes the isomorphism of $L$-vector spaces induced by $\sigma \in\gal(\hat{L}/L)$ on cohomology.

Now, since $Y$ and $Z$ are both defined over the field $L$, both the Poincar\'e duality isomorphism and the K\"unneth formula on de Rham cohomology are defined over the field $L$ as well. These maps, by construction, commute with the isomorphisms $\sigma^{*}$ so we get that $\sigma^{*}(\tilde{d})$ maps to $\sigma^{*}(\tilde{\phi}_d)\in H^{2n}_{DR}(Z_\C/\C)$ via the map $\alpha_{DR}$ we had in the proof of \Cref{propendodr}.

Writing $P$ for Grothendieck's comparison isomorphism we have that $P(\sigma^{*}(\tilde{d}))\in D\subset \End_{\Q} H^n(\bar{Y}^{an}_\C,\Q)$ is a Hodge endomorphism. Thus $\sigma^{*}(\tilde{d})\in i(D)$ with the notation of \Cref{propendodr} and the result follows.
\end{proof}

By the claim therefore we get an action of $G:=\gal(\hat{L}/L)$ on the $\Q$-vector space $D$, or more precisely its image in $\End_{\hat{L}}(H^{n}_{DR} (Y_{\hat{L}}/ \hat{L}))$. Let $\dim_{\Q}D= m_0$ and note that $m_0\leq m^2$ trivially. We may and do assume, without loss of generality, that the field extension $\hat{L}/L$ constructed in the previous proof is minimal with the property that every cycle of the above basis $\tilde{d}$ is defined over $\hat{L}$. This implies that the corresponding group homomorphism 
$\gal(\hat{L}/L)\rightarrow \aut(D)$ is in fact injective.

Let $\Lambda_1$ be a lattice in $D$, and consider $\Lambda := \Sum{g\in G}{} g(\Lambda_0)$. This will be a lattice that is also invariant by $G$. From the $G$-invariance of $\Lambda$ we get a group homomorphism\begin{center}
	$ G\rightarrow \GL(\Lambda )$.
\end{center}This homomorphism will be injective as well by our earlier assumption about the minimality of the extension $\hat{L}/L$.

Let $N\geq 3$. Then, we know\footnote{This is originally a result of Minkowski \cite{minkowski}. For a modern proof see \cite{guralorenz}, section $5.2$ and Lemma $9$ in particular.} that the kernel of the surjective map $\GL(\Lambda)\rightarrow \GL(\Lambda/N\Lambda)$ contains no element of finite order of the group $\GL(\Lambda)$. As a result we get $G\hookrightarrow \GL(\Lambda/N\Lambda)$ which implies that $|G|$ divides $|\GL(\Lambda/N\Lambda)|=|\GL_{m_0}(\Z/N\Z)|$.

Following the notation of \cite{silverberg} we let $g_r(N):=|\GL_{r}(\Z/N\Z)|$ and $G(r):=\gcd \{ g_r(N):N\geq 3  \}$. From Theorem $3.1$ of \cite{silverberg} we have that \begin{equation}\label{eq:silverbound}
	G(r)< ((6.31)r)^r.
\end{equation}

From the above argument we get that $|G|$ divides $G(m_0)$ and combining this with \eqref{eq:silverbound} and the fact that $m_0\leq m^2$ we get that \begin{equation}
	|G|< ( (6.31) m^2)^{m^2}.
\end{equation}
\end{proof}

\part{Determining the Trivial relations}
	
	Given a G-admissible variation of Hodge structures we will show that for some exceptional points $s\in S(\bar{\Q})$ we get so called ``non-trivial'' relations among the values of the relative periods at the point $s$. To be able to say that these relations we will create are in fact non-trivial we need to know what the trivial ones are first! 

We have devoted this part of the paper to determining these trivial relations in the case where the generic special Mumford-Tate group of our variation is a symplectic group.

\section{The action of the Local Monodromy}\label{section:monodromy}

We start by reviewing a key property of the local monodromy that we will need during this process. This follows the ideas in Chapter X, Lemma $2.3$ of \cite{andre1989g}.

We note that throughout this chapter we use the setup introduced in \Cref{section:notations}.\\

Let us a fix throughout this section a G-admissible variation of $\Q$-HS defined by the pair $f':X'\rightarrow S'$, $s_0\in S'(K)$, as in \Cref{section:notations}. Let $\Delta$ be a small disc embedded in $S'^{an}_\C$ centered at $s_0$ and such that $\Delta^{*}\subset S^{an}_\C$ . We have already remarked in \Cref{section:periodsandgfunctions} that the logarithm of the local monodromy of $\Delta^{*}\subset S^{an}_\C$ acting on $R_n(f^{an}_{\C})_{*}(\Q_{X^{an}_{\C}})|_{\Delta^{*}}$ defines the local subsystem $\mathcal{M}_0:=M_0 R_n(f^{an}_{\C})_{*}(\Q_{X^{an}_{\C}})|_{\Delta^{*}}$. This is contained in the maximal constant subsystem of $R_n(f^{an}_{\C})_{*}(\Q_{X^{an}_{\C}})|_{\Delta^{*}}$, since $2\pi iN^{*}$, the nilpotent logarithm associated with the action of monodromy on the limit Hodge structure, has degree of nilpotency $\leq n+1$.

We recall that, since the map $f:X\rightarrow S$ is smooth and projective, we have a bilinear form $\langle,\rangle$ on the local system $R_n(f^{an}_{\C})_{*}(\Q_{X^{an}_{\C}})$ induced by the polarizing form.

\begin{lemma}\label{maxisotropic}
	The local system $\mathcal{M}_0$ is a totally isotropic subsystem of the local system $R_n(f^{an}_{\C})_{*}(\Q)|_{\Delta^{*}}$ with respect to the polarizing form $\langle,\rangle$.
\end{lemma}

\begin{proof}The skew-symmetric form $\langle, \rangle $ defines a morphism of local systems \begin{center}
		$R_n(f^{an}_{\C})_{*}(\Q_{X^{an}_{\C}})|_{\Delta^{*}}\otimes R_n(f^{an}_{\C})_{*}(\Q_{X^{an}_{\C}})|_{\Delta^{*}}\rightarrow \Q(n)|_{\Delta^{*}}$.
	\end{center}Therefore it is invariant under the local monodromy and we conclude that for any $z\in\Delta^{*}$ and for all $v,w\in (R_n(f^{an}_{\C})_{*}(\Q_{X^{an}_{\C}}))_z$ we have \begin{equation}\label{eq:localmonodromypolar}
		\langle N^{*}_z v,w\rangle +\langle v,N^{*}_z w\rangle =0.
	\end{equation}
	
	Now let $v,w$ be any two sections of $\mathcal{M}_0$. Then for any $z\in \Delta^{*}$ there exist $v_{0,z},w_{0,z}\in (R_n(f^{an}_{\C})_{*}(\Q_{X^{an}_{\C}}))_z$ such that $v_z=(2\pi iN_z^{*})^n(v_{0,z})$ and $w_z=(2\pi iN_z^{*})^n(w_{0,z})$. Using \eqref{eq:localmonodromypolar} we thus get \begin{center}
		$\langle v_z,w_z\rangle= \langle  (2\pi iN_z^{*})^n(v_{0,z}),(2\pi iN_z^{*})^n(w_{0,z})\rangle =$
		
		$=-\langle (2\pi iN_z^{*})^{n-1}(v_{0,z}),(2\pi iN_z^{*})^{n+1}(w_{0,z})\rangle =0$,\end{center}where the last equality follows from the fact that $N_z^{*}$ has degree of nilpotency $\leq n+1$.
	
	Therefore we get that for all $v,w\in \mathcal{M}_0$ we have $\langle v,w\rangle =0$. Hence $\mathcal{M}_0$ is a totally isotropic local subsystem.\end{proof}
	
	\section{Trivial relations}\label{section:trivialrelations}
\subsection{Our setting and notations}\label{section:subsectionnotationsnontrivial}

Let $f:X\rightarrow S$ be a smooth projective morphism of $k$-varieties where $k$ is a subfield of $\bar{\Q}$. We also fix an embedding $\bar{\Q}\hookrightarrow \C$ so that we may consider $k$ as a subfield of $\C$. Assume that $S$ is a smooth irreducible curve, that the fibers of $f$ are $n$-dimensional, and let $\mu:=\dim_{\Q} H^n(X^{an}_s,\Q)$ for some $s\in S(\C)$. Throughout this section we assume that $n$ is \textbf{odd} and that $S=S'\backslash\{s_0\}$ for some point $s_0\in S'(k)$, where $S'$ is a curve as in \Cref{section:notations}.

We consider\begin{center}
	$P^n_{X/S}:H^n_{DR}(X/S)\otimes_{\mathcal{O}_S}\mathcal{O}_{S^{an}_{\C}}\rightarrow R^n(f^{an}_{\C})_{*}(\Q_{X^{an}_{\C}}) \otimes_{\Q_{S^{an}_{\C} }} \mathcal{O}_{S^{an}_{\C}}$,
\end{center} the relative period isomorphism.

\subsubsection*{The Riemann relations}

Let $\omega_i$, $1\leq i\leq \mu$, be a basis of $H^n_{DR}(X_{\eta})$ over $k(S)$, where $\eta$ is the generic point of $S$. Then there exists some dense affine open subset $U$ of $S$ over which these $\omega_i$ are sections of the sheaf $H^n_{DR}(X/S)$. We also fix a trivialization $\gamma_i$ of $R_n(f^{an}_{\C})_{*}(\Q_{X^{an}_{\C}})$, i.e. the relative homology, over an analytic open subset $V$ of $U^{an}_{\C}$. Since we are interested in describing the relations among the  periods archimedeanly close to the point of degeneration $s_0$, we may and do assume that the set $V$ is simply connected and contained in a fixed  small punctured disk $\Delta^{*}$ around $s_0$.

The matrix of $P^n_{X/S}$ with respect to this basis and trivialization will have entries in the ring $\mathcal{O}_V$. We multiply the matrix's elements by $(2\pi i)^{-n}$ and, by abuse of notation, we denote the above $\mu\times \mu$ matrix of relative $n$-periods by
\begin{center}
	$P_{X/S}:=((2\pi i)^{-n} \int_{\gamma_j}^{}   {\omega_i})$.
\end{center}

Since the morphism $f:X\rightarrow S$ is smooth, projective, and is also defined over $k$, it defines a polarization which will be defined over $k$ as a form on de Rham cohomology. In particular we get, since the weight $n$ of our variation is odd,\begin{itemize}
	\item a skew-symmetric form $\langle,\rangle_{DR}$ on $H^n_{DR}(X_\eta)$ with values in $k(S)$ and 
	
	\item a skew-symmetric form $\langle,\rangle_B=(2\pi i)^{n} \langle,\rangle$ on $R_n(f^{an}_{\C})_{*}(\Q_{X^{an}_{\C}})$ with values in $\Q(n)$.
\end{itemize}

These two skew-symmetric forms are compatible with the isomorphism $P^n_{X/S}$, in the sense that the dual form of $\langle , \rangle_B$ coincides with the form induced by $\langle,\rangle_{DR}$ via the isomorphism $P^n_{X/S}$. 
The compatibility of the polarizing forms translates to relations among the periods. These relations can be described succinctly by the equality\begin{equation}\label{eq:riemanrelationsnontrivialitychapter}
	\prescript{t}{}{P} M_{DR}^{-1} P= (2\pi i)^{-n}M_B^{-1},
\end{equation}where $M_{DR}$ and $M_B$ are the matrices of $\langle,\rangle_{DR}$ and the dual of $\langle,\rangle_{B}$ respectively with respect to some basis and trivialization.

For more on this see \Cref{appendixpolarizations}. The relations given on the periods by \eqref{eq:riemanrelationsnontrivialitychapter} are practically a direct consequence of the well known Hodge-Riemann bilinear relations defining a polarization of a Hodge structure. For this reason from now on we shall refer to \eqref{eq:riemanrelationsnontrivialitychapter} as the \textbf{Riemann relations} for brevity.	
	\subsubsection*{Enter: The G-functions}

From now on let $f:X\rightarrow S$ be one of the morphisms $F_{t}:\mathcal{X}'_{t}\rightarrow C_{t}$ coming from a good cover of a morphism $f':X'\rightarrow S'$, as in the discussion in \Cref{section:goodcovers}, underlying a fixed G-admissible variation of $\Q$-HS. In particular, we may now assume that our curve $S'$ is equipped with a local uniformizer $x$ at $s_0$ that vanishes only at $s_0$.

With this in mind, we may and do select the above basis $\omega_i$ and trivialization $\gamma_j$ so that the following are satisfied:\begin{itemize}
	\item the $\omega_i$ are a symplectic basis of $H^n_{DR}(X_{\eta})$ so that $\omega_1,\ldots, \omega_{\mu/2}$ constitute a basis of the maximal isotropic subspace $F^{\frac{n+1}{2}}H^n_{DR}(X_\eta)$ and the rest of the elements, i.e. $\omega_{\mu/2+1},\ldots, \omega_\mu$ are the basis of a transverse Lagrangian of $F^{\frac{n+1}{2}} H^n_{DR}(X_\eta)$, and 
	
	\item the $\gamma_j$ is a symplectic trivialization of $R_n(f^{an}_{\C})_{*}(\Q_{X^{an}_{\C}})|_V$, which is also such that $\gamma_1,\ldots,\gamma_h$ are a frame of the space $\mathcal{M}_0|_{V}  $ and the $\gamma_1,\ldots, \gamma_{\mu/2}$ are a frame of a maximal totally isotropic subsystem that contains $M_0R_n(f^{an}_{\C})_{*}(\Q_{X^{an}_{\C}})|_V$.
\end{itemize}

With these choices we may and do assume from now on that the matrices that correspond to the two aforementioned forms are
$M_{DR}=M_B=J_\mu=\begin{pmatrix}
	0&-I\\
	I&0
\end{pmatrix}$.
With this \eqref{eq:riemanrelationsnontrivialitychapter} translates to \begin{equation}\label{eq:riemannrelationodd}
	\prescript{t}{}{P} J_\mu P= (2\pi i)^{-n}J_\mu.
\end{equation}

\subsubsection*{The main result}

Let $y_{i,j}$ with $1\leq i\leq \mu$ and $1\leq j\leq h$ be the entries of the first $h$ columns of the matrix $P_{X/S}$. The aforementioned work of Andr\'e, see \Cref{existence}, guarantees that these are G-functions.
Note that this is all happening with respect to the rational function $x\in K(S')$ of $S'$ which only vanishes at $s_0$, with respect to which the $y_{i,j}$ can be written as power series. See also the proof of \Cref{changeofplace} and the remarks that follow it.

For the remainder of this section we consider the above notation fixed. The rest of this section is dedicated to describing the generic,  or ``trivial'', relations among the G-functions $y_{i,j}$. Indeed, we prove the following:

\begin{prop}\label{goalnontriviality} With the above notation, assume that the generic special Mumford-Tate group of the variation of $\Q$-HS on $S^{an}_{\C}$ given by $R^n(f^{an}_{\C})_{*}(\Q_{X^{an}_{\C}})$ is $Sp(\mu,\Q)$. 
	
	Then, the Zariski closure of the $\mu\times h$ matrix $Y:=(y_{i,j})$ over  $\bar{\Q}[x]$ in $\mathbb{A}^{\mu\times h}$ is the variety whose ideal is given by the Riemann relations.
\end{prop}
	
	\subsection{Trivial relations over $\C$ for the period matrix}

Under the notations and assumptions of \Cref{section:subsectionnotationsnontrivial} and \Cref{goalnontriviality} we have the following:

\begin{lemma}\label{monoatgeneric}
	Let $z\in V\subset U^{an}$ be a Hodge generic point of the $\Q$-VHS given by $R^n(f^{an}_{\C})_{*}(\Q_{X^{an}_{\C}})$. Then the monodromy group $H_z$ at $z$ is $Sp(\mu,\Q)$.
\end{lemma}
\begin{proof} Let $\rho_H:\pi_1(S^{an},z)\rightarrow GL(H^n(X^{an}_z,\Q))$ be the monodromy representation at $z$. Then, by Andr\'e's Theorem of the fixed part\cite{andrefixed} we know that $H_z$, which is the connected component of the the $\Q$-algebraic group $\rho_H(\pi_1(S^{an},z))^{\Q-Zar}$, is a normal subgroup of the derived subgroup of the Mumford-Tate group $G_{mt,z}$ at $z$. In other words \begin{center}
		$H_z\trianglelefteq DG_{mt,z}$.
	\end{center}
	
	On the other hand we have that $DG_{mt,z}\leq G_{smt,z}$ and trivially that $DG_{smt,z}\leq DG_{mt,z}$, where $G_{smt,z}$ is the special Mumford-Tate group at $z$. But, by assumption, we know that $G_{smt,z}\simeq Sp(\mu,\Q)$, since $z$ is Hodge generic for our variation. It is classical that $Sp(\mu,\Q)$ satisfies $DSp(\mu,\Q)=Sp(\mu,\Q)$. Hence we have $DG_{mt,z}= Sp(\mu,\Q)$.
	
	We thus get that $H_z\trianglelefteq Sp(\mu,\Q)$. Finally, $Sp(\mu,\Q)$ is a simple $\Q$-algebraic group, therefore $H_z=1$ or $H_z =Sp(\mu,\Q)$. But, if we had $H_z=1$, then the variation of $\Q$-HS in question would be isotrivial\footnote{By an \textbf{isotrivial} variation of $\Q$-HS, we mean one for the monodromy acts by a finite group. See \cite{ggkpaper} for more on this.}, and hence extend to ${S'}^{an}=S^{an}\cup \{s_0\}$. We get a contradiction since the local monodromy at $s_0\in S'(\C)$ is non-trivial by assumption, and hence the above VHS does not extend to ${S'}^{an}$.\end{proof}

From now on, by taking a finite \'etale cover of $S$ if necessary, we may and do assume that $\rho_H(\pi_1(S^{an},z))^{\Q-Zar}$ is connected, i.e. that for the Hodge generic points $z\in V$ we have $H_z=\rho_H(\pi_1(S^{an},z))^{\Q-Zar}$.

\subsubsection{The matrix of Periods and differential equations}

Let us denote by $M_{\mu}$ the variety of $\mu\times \mu$ matrices over $\C$, where $\mu:=\dim_{\Q}H^n(X^{an}_{s,\C},\Q)$ for any $s\in S(\C)$.

The period matrix $P_{X/S}$ defines a holomorphic map\begin{center}
	$\phi :V\rightarrow M_{\mu}$.
\end{center}We let $Z\subset V\times M_{\mu}$ be the graph of this function. The first step in our process is determining the $\C$-Zariski closure of $Z$.

\begin{lemma}\label{czarclosure}Let $Z$ be as above. The $\C$-Zariski closure of $Z$ is\begin{center}$S_{\C}\times \{M: \prescript{t}{}{M}J_{\mu} M=(2\pi i)^{-n} J_\mu\}$.	
	\end{center}
	
\end{lemma}

In order to prove this we will employ the monodromy action in an essential way. For this purpose we will need to review some further properties of the isomorphism $P^n_{X/S}$. 

To this end, let us consider\begin{center}
	$Q^n_{X/S} : R^n(f^{an}_{\C})_{*}(\Q_{X^{an}_{\C}}) \otimes_{\Q_{S^{an}}} \mathcal{O}_{S^{an}} \overset{\sim}{\rightarrow}H^n_{DR} (X/S)\otimes_{\mathcal{O}_S}\mathcal{O}_{S^{an}}$,
\end{center}the inverse of $P^n_{X/S}$.

It is known, see \cite{katzalgsol} Prop.$4.1.2$, that this isomorphism restricts to an isomorphism of local systems\begin{center}
	$Q:R^n(f^{an}_{\C})_{*}(\C_{X^{an}_{\C}})\overset{\sim}{\rightarrow} \R^n(f^{an}_{\C})_{*}\Omega^{\bullet}_{X^{an}_\C/\C}\overset{\sim}{\rightarrow} (H^n_{DR}(X/S) \otimes_{\mathcal{O}_S}\mathcal{O}_{S^{an}})^{\nabla}$   
\end{center}where $(H^n_{DR}(X/S) \otimes_{\mathcal{O}_S}\mathcal{O}_{S^{an}})^{\nabla}\subset \R^{n}(f^{an}_{\C})_{*} \Omega^{\bullet}_{X^{an}/S^{an}}$ is the local system of horizontal sections with respect to the Gauss-Manin connection.

Note that we have an inclusion of local systems $R^n(f^{an}_{\C})_{*}(\Q_{X^{an}_{\C}})\hookrightarrow R^n(f^{an}_{\C})_{*}(\C_{X^{an}_{\C}})$ on $S^{an}$. This leads to a commutative diagram \begin{center}
	\begin{tikzcd}
		\pi_1(S^{an},z)\arrow[r, "\rho_{H,\C}"] \arrow[d, "\rho_H"'] &  \GL(H^n(X^{an}_z,\C))\\
		\GL(H^n(X^{an}_z,\Q)) \arrow[ru, hook]               &  
\end{tikzcd}\end{center}
for any point $z\in S^{an}$. 

In particular, we get, under our assumptions on the connectedness of the group $(\rho_H(\pi_1(S^{an},z)))^{\Q-Zar}$, that the group \begin{center}
	$G_{mono,z}:=(\rho_{H,\C}(\pi_1(S^{an},z)))^{\C-Zar}$,
\end{center} i.e. the $\C$-Zariski closure of the image of the fundamental group under $\rho_{H,\C}$, is such that\begin{equation}\label{eq:monocomp1}
	G_{mono,z}= H_z\otimes_{\Q}\C.
\end{equation}

Earlier we saw that we have an isomorphism $Q$ of local systems over $S^{an}$.
By the equivalence of categories between local systems over $S^{an}$ and representations of the fundamental group $\pi_1(S^{an},z)$ we thus have that the representations \begin{center}
	$\rho_{H,\C}:\pi_1(S^{an},z) \rightarrow \GL(H^n(X^{an}_z,\C))$, and 
	
	$\rho_{DR}:\pi_1(S^{an},z) \rightarrow  \GL((H^n_{DR} (X/S)\otimes_{\mathcal{O}_S} \mathcal{O}_{S^{an}})^{\nabla})$,
\end{center}are conjugate. In fact, keeping in mind that all actions are on the right, we have that $\rho_{DR}(\lambda)= Q(z)^{-1} \rho_{H,\C} (\lambda) Q(z)$, for all $\lambda\in \pi_1(S^{an},z)$, where $Q(z)$ is the fiber of $Q$ at $z$. From this we get that
\begin{equation}\label{eq:conjugatemono}
	G_{DR,z}:=(\rho_{DR}(\pi_1(S^{an},z)))^{\C-Zar}=Q(z)^{-1} G_{mono,z}Q(z).
\end{equation}

Let $B$ be the matrix of the isomorphism $Q|_{V}$ with respect to the frame $\{\gamma_j^{*}:1\leq j\leq \mu\}$ of the trivialization of $R^n(f^{an}_{\C})_{*}(\Q_{X^{an}_{\C}})|_V$, i.e. the dual of the frame given by the $\gamma_j$ on $R_n(f^{an}_{\C})_{*}(\Q_{X^{an}_{\C}})|_V$, and the basis $\{\omega_i:1\leq i\leq \mu\}$ chosen above. We then have that the rows $b_i$ of $B$, which will correspond to $Q|_V(\gamma_i^{*})$ written in the basis $\omega_i$, will constitute a basis of the space $\Gamma(V, (H^n_{DR}(X/S)\otimes_{\mathcal{O}_S}\mathcal{O}_{S^{an}})^{\nabla})$. In other words $B$ is a complete solution of the differential equation $\nabla(\omega)=0$, defined by the Gauss-Manin connection. We note that in our setting the Gauss-Manin connection is known to be defined over the field $k$ by work of Katz and Oda. see \cite{katzoda} and \cite{katznilpo}.

Let $\Gamma\in M_\mu(k(S))$ be the (local) matrix of $\nabla$ on $U$ with respect to the basis given by the $\omega_i$. Writing $\nabla(\omega)=d\omega+\omega\Gamma$, identifying the $\omega$ with the $1\times n$ matrix given by the coefficients of $\omega$ in the basis given by the $\omega_i$, we may rewrite the above equation as $d\omega=-\omega\Gamma$. The corresponding matricial differential equation then becomes\begin{equation}\label{eq:matricialde}
	X'=-X\Gamma.
\end{equation}

The monodromy representation $\rho_{DR}$ defines analytic continuations of solutions at $z$ of the differential equation \eqref{eq:matricialde}. So in considering the value at the point $z$ of the analytic continuation $B^{\lambda}$ of the matrix $B$ along the cycle $\lambda\in\pi_1(S^{an},z)$, corresponding to a loop $\gamma$ passing through $z$, all we are doing is multiplying the matrix $B_z$ by $\rho_{DR}(\lambda)$. In other words for $\lambda\in \pi_1(S^{an},z)$ we have that\begin{equation}\label{eq:derhamaction}
	(B^{\lambda})_z= B_z\rho_{DR}(\lambda).\end{equation}

We apply the ideas presented in the above discussion to prove the following lemma.
\begin{lemma}\label{lemmancont}
	Consider $A$ to be the matrix of the isomorphism $P^n_{X/S}$ on the open analytic set $V$ with respect to the basis $\omega_i$ and frame $\gamma_j^{*}$ chosen above. Let $z\in V$ and let $\lambda \in \pi_1(S^{an},z)$. Then the value at $z$ of the analytic continuation $A^{\lambda}$ of $A$ along the loop that corresponds to $\lambda$ is given by \begin{center}
		$(A^{\lambda})_z =A_z \rho_{H,\C}(\lambda)^{-1} $,
	\end{center}where $\rho_{H,\C}$ is the above representation on Betti cohomology.\end{lemma}
\begin{proof}
	We have $A\cdot B=I_\mu$ hence $A^{\lambda}\cdot B^{\lambda}=I_\mu$. Using \eqref{eq:derhamaction} we get that $(A^{\lambda})_z=\rho_{DR}(\lambda)^{-1} B_z^{-1} = \rho_{DR}(\lambda)^{-1}A_z$. 
	
	On the other hand, with the above notation we have that $\rho_{DR}(\lambda) =B_z^{-1} \rho_{H,\C}(\lambda)B_z$. This combined with the above leads to the result.\end{proof}

\begin{remark}
	The same relation holds for the value $P_{X/S}(z)$ at $z$ of the matrix of relative periods $P_{X/S}$, since $P_{X/S}=(2\pi i)^{-n} A$. 
\end{remark}

We are now in the position to prove \Cref{czarclosure}.
\begin{proof}[Proof of \Cref{czarclosure}]
	Let $Z\subset V\times M_\mu\subset S_{\C} \times M_\mu$ be the graph of the isomorphism $P_{X/S}|_V$. Let $\tilde{Z}$ be the union of the graphs of all possible analytic continuations of $Z$. It is easy to see via analytic continuation that we have $(\tilde{Z})^{\C-Zar} =Z^{\C-Zar}$. We also note that for all $z\in V$ we have $(\tilde{Z}_z)^{\C-Zar}\subset (\tilde{Z}^{\C-Zar})_z$ for trivial reasons.
	
	We focus on the points $z\in V$ that are Hodge generic for the variation of $\Q$-HS given by $R^n(f^{an}_{\C})_{*}(\Q_{X^{an}_{\C}})|_V$. We note that the set of such $z$ in $V$, which we denote by $V_{Hgen}$, is uncountable. 
	
	By the remark following \Cref{lemmancont} we know that \begin{center}
		$\tilde{Z}_z=P_{X/S}(z) \rho_{H,\C} (\pi_1(S^{an},z)) $.
	\end{center} From this we get that $(\tilde{Z}_z)^{\C-Zar}= P_{X/S}(z)G_{mono,z} $.
	
	From \eqref{eq:monocomp1} we know that $G_{mono,z}=H_z\otimes_{\Q}\C$ while from \Cref{monoatgeneric} we know that, since $z\in V_{Hgen}$, we have $H_z\simeq  Sp(\mu,\Q)$, hence $G_{mono,z}\simeq Sp(\mu,\C)$. Hence we have $(\tilde{Z}_z)^{\C-Zar}= P_{X/S}(z)Sp(\mu, \C)$.
	
	Using \eqref{eq:riemannrelationodd} together with the above we arrive through elementary reasoning to\begin{equation}\label{eq:tildefiber}
		(\tilde{Z}_z)^{\C-Zar} = \{ M\in \GL_\mu(\C) : \prescript{t}{}{M} J_\mu M =(2\pi i)^{-n}  J_\mu\}.
		\end{equation}Applying this to the fact that for all $z\in V$ we have $(\tilde{Z}_z)^{\C-Zar} \subset (\tilde{Z}^{\C-Zar})_z=(Z^{\C-Zar})_z$, we get that \begin{equation}\label{eq:almostthere}
		V_{Hgen} \times  \{ M\in \GL_\mu(\C) : \prescript{t}{}{M} J_\mu M =(2\pi i)^{-n}  J_\mu \} \subset  Z^{\C-Zar}.
	\end{equation}Now, using the fact that $V_{Hgen}$ is uncountable and taking Zariski closures in \eqref{eq:almostthere} we get that \begin{center}
		$S_{\C}\times  \{ M\in \GL_\mu(\C) : \prescript{t}{}{M}J_\mu M =(2\pi i)^{-n} J_\mu \} \subset Z^{\C-Zar}$.
	\end{center}
	
	On the other hand, once again from \eqref{eq:riemannrelationodd}, we know that \begin{center}
		$Z\subset S_{\C}\times  \{ M\in \GL_\mu(\C) : \prescript{t}{}{M} J_\mu M =(2\pi i)^{-n} J_\mu  \}$
	\end{center} which, by once again taking Zariski closures, gives the reverse inclusion.	
\end{proof}	
	\subsection{Trivial relations over $\C$ for the G-functions}

As we remarked earlier, the entries of the first $h$ columns of our matrix $P_{X/S}$ are G-functions, under our choice of basis and trivialization. Let us denote by $y_{i,j}$ these entries and by $Y$ the respective $\mu\times h$ matrix they define. Consider the projection map $\pr :M_\mu\rightarrow \mathbb{A}^{\mu \times h}$ that maps a matrix $(a_{i,j})\in M_{\mu}$ to the $\mu\times h$ matrix that consists of its first $h$ columns. This maps $P_{X/S}$ to $Y$. 

\begin{lemma}\label{gfunczar}Let $T$ be the subvariety of $\mathbb{A}^{\mu\times h}$ defined by the following set of polynomials \begin{center}
		$\{ \prescript{t}{	}{b_i} J_\mu b_j: 1\leq i,j \leq h\}$,
	\end{center}where $b_i$ denotes the $i$-th column of a matrix of indeterminates.
	
	Then $Y^{\C(S)-Zar}=T_{\C(S)}$.	
\end{lemma}
\begin{proof}
	Let $Z_Y\subset V\times M_{\mu\times h}(\C)$ denote the graph of $Y$ as a function $Y:V\rightarrow M_{\mu\times h}(\C)$. It suffices to show that $Z_Y^{\C-Zar} =S\times T$.

	The inclusion $Z_Y^{\C-Zar} \subset S\times T$ follows trivially from \eqref{eq:riemannrelationodd}, which shows that $Z_Y\subset V\times T(\C)$. On the other hand, we consider the map $\id_S\times \pr: S\times M_\mu \rightarrow S\times \mathbb{A}^{\mu\times h}_{\C}$. Note that $Z_Y^{\C-Zar} = (\id_S\times\pr ) (Z^{\C-Zar})$.
	
	By construction we have that the columns $c_i$ of any $\mu\times h$ matrix $C\in T(\C)$ will be a basis that spans an isotropic subspace of dimension $h$ with respect to the symplectic form defined by $J_{\mu}$ on $\C^{\mu}$. It is easy to see that we can extend this set of vectors to a basis $\{c_j:1\leq j\leq \mu \}$ of $\C^\mu$ that satisfies\begin{enumerate}
		\item $\prescript{t}{}{c_i} J_\mu c_j=0$ for all $i,j$ with $|i-j|\neq \mu/2$, and 
		
		\item $\prescript{t}{}{c_i} J_\mu c_j=(2\pi i)^{-n}$ for $i=j+\mu/2$.
	\end{enumerate}In other words, this is a symplectic basis ``twisted'' by a factor $(2\pi i)^{-n/2}$. The $\mu\times \mu$ matrix $M_C$ with columns $c_i$ will then be such that $(s,M_C)\in Z^{\C-Zar}$ by \Cref{czarclosure} and by construction $\pr(M_C)=C$. 
	
	Combining the above with the fact that $Z_Y^{\C-Zar} = (\id_S\times\pr ) (Z^{\C-Zar})$ we have that $S\times T\subset Z_Y^{\C-Zar}$ and our result follows.
\end{proof}
	\subsection{Trivial relations over $\bar{\Q}$}

So far we have not used any arithmetic information about the $y_{i,j}$, namely the fact that they are G-functions.

Let $\xi\in \bar{\Q}$. Then a trivial polynomial relation with coefficients in $\bar{\Q}$ among the values of the $y_{i,j}\in \bar{\Q}[[x]]$ at the point $\xi$ is a relation that satisfies the following:\begin{enumerate}
	\item there exists homogeneous $p(x_{i,j})\in \bar{\Q}[x_{i,j}]$ such that the relation we have is of the form $p(y_{i,j}(\xi))=0$,  
	\item the relation holds $v$-adically for some place $v$ of $\bar{\Q}$, i.e. letting $i_v:\bar{\Q}\hookrightarrow \bar{\Q}_v$ we have that the $y_{i,j}$ converge at $i_v(\xi)$ and the above relation is an equality in $\bar{\Q}_v$,
	\item there exists a polynomial $q(x)(x_{i,j})\in \bar{\Q}[x][x_{i,j}:1\leq i\leq \mu,1\leq j\leq h]$ such that it has the same degree as $p$, with respect to the $x_{i,j}$, and $q(\xi)(x_{i,j})=p(x_{i,j})$ and $q(x)(y_{i,j})=0$.
\end{enumerate}

Therefore, to describe the trivial relations among the values of our G-functions $y_{i,j}$ at some $\xi\in\bar{\Q}$, it is enough to determine the $\bar{\Q}[x]$-Zariski closure of the matrix $Y$. We do this in the following lemma, which is practically a more detailed rephrasing of \Cref{goalnontriviality}.

\begin{lemma}\label{trivialrelationsfinal} Let $Y$ be the $\mu\times h$ we had above. Then the ${\bar{\Q}[x]}$-Zariski closure $Y^{\bar{\Q}[x]-Zar}$ of $Y$ is the subvariety of $\mathbb{A}^{\mu\times h}_{\bar{\Q}[x]}$ defined by the following set of polynomials\begin{center}
		$\{ \prescript{t}{	}{b_i} J_\mu b_j : 1\leq i,j \leq h\}$,
	\end{center}where $b_i$ denotes the $i$-th column of a matrix of indeterminates.
\end{lemma}
\begin{proof}We let $\Sigma$ be the set of polynomials above and let $I_R$ be the ideal generated by $\Sigma$ in the ring $R[x_{i,j}]$, where $R$ will denote different fields in our proof.
	
	In this case from \Cref{gfunczar} we know that $Y^{\C(S)-Zar}$ is equal to $V(I_{\C(S)})$. Note that the elements of $\Sigma$ all have coefficients in $\bar{\Q}[x]$, in fact they have coefficients in $\bar{\Q}$. From this we get the result we wanted, i.e. $Y^{\bar{\Q}[x]-Zar} =V(I_{\bar{\Q}[x]})$. \end{proof}

\begin{remark}
	Implicit in the previous proof is the fact that we have a polarization that is defined over $k\subset \bar{\Q}$ as a cycle in some de Rham cohomology group.\end{remark}
	\subsubsection{Application to the case of interest}

As we saw in \Cref{section:goodcovers}, to a fixed G-admissible variation of $\Q$-HS we can associate a family $\mathcal{G}$ of G-functions, which can be furthermore written, as in \eqref{eq:familyofgfuns}, as the disjoint union of the families of G-functions $\mathcal{G}^{(t)}$, $1\leq t\leq t_0:=|x^{-1}(s)|$, one for each of the roots of the morphism $x:C'_{S}\rightarrow \mathbb{P}^1$ of the good cover of our curve $S$. Let us write $y^{(t)}_{i,j}(x)$ from now on for these G-functions. 

If we wanted to describe the trivial relations of the whole family $\mathcal{G}$ we would have to consider the $\bar{\Q}[x]$-Zariski closure of $Y^{(1)}\times \ldots Y^{(t_0)}$, where $Y^{(t)}$ denotes the respective matrix as in \Cref{trivialrelationsfinal} for each of the $\mathcal{G}^{(t)}$, inside the affine space $\mathbb{A}^{\mu\cdot h\cdot t_0}_{\bar{\Q}[x]}=\spec(\bar{\Q}[x][X_{i,j,t}:1\leq i\leq \mu, 1\leq j\leq h, 1\leq t\leq t_0])$. However, as we shall see later on, for all our applications it is enough to be able to describe the trivial relations for the G-functions of a fixed family $\mathcal{G}^{(t)}$.

With that in mind, what will be most useful for us later on is the following reformulation of \Cref{trivialrelationsfinal}, which gives a very convenient description for the trivial relations among the G-functions in a fixed family $\mathcal{G}^{(t)}$.

\begin{cor}\label{desidealtriv} Let $y^{(t)}_{i,j}(x)$, with $t$ fixed, $1\leq i\leq \mu$, and $1\leq j\leq h$, be as above. Then, assuming the generic special Mumford-Tate of the variation of $\Q$-HS with underlying local system $R^nf^{an}\Q$ is $Sp(\mu,\Q)$, the trivial relations among these G-functions are given precisely by the ideal \begin{equation}
		I_t:=\langle P_{j,j'}:1\leq j\leq j'\leq h \rangle,
	\end{equation}where $P_{j,j'}=\sum_{l=1}^{\mu/2} X_{\mu/2+l,j,t} X_{l,j',t}- \sum_{l=1}^{\mu/2} X_{\mu/2+l,j',t} X_{l,j,t}$.   
\end{cor}
\begin{proof}By \Cref{goalnontriviality} and \Cref{trivialrelationsfinal}, it suffices to show that the generic special Mumford-Tate group of the variation of $\Q$-HS $R^nF^{an}\Q$ will also be $Sp(\mu,Q)$ under the above hypothesis.
	
	Since $C\rightarrow S$ is a non-constant morphism we know that it will be flat and of finite presentation. Write $C_0$ for the complement in $C$ of the finite set of ramification points of $c$, and likewise $S_0$ for the complement of $S$ of the finite set of the values of these ramification points. In particular, the restriction $c_0:=c|_{C_0}:C_0\rightarrow S_0$ is \'etale, and therefore its analytification  defines a covering $C^{an}_{0}\rightarrow S^{an}_0$.
	
	Writing $F_0$ and $f_0$ for the pullbacks of the morphisms $F$ and $f$ on $C_0$ and $S_0$ respectively, we get the variations of Hodge structures $\V_C:=R^n(F^{an}_0)_{*}\Q$ and $\V_0=R^n(f^{an}_0)_{*}\Q$, supported on $C_0^{an}$ and $S_0^{an}$ respectively.

	We notice that the monodromy groups of the variations $\V_C$ and $\V_0$ are in fact equal. Indeed, for any point $s\in C_0^{an}$ we have that $\rho(\pi_1(C_0^{an},z))$ is a finite index subgroup of $\rho(\phi_1(S_0^{an},c_0(z)))$, therefore the identity components of their $\Q$-Zariski closures coincide. We write $H_z$ for this group as we did earlier.
	
	Now from \Cref{monoatgeneric} we know that $H_{c_0(z)}=Sp(\mu,\Q)$ and therefore also that $H_z=Sp(\mu,\Q)$ from the above remarks. On the other hand, again by \cite{andrefixed} as in the proof of \Cref{monoatgeneric}, we get that $Sp(\mu,\Q)\leq G_{smt}(\V_{C})$. 
	
	Finally, we knot that for all $z\in C^{an}_0$, $G_{smt}((\V_C)_z)=\leq Sp(\mu,\Q)$, since the polarization of the variation of Hodge structure $R^n(F^{an}_{0})_{*}\Q$ will be fixed by $G_{smt}((\V_C)_z)$. Thus the result follows by combining these two observations.	
\end{proof}

\part{Constructing non-trivial and global relations}\label{part:constructingnontriv}

	Having determined he trivial relations for each of the sub-families $\mathcal{G}^{(t)}$ we embark towards constructing non-trivial relations among the values of the G-functions of these families at points where we have an abundance of Hodge endomorphisms. This is essentially done in two steps. 

To do this the results of \Cref{section:derhamendo} are crucial. To take advantage of these we work throughout under the assumption that the absolute Hodge conjecture holds for the endomorphisms of the fibers of the family $f:X\rightarrow S$. As a first step we construct such relations for the values of these G-functions at archimedean places. Our construction follows the general strategy employed by Andr\'e in Chapter $X$ of \cite{andre1989g}. 

After that we construct relations, under slightly milder ''combinatorial'' assumptions than in the archimedean case, among the values of these G-functions at non-archimedean places. In this case we need to assume the full Hodge conjecture, rather than its absolute variant, to be able to take advantage of $p$-adic Hodge tools. The general strategy we use to do this is based on techniques recently developed by Urbanik in \cite{davidg}, using $p$-adic Hodge-theoretic tools. In \cite{davidg} Urbanik used these techniques to establish such relations at non-archimedean places who are ``partly CM''. meaning the associated Hodge structure has an irreducible CM-summand. While the relations we construct apply to a vastly more general set of points, they depend on having ''enough'' G-functions in each family.

\section{Some algebraic preliminaries}\label{section:pseudocmrelations}

We start first by presenting a review of some of the more ``combinatorial'' aspects of the action of the algebra of Hodge endomorphisms together with some minor lemmas pertinent to our situation. We also chose this place to fix some notation that we will use for the remainder of this part, unless explicitly stated otherwise.\\

Let $f:X\rightarrow S$ be a G-admissible variation of $\Q$-HS. We start with the following definition that will be convenient for our purposes.
\begin{definition}
	Let $s\in S(\bar{\Q})$. Assume that in the decomposition of $V_s:=H^n(X^{an}_s,\Q)$ into irreducible $\Q$-Hodge structures, as in \eqref{eq:decomprepeat}, there exists at least one irreducible factor $V_i$ whose algebra of endomorphisms $D_i$ is of type IV in Albert's classification. We then say that the point $s$, or equivalently the corresponding $\Q$-HS, is \textbf{pseudo-CM}.
\end{definition}
\begin{remark}We note here that all CM-points $s\in S(\bar{\Q})$ of the variation will satisfy the above definition. The term ``pseudo-CM'' reflects the fact that the center of a type IV algebra in Albert's classification is a CM field. We note that the points considered here are far more general, at least in principle, than special points.
\end{remark}

\subsection{Notational Conventions}\label{section:notationspseudocm}

Let $f':X'\rightarrow S'$ be the morphism associated to a G-admissible variation as above and let $(C',x)$ be the good cover of $S'$ and $C:=C'\backslash \{x^{-1}(0)\}$, as in \Cref{section:goodcovers}. Then, as we saw, we get for each $1\leq t\leq t_0:=|x^{-1}(0)|$ a G-admissible variation of $\Q$-HS associated to the morphisms $F'_t:\mathcal{X}'_t\rightarrow C'_t$ described in \Cref{section:goodcovers}. Note furthermore that all of these $F_t$ extend the same $F:\mathcal{X}\rightarrow C$ to $s_t$, for each $1\leq t\leq t_0$.

\begin{definition}\label{vadicprox}
	Let $s\in C(L)$ for some $L/K$ and set $\xi=x(s)\in L$. For a place $v\in \Sigma_{L}$ we say that $s$ is $v$\textbf{-adically close to} $0$ if $|\xi|_v < \min\{1,R_{v}(\mathcal{G}) \}$, where $\mathcal{G}$ is the family of G-functions associated to the G-admissible variation as in \Cref{defnfamilygfuns}.
	
	We say that such an $s$ is $v$\textbf{-adically close to }$s_t$ if $s$ is $v$-adically close to $0$ and furthermore it is contained in the connected component of $s_t$ in $\{|x(P)|_v<\min\{1,R_{v}(\mathcal{G}) \}\}\subset{C'}^{an}$ as in the discussion in \Cref{section:goodcovers}.
\end{definition}

From now on, we fix such a point $s\in C(L)$ that is $v$-adically close to $0$ with respect to some fixed archimedean embedding $\iota_v:L\hookrightarrow \C$.

First of all, note that from the semisimplicity of the category of polarized Hodge structures, we know that we may write \begin{equation}\label{eq:decomprepeat}
	V_s:=H^n(\mathcal{X}^{an}_s,\Q)=V_1^{m_1}\oplus\ldots \oplus V_r^{m_r},
\end{equation}with $ (V_i,\varphi_i)$ irreducible polarized $\Q$-HS that are non-isomorphic to each other. Let $D_i:=\End(V_i)^{G_{mt}(V_i)}$ be the respective endomorphism algebras so that \begin{center}
	$D_s=M_{m_1}(D_1)\times \ldots \times M_{m_r}(D_r)$.
\end{center}

From \Cref{propendodr} we know that, assuming the absolute Hodge conjecture, there exists a finite extension $\hat{L}$ of $L$ such that $D_s$ acts on $H^n_{DR}(\mathcal{X}_{s,\hat{L}} /\hat{L})$ and that this action is compatible with the comparison isomorphism between algebraic de Rham and singular cohomology. Again assuming the absolute Hodge conjecture, we know from \Cref{propdegreebound} that the degree $[\hat{L}:L]$ of the extension is bounded independently of the point $s$. We assume from now on that $\hat{L}=L$ and return to this issue in the proof of \Cref{maintheorem}.

We let $F_i$ denote the center of the algebra $D_i$ for $1\leq i\leq r$ and note that these are number fields, due to Albert's classification. We introduce the following notation \begin{itemize}
	\item $\hat{E}_s= F_1^{m_1}\times \ldots \times F_r^{m_r}$ is the maximal commutative semi-simple algebra of $D_s$,
	\item $\hat{F}_i$ is the Galois closure of the field $F_i$ in $\C$,
	\item $\hat{F}_s$ is the compositum of the fields $\hat{F_i}$ together with the field $L$.
\end{itemize}
	\subsubsection*{Splittings in cohomology and homology}

Let us assume $F:\mathcal{X}\rightarrow C$ is a G-admissible variation as above and let $s\in S(L)$, where $L/K$ is a finite extension. We assume that $s$ is $v$-adically close to $0$, with respect to our fixed inclusion $\iota_v:L\hookrightarrow \C$. In particular, by the discussion in \Cref{section:goodcovers}, we know that it will be $v$-adically close to $s_t$, for some $t$.

In particular, we assume that it is in the image of the inclusion of a punctured unit disc $\Delta^{*}\subset C^{an}_{\C}$ centered at $s_t$.\\

Under the above assumption, $L=\hat{L}$,  we know that we have two splittings. 
Namely, on the one hand we get a splitting \begin{equation}\label{eq:pseudocmhomsplitting}
	H_n(\mathcal{X}^{an}_{s,\C},\Q)\otimes_\Q \hat{F}_s= \Bigsum{\sigma:\hat{E}\rightarrow \C}{ }\hat{W}_{\sigma},
\end{equation}induced from the splitting $\hat{E}_s\otimes_{\Q} \hat{F}_s= \Bigsum{\sigma:\hat{E}_s\rightarrow \C}{} \hat{F}_s^{\sigma}$, where $\hat{F}^\sigma_s$ denotes the field $\hat{F}_s$ viewed as an $\hat{E}_s$-module with the action of $\hat{E}_s$ being multiplication by $\sigma$. We also note that on $\hat{W}_\sigma$ the algebra $\hat{E}_s$ acts again via multiplication with its character $\sigma$.

On the other hand, we have a splitting
\begin{equation}\label{eq:pseudocmdrsplitting}
	H^{n}_{DR}(\mathcal{X}_s/L)\otimes_{L} \hat{F}_s=\Bigsum{\sigma:\hat{E}_s\rightarrow \C}{}\hat{W}^{\sigma}_{DR},
\end{equation}
which once again comes from the above splitting of $\hat{E}_s\otimes_{\Q} \hat{F}_s$. In particular, we note that the action of $\hat{E}_s$ on $\hat{W}^{\sigma}_{DR}$ comes once again via $\sigma$.

\subsubsection{Duality of the splittings}

We start by highlighting how the two splittings interact with one another via the comparison isomorphism\begin{center}
	$P^n_{\mathcal{X}_s}:H^n_{DR} (\mathcal{X}_s/L)\otimes_{L} \C\rightarrow H^n(\mathcal{X}^{an}_{s,\C},\Q)\otimes_{\Q}\C$.
\end{center}The following lemma is already noted as a property of the splittings by Andr\'e, who studied this in the case of abelian varieties. We include a short proof for the sake of completeness.

\begin{lemma}\label{dualitycmsplit}
	
	For all $\sigma\neq \tau$ if $\omega\in \hat{W}^{\tau}_{DR}$ and $\gamma\in \hat{W}_{\sigma}$ then \begin{center}
		$\int_{\gamma}{} \omega=0$. 
\end{center}\end{lemma}
\begin{proof}Let us fix $\sigma\neq \tau$ as above and let $\omega \in \hat{W}^{\tau}_{DR}$ and $\gamma\in \hat{W}_\sigma$.
	
	For all $d\in \hat{E}_s$ we have that $P^n_{\mathcal{X}_s}(d\omega)=P^n_{\mathcal{X}_s} (\tau(d)\omega)=\tau(d)P^n_{\mathcal{X}_s}(\omega)=d\cdot P^n_{\mathcal{X}_s}(\omega)$, where the middle equality follows from the moreover part of \Cref{propendodr}. The algebra $\hat{E}_s$, and in particular its group of invertible elements $\hat{E}_s^{\times}$, acts by definition on $V_s$ as endomorphisms of the Hodge structure. The action of $\hat{E}_s^{\times}$ on the dual space $V_s^{*}$ will thus be the dual of that of $V_s$. 
	
	In particular for any $\gamma\in \hat{W}_\sigma$, for any $e\in\hat{E}_s^{\times }$, and for any $\delta\in V_s$, we get that $(e\cdot \gamma) (e\cdot \delta)=\gamma(\delta)$. Taking $\delta=P^n_{\mathcal{X}_s}(\omega)$ we get that for all $\gamma \in \hat{W}_\sigma$ and for all $e\in\hat{E}_s^{\times}$\begin{center}
		$ \int_{\gamma}^{}\omega=\gamma(P^n_{X_s}(\omega))=(e \cdot \gamma)(e\cdot P^n_{\mathcal{X}_s}(\omega))$.
	\end{center}But we know that $(e \cdot \gamma)(e\cdot P^n_{\mathcal{X}_s}(\omega))= (\sigma (e^{-1}) \gamma ) (\tau(e)P^n_{\mathcal{X}_s}(\omega))$, where we used the duality between the actions of $\hat{E}_s^{\times}$ on $V_s$ and $V_s^{*}$. Putting everything together we get that for all $e\in\hat{E}_s^{\times}$ we will have that
	\begin{center}
		$ \int_{\gamma}^{}\omega= \sigma(e)^{-1}\tau(e) \int_{\gamma}^{}\omega$.
	\end{center}Since $\sigma\neq\tau$ we can find such an $e$ with $\sigma(e)\neq \tau(e)$ and the lemma follows.	
\end{proof}	
	\subsection{Involutions and symplectic bases}\label{section:involutions}

In creating the relations we want we will need to construct symplectic bases with particular properties. To construct these we will need to review some facts about the involutions of the algebras of Hodge endomorphisms and see how they interact with the splittings we have.\\

For the weight $n$ $\Q$-HS given by $V_s$ we denote by $\langle,\rangle $ the symplectic form defined by the polarization on $V_s$. By duality we get a polarized $\Q$-HS of weight $-n$ on the dual space $V_s^{*}:=H_n(\mathcal{X}_{s,\C}^{an},\Q)$, and we denote the symplectic form given  by the polarization again by $\langle,\rangle$. We note that these two symplectic forms are dual.\\

The algebra $D_s$ comes equipped with an involution, which we denote by $d\mapsto d^{\dagger}$, that is defined by the relation \begin{equation}\label{eq:involdef}
	\langle d\cdot v,w\rangle =	\langle v,d^{\dagger}\cdot w\rangle ,
\end{equation}for all $d\in D_s$ and for all $v,w \in V^{*}_s$, or equivalently for all $v,w\in V_s$. 

In the decomposition \eqref{eq:decomprepeat} of $V_s$, or its dual $V^{*}_s$, the polarization on each $V_i$, or $V^{*}_i$ respectively, is given by the restriction of the polarization of $V_s$, or its dual respectively. Therefore the involution $d\mapsto d^{\dagger}$ of $D_s$  restricts to the positive involutions of the respective algebras $D_i$.

The algebra homomorphisms $\sigma:\hat{E}_s\rightarrow \C$ have a convenient description. Writing \begin{center}
	$\hat{E}_s= F_1^{m_1}\times \ldots \times F_r^{m_r}$,
\end{center} we let $\pr_{j,l}:\hat{E}_s\rightarrow F_j$, where $1\leq j\leq r$ and $1\leq l\leq m_j$, denote the projection of $\hat{E}_s$ onto the $l$-th factor of $F_j^{m_j}$, which will act respectively on the $l$-th factor of $V_j^{m_j}$ that appears in the decomposition. Then any algebra homomorphism $\sigma:\hat{E}_s\rightarrow \C$ can be written as \begin{equation} \sigma=\tilde{\sigma}\circ \pr_{j,l}
\end{equation}for some $j$ and $l$ as above and some $\tilde{\sigma}:F_j\hookrightarrow\C$. For convenience, from now on we define the notation\begin{equation}\label{eq:charofmax}
\tilde{\sigma}_{j,l}:=\tilde{\sigma}\circ \pr_{j,l}.
\end{equation}

\begin{lemma}\label{lemmainvolutions}Consider the splitting \eqref{eq:pseudocmhomsplitting}. Then for the subspaces $\hat{W}_\sigma$ the following hold:\begin{enumerate}
		\item If $\sigma =\tilde{\sigma}_{j,l}$ then $\hat{W}_\sigma$ is contained in the $l$-th factor of $(V_j^{*})^{m_j}$,
		
		\item Let $\sigma =\tilde{\sigma}_{j,l}$ and let $\tau$ be some non-zero algebra homomorphism $\hat{E}_s\rightarrow \C$. Consider non-zero vectors $v\in \hat{W}_\sigma$ and $w\in \hat{W}_\tau$. If we assume that $\langle v,w\rangle \neq 0$ then one of the following cases holds \begin{enumerate} 
		\item $\sigma =\tau$ and the algebra $D_j$ of Hodge endomorphisms is of Type I, II or III in Albert's classification, or
		\item $\sigma =\bar{\tau}$, where $\bar{(\cdot)}$ denotes complex conjugation, and $D_j$ is of Type IV in Albert's classification.
	\end{enumerate}
	\end{enumerate}\end{lemma}
\begin{proof}The first part of the lemma is trivial. 
	
	For the second part let $v\in \hat{W}_\sigma$ and $w\in \hat{W}_\tau$ be non-zero vectors as above with $\langle v,w\rangle \neq 0$. From the preceding discussion there exists a pair $(j',l')$ for $\tau $ such that $\tau=\tilde{\tau}\circ \pr_{j',l'}$, where $\tilde{\tau}:F_{j'}\hookrightarrow\C$. From the first part of this lemma we also know that $\hat{W}_\tau$ is contained in the $l'$-th factor of $(V_{j'}^{*})^{m_{j'}}$. 
	
The subspaces $V_i^{*}$ of $V_s^{*}$ are symplectic, with their symplectic inner product being the restriction of that of $V_s^{*}$. This immediately implies that $(j,l)=(j',l')$.

	For any $d\in \hat{E}_s$ we have that
	 $\langle d\cdot v,d\cdot w\rangle =     \langle \sigma(d) v,\tau(d)w\rangle   = \sigma(d)\tau(d) \langle v,w\rangle$. On the other hand using the defining property of the involution we get \begin{equation}\label{eq:involskew} 
	 	\langle d\cdot v,d\cdot w\rangle=\langle v,(d^{\dagger}d)\cdot w\rangle =
\tau(d^\dagger d) \langle v,w\rangle .\end{equation} 
	
Since, by assumption $\langle v,w\rangle \neq0$ the above relations imply that for all $d\in\hat{E}_s$ we have\begin{equation}\label{eq:characters1}
	\sigma(d)\tau(d)=\tau(d^{\dagger})\tau(d).
\end{equation}

Let $F_j$ be the center of the algebra $D_j$. Then \eqref{eq:characters1} implies that for all $d\in F_j$ \begin{equation}\label{eq:characters2}
\tilde{\sigma}(d)\tilde{\tau}(d)=\tilde{\tau}(d^{\dagger})\tilde{\tau}(d).
\end{equation}In particular, this implies that for all $d\in F_j$ we have that \begin{equation}\label{eq:characters3}\tilde{\tau}(d^{\dagger})=\tilde{\sigma}(d).\end{equation}

If $D_j$ is of Type I in Albert's classification then the involution restricts  to the identity and we get trivially that $\tilde{\tau}=\tilde{\sigma}$, and hence also $\sigma=\tau$. So our result follows in this case.

If $D_j$ is of Type II then we have that $F_j$ is a totally real field, $D_j$ is a quaternion algebra over $F_j$ and there exists $a\in D_j$ such that the involution is given by $d^{\dagger }=a d^{*}a^{-1}$ on $D_j$, where $d^{*}=\tr_{D_j/F_j}(d)-d$. Note that for $d\in F_j=Z(D_j)$ we have that $\tr_{D_j/F_j}(d)=2d$, so that $d^{\dagger}=d$ for all $d\in F_j$. Combining these observations with \eqref{eq:characters3} we get that $\tau=\sigma$. 

The same argument we just used for the case of Type II algebras works for the case of Type III algebras, though we do not need to introduce any element $a$ as above since the involution in this case is equal to the canonical involution.

Finally, let us assume that $D_j$ is of type IV in Albert's classification. In this case $F_j$ is a CM-field. In this case the involution is known to restrict to complex conjugation on the field $F_j$. In other words $d^{\dagger}=\bar{d}$ for $d\in F_j$. This, together with \eqref{eq:characters3}, implies that $\tilde{\sigma}(d)=\tilde{\tau}(\bar{d})$. Since $F_j$ is a CM-field this implies that $\tilde{\sigma}=\bar{\tilde{\tau}}$ and by extension $\sigma =\bar{\tau}$.
	\end{proof} 

\begin{remark}
	The above lemma shows that the splitting \eqref{eq:pseudocmhomsplitting} of $V_s^{*}$ is comprised of two types of mutually skew-orthogonal symplectic subspaces. On the one hand, we have the symplectic subspaces $\hat{W}_\sigma$ that are contained in some $V_j$ that is of Type I-III, and on the other hand we have the symplectic subspaces of the form $\hat{W}_{\tau}\oplus \hat{W}_{\bar{\tau}}$, where $\hat{W}_\tau$ is contained in some $V_j$ that is of Type IV. For the second type, note that we also have that $\hat{W}_\tau$ and $\hat{W}_{\bar{\tau}}$ are transverse Lagrangians of these symplectic subspaces. 
\end{remark}

	\section{Towards non-trivial relations at infinite places}\label{section:archim}

We return to our families of G-admissible variation of $\Q$-HS, given by the pairs of morphisms $(F'_t,F)$, $1\leq t\leq t_0$. As in the previous section, we let $s\in C(L)$ be a fixed point which we assume is in the punctured complex-analytic disk $\Delta^{*}\subset C^{an}$, taken with respect to a fixed archimedean embedding $i_v:L\rightarrow \C$, centered at $s_t$. We then have the totally isotropic local subsystem of rank $h$, see \Cref{goodcoverslemma}, over the ring $\mathcal{O}_{S^{an}_{\C}}|_{\Delta^{*}}$\begin{center}
	$\mathcal{M}_0:=M_0 R_n(F^{an}_{\C})_{*}(\Q)|_{\Delta^{*}}$
\end{center} of the local system $R_n(F^{an}_{\C})_{*}(\Q)|_{\Delta^{*}}$, which has rank $\mu:= \dim_\Q V_s$. 

We fix a basis of sections\footnote{See the discussion in $\S$ $4$ and $\S$ $6$ of \cite{davidg}.} $\{\omega_i:1\leq i\leq \mu\}$ of the canonical extension $\mathcal{H}$ of the vector bundle $H^n_{DR}(\mathcal{X}/C)$ to $C'$ over some dense affine open subset $U'\subset C'$ that contains the set $x^{-1}(0)$ and a trivialization $\{\gamma_j:1\leq j\leq \mu\}$ of $R_nF^{an}_{*} \Q|_V$ where $V$ is some open analytic subset of $U^{an}$, where $U:=U'\cap C$, with $s\in V\subset \Delta^{*}$. We may and do choose these so that the following conditions are satisfied:\begin{enumerate}
	\item the matrices of the skew-symmetric forms on $H^n_{DR}(\mathcal{X}/C)(U)$ and $R_nF^{an}_{*} \Q$ induced by the polarization written with respect to the basis $\{\omega_i\}$, restricted to $U$, and trivialization $\{\gamma_j\}$ respectively are both equal to $J_\mu$,
	
	\item $\gamma_1,\ldots, \gamma_h\in \mathcal{M}_0|_V$ and $\gamma_1,\ldots, \gamma_{\mu/2}\in \mathcal{M}^{+}$,
\end{enumerate}where $\mathcal{M}^{+}$ is a maximal totally isotropic local subsystem of $ R_n(F^{an}_{\C})_{*}(\Q)|_V$ that contains $\mathcal{M}_0|_V$.

Let us now consider the relative comparison isomorphism\begin{center}
	$P^n_{\mathcal{X}/C}:H^n_{DR}(\mathcal{X}/C)\otimes_{\mathcal{O}_C}\mathcal{O}_{C^{an}_{\C}}\rightarrow R^nF^{an}_{*}\Q_{\mathcal{X}^{an}_{\C}} \otimes_{\Q_{C^{an}_{\C} }} \mathcal{O}_{C^{an}_{\C}}$,
\end{center}and restrict it over the set $V$. With respect to the above choices we let $P_{\mathcal{X}/C}=\frac{1}{(2\pi i)^n}(\int_{\gamma_j}{}\omega_i)$ for the matrix of periods of $F$.

Let us write $P_{\mathcal{X}/C}=\begin{pmatrix}
	\Omega_1& \Omega_2\\
	N_1&N_2 
\end{pmatrix}$. From Theorem 2 of Chapter IX, \S 4 of \cite{andre1989g}, we know that the first $h$ columns of this matrix have entries that are G-functions. It is among their values at $\xi=x(s)$ that we want to find some relation that reflects the action of $\hat{E}_s$.

\begin{lemma}\label{constpseudocm}Assume that $h\geq 2$ and that for the point $s\in C(L)$ one of the following is true\begin{enumerate}
		
		\item there exists $\tau:\hat{E}_s\rightarrow \C$ such that $h> \dim_{\hat{F}_s} \hat{W}_\tau$, or

		\item $s$ is a pseudo-CM point and \begin{center}
			$h\geq \min\{\dim\hat{W}_\tau: \hat{W}_\tau\subset V_{i(\tau)} \text{, with } V_{i(\tau)}\text{ of type IV }\}$.
		\end{center}		
	\end{enumerate}

Then there exists a polynomial $R^{(t)}\in \bar{\Q}[X_{i,j,t}:1\leq i\leq \mu, 1\leq j\leq h]$ such that $R^{(t)}(i_v(y_{i,j}^{(t)}))=0$, and $R^{(t)}\notin I_t$, where $I_t$ is the ideal described in \Cref{desidealtriv}. Moreover, $R^{(t)}$ is a homogeneous polynomial with degree $\leq 2$.
\end{lemma}

\begin{proof} We take cases depending on the interplay between $\mathcal{M}_{0,s}\otimes \hat{F}_s$ and the splitting \eqref{eq:pseudocmhomsplitting}. We also assume that $V_s=H^n(\mathcal{X}_{s,\C},\Q)$ has a decomposition as in \eqref{eq:decomprepeat}.
	
	\textbf{Case 1:} Assume there exists some $\tau :\hat{E}_s\rightarrow \C$ such that the following holds \begin{equation}\label{eq:condicase1}
		\bigg(\Bigsum{\underset{\sigma\neq \tau}{\sigma:\hat{E}_s\rightarrow \C}}{}\hat{W}_{\sigma}\bigg)\cap (\mathcal{M}_{0,s}\otimes \hat{F}_s)\neq 0.
	\end{equation}Then we get at least one relation of degree $1$.

Note that for dimension reasons \eqref{eq:condicase1} is satisfied for the $\tau$ as in the first condition above.\\
	
	Indeed, let $\gamma \in \Gamma (V,R_n(f^{an}_\C)_{*}(\Q))$ be a section such that $\gamma(s)$ belong to the non-zero space of \eqref{eq:condicase1}. From \Cref{dualitycmsplit} we get that for all $\omega \in \hat{W}^{\tau}_{DR}$ we have 
	\begin{equation}\label{eq:case1relationps}
		\frac{1}{(2\pi i)^n}\int_{\gamma(s)}\omega=0.
	\end{equation}
	Writing $\gamma$ as an $\hat{F}_s$-linear combination of the $\gamma_j$ with $1\leq j\leq h$ and $\omega$ as an $\hat{F}_s$-linear combination of the $\omega_i$ with $1\leq i\leq \mu$, we have that \eqref{eq:case1relationps} leads to a linear equation among the values of the G-functions in question at $\xi$.
	
	Finally, we write $R^{(t)}_0\in \bar{\Q}[X_{i,j,t}:1\leq i\leq \mu, 1\leq j\leq h]$ for the linear polynomial corresponding to the above relation.\\

\textbf{Case 2:} Assume that for all $\tau :\hat{E}_s\rightarrow \C$ we have \begin{equation}\label{eq:condicase2}
	\bigg(\Bigsum{\underset{\sigma\neq \tau}{\sigma:\hat{E}_s\rightarrow \C}}{}\hat{W}_{\sigma}\bigg)\cap (\mathcal{M}_{0,s}\otimes \hat{F}_s)= 0.
\end{equation}Then we want to show that we can create a relation of degree $2$.\\	

First of all, we may assume, which we do from now on, that $\dim_{\hat{F}_s}\hat{W}_\sigma\geq h$ for all $\sigma$, otherwise we are in case 1, for dimension reasons.  

The first step in creating the relations we want is defining symplectic bases with particular properties, which we do in the following claims.\\

\textbf{Claim 1:} There exists a symplectic basis $e_1,\ldots,e_{\mu/2},f_1,\ldots,f_{\mu/2}$ of the symplectic vector space $V^{*}_s\otimes_\Q\hat{F}_s:=H_n(\mathcal{X}_{s,\C},\Q)\otimes \hat{F}_s$ that satisfies the following properties:\begin{enumerate}
	\item $\langle e_i,e_j\rangle =\langle f_i,f_j\rangle =0$ and $\langle e_i,f_j\rangle=\delta_{i,j}$ for all $i,j$.
	\item $e_j=\gamma_j(s)$ for $1\leq j\leq h$.
	\item There exists $\tau:\hat{E}_s\rightarrow \C$ such that\begin{enumerate}
		\item\label{claim13a} $\hat{W}_\tau$ is contained in $V_{i(\tau)}^{*}\otimes \hat{F}_s$, where $V_{i(\tau)}^{*}$ is some irreducible sub-Hodge Structure of $V^{*}_s$ which is of Type IV, and 
		\item\label{claim13b} $f_j\in \hat{W}_\tau$ for $1\leq j\leq  h$,
		\item $\dim_{\hat{F}_s} \hat{W}_\tau=h$.
\end{enumerate}
\end{enumerate}
\begin{proof}[Proof of Claim 1]
From \Cref{maxisotropic} we know that choosing any basis of local sections $\gamma_j(s)$ of $\mathcal{M}_{0,s}$, its vectors will satisfy $\langle \gamma_i(s),\gamma_j(s)\rangle =0$ for all $i,j$. Assume that we have fixed one such basis as above and fix an indexing of the set $\{\sigma :\sigma:\hat{E}_s\rightarrow \C \}=\{  \sigma_i:1\leq i\leq m(s)\}$. We can then write uniquely \begin{equation} \gamma_j(s)= w_{j,1}+\ldots +w_{j,m(s)}
 \end{equation}where $1\leq j\leq h$ and $w_{j,i}\in \hat{W}_{\sigma_i}$.

By assumption the $\Q$-HS $V_s$ is  pseudo-CM. Therefore, there exists $\tau$ such that $\hat{W}_\tau$ is as we want in \ref{claim13a} and the same holds for $\hat{W}_{\bar{\tau}}$. Without loss of generality assume that $\bar{\tau}=\sigma_1$. Since we are in Case 2, we also know that $\dim_{\hat{F}_s}\hat{W}_{\bar{\tau}}\geq h$ and that \eqref{eq:condicase2} holds. 

From \eqref{eq:condicase2} we get that the vectors $w_{j,1}\in\hat{W}_{\bar{\tau}}$ are in fact linearly independent. Indeed, assume that there exist $b_j\in \hat{F}_s$ not all zero with $\Sum{j=1}{h}b_j w_{j,1}=0$. We must then have \begin{center}
$\Sum{j=1}{h} b_j\gamma_j(s) = \Sum{i=2}{m(s)} \Sum{j=1}{h} b_j w_{j,i}$. 
\end{center}Now notice that the element in the right hand side of the above is in fact an element of $\bigg(\Bigsum{\underset{\sigma\neq \bar{\tau}}{\sigma:\hat{E}_s\rightarrow \C}}{}\hat{W}_{\sigma}\bigg)$. On the other hand, the element in the left hand side of the above equality is in $(\mathcal{M}_{0,s}\otimes \hat{F}_s)$. Since the $\gamma_{j}(s)$ are linearly independent this element is non-zero, this contradicts \eqref{eq:condicase2}.

By \Cref{lemmainvolutions} we know that $\hat{W}_{\bar{\tau}}\oplus \hat{W}_{\tau}$ is a symplectic vector space with $\hat{W}_{\tau}$ and $\hat{W}_{\bar{\tau}}$ being transverse Lagrangians. 

Let $v_j$ with $1\leq j\leq \dim_{\hat{F}_s}\hat{W}_{\bar{\tau}}$ be a basis of $\hat{W}_{\bar{\tau}}$ with $v_j=w_{j,1}$ for $1\leq j\leq h$. We complete this to  a symplectic basis $v_i,f_j$, with $1\leq j\leq \dim_{\hat{F}_s}\hat{W}_{\bar{\tau}}$ of $\hat{W}_{\bar{\tau}}\oplus \hat{W}_{\tau}$ such that the $f_j$ are a basis of $\hat{W}_{\tau}$. Then we have, by construction and by \Cref{lemmainvolutions}, that 
\begin{equation}
\langle \gamma_i(s),f_j\rangle =\delta_{i,j}
\end{equation}for all $1\leq i,j\leq h$. 

Therefore, setting $e_i :=\gamma_i(s)$ for $1\leq i\leq h$ the result follows by extending the set of vectors $\{e_i, f_i:1\leq i\leq h\}$ to a symplectic basis of $V_s^{*}\otimes \hat{F}_s$. Finally, note that the $\tau$ was arbitrary with $\hat{W}_{\tau}$ being contained in a type IV sub-Hodge structure of $V_s^{*}$. Therefore, by the assumption that $h\geq \min\{\dim\hat{W}_\tau: \hat{W}_\tau\subset V_{i(\tau)}^{*} \text{, with } V_{i(\tau)}^{*}\text{ of type IV }\}$ in our lemma and the  assumption in this second case that $\dim\hat{W}_\sigma\geq h$ for all $\sigma$ we get that we may find such a $\tau$ that also satisfies the last condition of our claim.
\end{proof}

From now on we fix the $\tau$ we found in Claim 1. Having created a symplectic basis for $H_n(\mathcal{X}_{s,\C}^{an},\Q)\otimes \hat{F}_s$ we want to construct a symplectic basis of $H^{n}_{DR}(\mathcal{X}_s/L)\otimes_{L} \hat{F}_s$ in a way that lets us take advantage of \Cref{dualitycmsplit}.\\

\textbf{Claim 2:} There exists a symplectic basis $e^{DR}_1,\ldots,e^{DR}_{\mu/2}, f^{DR}_1,\ldots, f^{DR}_{\mu/2}$ of $H^n_{DR}(\mathcal{X}_s/L)\otimes_{L} \hat{F}_s$ such that the following holds \begin{enumerate}
	\item\label{claim2-1} $\forall j$ we have $e^{DR}_j\in \hat{W}_{DR}^{\sigma}$ for some $\sigma \neq \tau$,
	\item for $1\leq j\leq h$ we have $f^{DR}_j\in \hat{W}^\tau_{DR}$,
	\item for $h+1\leq j\leq \mu/2$ we have $f^{DR}_j\in \hat{W}^{\sigma}_{DR}$ for some $\sigma\neq \tau$.
\end{enumerate}
\begin{proof}[Proof of Claim 2]We start by noting that the results of \Cref{lemmainvolutions} apply easily via duality to the splitting \eqref{eq:pseudocmdrsplitting} via $P^{n}_{\mathcal{X}_s}$, due to our assumption that $L=\hat{L}$. In particular, via duality we get that for $\sigma:\hat{E}_s\rightarrow \C$ the subspaces $\hat{W}_{DR}^{\sigma}$  are once again divided into two categories\begin{itemize}
		\item $\hat{W}_{DR}^{\sigma}$ that are symplectic subspaces, corresponding to $\hat{W}_\sigma$ that are contained in simple sub-Hodge structures of $V_s^{*}$, after these are tensored with $\hat{F}$, that are of Type I, II or III, and 
		\item $\hat{W}_{DR}^{\sigma}$ that are isotropic subspaces appearing in pairs such that $\sigma$ and $\bar{\sigma}$ are both algebra homomorphisms $\hat{E}_s\rightarrow \C$ and $\hat{W}_{DR}^{\sigma}\oplus\hat{W}_{DR}^{\bar{\sigma}}$ is a symplectic subspace. These correspond via duality to the $\hat{W}_\sigma$ that are contained in simple sub-Hodge structures of $V_s^{*}$, again after these are tensored with $\hat{F}$, that are of Type IV.		
	\end{itemize}
	
With that in mind, for each $\sigma$ we pick vectors $e^{\sigma}_i$ so that\begin{itemize}
		\item the $e^{\sigma}_i$ are the basis of a Lagrangian subspace of $\hat{W}_{DR}^{\sigma}$ if we are in the first case above, so that in this case $1\leq i\leq \frac{1}{2}\dim_{\hat{F}_s}\hat{W}_{DR}^{\sigma}$, 
		\item the $e^{\bar{\tau}}_i$ are a basis of $\hat{W}_{DR}^{\bar{\tau}}$ of our fixed $\tau$, and
		\item in the second case above for each $\sigma \neq \tau,\bar{\tau}$ we pick one $\sigma $ for each pair $(\sigma , \bar{\sigma})$ and we let $e^{\sigma}_i$ be a basis of $\hat{W}_{DR}^{\sigma}$.		
	\end{itemize}

Let $e^{DR}_j$, with $1\leq j\leq \mu$, be any indexing of the set of all the $e^{\sigma}_i$ above. The spanning set of these defines a Lagrangian subspace of $H^n_{DR}(\mathcal{X}_s/L)\otimes_{L} \hat{F}_s$. In a similar manner, by the above remarks derived from \Cref{lemmainvolutions}, we can construct a basis of a transverse Lagrangian to the Lagrangian spanned by the $e^{DR}_i$ with $f^{DR}_j$ also elements of the various $\hat{W}_{DR}^\sigma$. It is also straightforward from the above that we may pick $f^{DR}_1,\ldots,f^{DR}_h\in \hat{W}^{\tau}_{DR}$. 
\end{proof}
	\textbf{Step 1: Changing bases.} We note that the bases $\beta_{2}:=\{e_i,f_i:1\leq i \leq \mu/2 \}$  and $\beta_{2}^{DR}:=\{e^{DR}_i, f^{DR}_i:1\leq i\leq \mu/2 \}$ that were created above are $\hat{F}_s$-linear combinations of the bases $\beta_{1}:=\{\gamma_j(s):1\leq j\leq \mu\}$ and $\beta_{1}^{DR}:=\{\omega_j(s):1\leq j\leq \mu\}$ respectively. Since all bases are by construction symplectic the base change matrices are all symplectic matrices. Note that for the change of base matrix $[I_{\mu}]^{\beta_2}_{\beta_{1}}$ we will have by construction of $\beta_2$ that its first $h$ columns will be 
\begin{center}
	$\begin{pmatrix}
		I_h\\
		0
	\end{pmatrix}$ . 
\end{center}

Let us consider the isomorphism $P^n_{\mathcal{X}_s}:H^n_{DR}(\mathcal{X}_s/L)\otimes_{L} \C \rightarrow H^n(\mathcal{X}^{an}_s, \Q)\otimes_{\Q }\C$. Let $\tilde{P}_j$ for $j=1,2$ be the matrix\footnote{Note that in keeping with our earlier notation the matrix acts via multiplication on the right, i.e. $P^n(x)=[x]_{\beta^{DR}_j}\tilde{P}_j$.   } of this isomorphism with respect the basis $\beta^{DR}_j$ and the dual of the basis $\beta_j$. We are interested in the matrices $P_j:=\frac{1}{(2\pi i)^n}\tilde{P}_j$. Note that $P_1$ is the value of the relative period matrix at $\xi=x(s)$. For the matrices $P_j$ we have 
\begin{equation}\label{eq:changebases}
	P_2=[I_{2\mu}]^{\beta^{DR}_1}_{\beta^{DR}_{2}} P_1 [I_{2\mu}]^{\beta^{\vee}_2}_{\beta^{\vee}_{1}},
\end{equation}
and all these matrices are symplectic, while the two change of base matrices will have coefficients in the field $\hat{F}_s$.\\

\textbf{Step 2: Relations on $P_2$.} Let us examine the matrix $P_2$ in more detail. Write \begin{equation}\label{eq:matrixp2}
	P_2=\begin{pmatrix} \Gamma_1 & \Gamma_2\\
		\Delta_1 & \Delta_2
	\end{pmatrix},
\end{equation}where $\Gamma_i$ and $\Delta_i$ are $\mu/2\times \mu/2$ matrices. For convenience we also let $\tilde{\Gamma}_i$ and $\tilde{\Delta}_i$ for the $\mu/2\times h$ matrices defined by the first $h$ first columns of the matrices $\Gamma_i$ and $\Delta_i$ for $i=1,2$ respectively.

From the fact that $(2\pi i)^nP_2$ is symplectic we have the relations\footnote{This follows from the Riemann relations. See the \Cref{appendixpolarizations} for more details.}
\begin{equation}\label{eq:relationsp2-1}^t\Delta_j\Gamma_j=^t\Gamma_j \Delta_j,	
\end{equation}for $j=1,2$ and also
\begin{equation}\label{eq:relationsp2-2}^t\Delta_2\Gamma_1-^t\Gamma_2\Delta_1=\frac{I_{\mu/2}}{(2\pi i)^n}
\end{equation}

By construction of the bases in the two claims, and in particular Claim 1 \ref{claim13b} and Claim 2 \ref{claim2-1},  we immediately get that $\tilde{\Gamma}_2=0$, by virtue of \Cref{dualitycmsplit}. 

Let us set $\Delta_{2,h}$ to be the $h\times h$ matrix given by $(\int_{f_j}^{}f^{DR}_i)_{1\leq i,j\leq h}$. Let us also set $\Delta_{1,h}$ to be the $h\times h$ matrix given by $(\int_{\gamma_{j}}^{}f^{DR}_i)_{1\leq i,j\leq h}$. In other words, $\Delta_{i,h}$ is the submatrix of $\Delta_i$ that is comprised of the entries in the first $h$ columns and first $h$ rows of $\Delta_i$.\\

\textbf{Claim 3: }There exists an $h\times h$ matrix $T\in\GL_h(\hat{F}_s)$ which is such that \begin{equation}\label{eq:deltasrelation}\Delta_{1,h}={\Delta}_{2,h} \cdot^t T.
\end{equation}
\begin{proof}[Proof of Claim 3] We have a pairing \begin{equation}\label{eq:deltachange}
		\hat{W}^{\tau}_{DR}\times H_n(\mathcal{X}^{an}_s,\Q)\otimes \hat{F}_s \rightarrow \C,
	\end{equation}defined by $(\omega, \gamma)\mapsto \int_{\gamma}^{} \omega$.

By \Cref{dualitycmsplit} this induces a perfect pairing \begin{center}
	$\hat{W}^{\tau}_{DR}\times \bigg((H_n(\mathcal{X}^{an}_s,\Q)\otimes \hat{F}_s)\bigg{/}\bigg(\Bigsum{\underset{\sigma\neq \tau}{\sigma:\hat{E}_s\rightarrow \C}}{}\hat{W}_{\sigma}\bigg)\bigg)\rightarrow \C$,
\end{center}

On the one hand, we know that $\{f_1,\ldots, f_h\}$, the basis of $\hat{W}_\tau$, maps to a basis in the quotient $(H_n(\mathcal{X}^{an}_s,\Q)\otimes \hat{F}_s)\bigg{/}\bigg(\Bigsum{\underset{\sigma\neq \tau}{\sigma:\hat{E}_s\rightarrow \C}}{}\hat{W}_{\sigma}\bigg)$. On the other hand, from the assumption \eqref{eq:condicase2}, we get that the basis $\{\gamma_1,\ldots, \gamma_h\}$ also maps to a basis of the same quotient.

Let $T$ be the transpose of the change of basis matrix  from the basis induced by the $\gamma_j$ on $(H_n(\mathcal{X}^{an}_s,\Q)\otimes \hat{F}_s)\bigg{/}\bigg(\Bigsum{\underset{\sigma\neq \tau}{\sigma:\hat{E}_s\rightarrow \C}}{}\hat{W}_{\sigma}\bigg)$, to that induced on the same space by the $f_j$. Then $T\in \GL_h(\hat{F}_s)$ and \eqref{eq:deltachange} holds.
\end{proof}

Let $\Gamma_{1,h}$ be, once again, the submatrix of $\Gamma_1$ that is comprised of the entries in the first $h$ columns and first $h$ rows of $\Gamma_1$. We have already seen that $\tilde{\Gamma}_2=0$ and, by the construction of Claims $1$ and $2$ and \Cref{dualitycmsplit}, that $\tilde{\Delta}_2 =\begin{pmatrix}
	\Delta_{2,h}\\0\\
	\vdots \\0
\end{pmatrix}$.

Using these facts we derive from \eqref{eq:relationsp2-2} that \begin{equation}\label{eq:relationsp2-3}
		\prescript{t}{}{\Delta_{2,h}} \Gamma_{1,h}=\frac{1}{(2\pi i)^n}I_h.
\end{equation}

Let us now multiply both sides of \eqref{eq:relationsp2-3} by $T$. Then, using \eqref{eq:deltasrelation}, we get\begin{equation}\label{eq:relationsp2-4}
\prescript{t}{}{\Delta}_{1,h}\Gamma_{1,h}=\frac{T}{(2\pi i)^n}.
\end{equation}

	\textbf{Step 3: Relations on $P_1$.} The relation \eqref{eq:relationsp2-4} we created on the first $h$ columns of the matrix $P_2$ will translate to relations among the coefficients of the first $h$ columns of the matrix $P_1$. Since these are the values of the G-functions we are interested in, this will finish the proof of \Cref{constpseudocm}.\\

We start with introducing some notation let $[I_{\mu}]^{\beta^{\vee}_2}_{\beta^{\vee}_{1}} =\begin{pmatrix}
	A_1& B_1 \\C_1& D_1
\end{pmatrix}$ and $[I_{\mu}]^{\beta^{DR}_1}_{\beta^{DR}_{2}}  =\begin{pmatrix}
	A_2& B_2 \\C_2& D_2
\end{pmatrix}$, where the $A_i$, $B_i$, $C_i$, and $D_i\in M_{\mu/2}(\hat{F}_s)$. In keeping the same notation as above, for any matrix $A\in M_{\mu/2}(\C)$ we define $\tilde{A}$ to be the $\mu/2\times h$ matrix defined by the first $h$ columns of $A$. Note that by our construction in Claim 1 we know that $\tilde{C}_1=0$ and $\tilde{A}_1=\begin{pmatrix}
	I_h\\0
\end{pmatrix}$.

With this notation \eqref{eq:changebases} becomes \begin{equation}
	\begin{pmatrix}
		\Gamma_1& \Gamma_2 \\
		\Delta_1& \Delta_2
	\end{pmatrix}=\begin{pmatrix}
		A_2& B_2 \\C_2& D_2
	\end{pmatrix} 
	\begin{pmatrix}
		\Omega_1(s)& \Omega_2(s) \\N_1(s)& N_2(s)
	\end{pmatrix}  
	\begin{pmatrix}
		A_1& B_1 \\C_1& D_1
	\end{pmatrix}.
\end{equation}

From this we get the following two relations\begin{equation}\label{eq:gamma1}
	\Gamma_1= A_2\Omega_1(s) A_1 +A_2 \Omega_2(s) C_1+B_2N_1(s)A_1+B_2 N_2(s)C_1,
\end{equation}and 
\begin{equation}\label{eq:delta1}
	\Delta_1=C_2\Omega_1(s)A_1 + C_2 \Omega_2(s) C_1+ D_2 N_1(s) A_1+ D_2 N_2(s) C_1.
\end{equation}

Now we notice that for any matrices $A,B\in M_\mu(\C)$ we have that $\widetilde{AB}=A\tilde{B}$ and that $A\cdot \begin{pmatrix}
	I_h\\
	0
\end{pmatrix}=\tilde{A}$. Using these observations on \eqref{eq:gamma1} and \eqref{eq:delta1} we get
\begin{equation}
	\tilde{\Gamma}_1=A_2 \tilde{\Omega}_1(s)+B_2 \tilde{N}_1(s), 	
\end{equation}and 
\begin{equation}
	\tilde{\Delta}_1=C_2 \tilde{\Omega}_1(s)+D_2 \tilde{N}_1(s)
\end{equation}
Substituting these in \eqref{eq:relationsp2-4} we get
\begin{equation}\label{eq:case2final}
	\prescript{t}{}{((C_2\tilde{\Omega}_1(s)+D_2 \tilde{N}_1(s) )_h)}( A_2\tilde{\Omega}_1(s) +B_2\tilde{N}_1(s)    )_h= \frac{T}{(2\pi i)^n},
\end{equation}where, using the same notation as earlier, the subscript $h$ signifies that we are considering the $h\times h$ submatrices that are comprised by the first $h$ rows of these $\mu/2\times h$ matrices.

Since we are assuming that $h\geq 2$, equation \eqref{eq:case2final} provides relations among the values of the G-functions we want at $\xi=x(s)$ that, upon getting rid of the factor $(2\pi i)^n$, correspond to homogeneous polynomials with coefficients in $\hat{F}_s$ and degree $\leq 2$. 

As in the previous case, we let $R^{(t)}_1\in \bar{\Q}[X_{i,j,t}:1\leq i\leq \mu, 1\leq j\leq h]$ be the polynomial corresponding to the relation defined by \eqref{eq:case2final}.\\

\textbf{Non-triviality of the relations:} Let now $R^{(t)}$ be either $R^{(t)}_0$ or $R^{(t)}_1$ depending on whether for our chosen point $s$ we are in case $1$ or case $2$ above. Notice that by construction $R^{(t)}$ is a homogeneous polynomial of degree $\leq 2$. 

To check non-triviality of the relation among the G-functions in the family $\mathcal{G}$, associated to the G-admissible variation, we just  have to establish that $R^{(t)}\notin I_t$, where $I_t$ is the ideal of \Cref{desidealtriv}, since only the values of the G-functions from the sub-family $\mathcal{G}^{(t)}$ appear in the relation.

If we are in case $1$ this is easily done, since $R^{(t)}_0$ is linear and the ideal $I_t$ is generated by the polynomials $P_{j,j'}$ described in \Cref{desidealtriv}. Finally, if we are in case $2$ above, the result follows again by comparing the description of the generators $P_{j,j'}$, and in more detail the monomials that appear in those, with the polynomial corresponding to \eqref{eq:case2final}. This can be more easily seen by comparing the $P_{j,j'}$ with the polynomial we would have from \eqref{eq:relationsp2-4}. 
\end{proof}	
	\subsubsection{Some cleaning up}

The technical conditions 
\begin{equation}\label{eq:conditionoflemma1}
\exists\tau:\hat{E}_s\rightarrow \C \text{ such that } h> \dim_{\hat{F}_s} \hat{W}_\tau
\end{equation}

\begin{equation}\label{eq:conditionoflemma2}
	h\geq \min\{\dim\hat{W}_\tau: \hat{W}_\tau\subset V_{i(\tau)} \text{, with } V_{i(\tau)}\text{ of type IV }\}.
\end{equation}that appear in \Cref{constpseudocm} are by no means aesthetically pleasing! We have dedicated this short section to remedy this fact. In fact we prove the following lemma.

\begin{lemma}\label{remedy}
	
	Condition \eqref{eq:conditionoflemma1} is equivalent to the condition\begin{center}
		$h> \frac{\dim_\Q V_j}{[Z(D_j):\Q]}$ for some $j$,
	\end{center}and condition \eqref{eq:conditionoflemma2} is equivalent to the condition\begin{center}
		$h\geq \min\{  \frac{\dim_\Q V_i}{[Z(D_i):\Q] } : i \text{ such that }  D_i=\End_{HS}(V_i)  \text{ is of type IV } \}.$
	\end{center}
\end{lemma}

To prove this we work in greater generality with modules of semisimple algebras over $\Q$. The material in this section is definitely not new but we include it for the sake of completeness of our exposition.\\

Let us fix some notation. We consider a $\Q$-HS $V$ with $\mu:=\dim_\Q V$ that decomposes as $V=V_1^{m_1} \oplus \cdots\oplus V_r^{m_r}$. We write $D= M_{m_1} (D_1) \oplus \cdots M_{m_r }(D_r)$ for the algebra of Hodge endomorphisms of $V$, where $D_i$ is the algebra of Hodge endomorphisms of $V_i$. For each $i$ we let $F_i:= Z(D_i)$ be the center of $D_i$ and $f_i:=[F_i:\Q]$. Finally, we let $\hat{F}$ be the Galois closure of the compositum of the fields $F_i$ and $\hat{E}:= F_1^{m_1} \oplus \cdots\oplus F_r^{m_r}$ be the maximal commutative semisimple sub-algebra of $D$.

For the non-trivial homomorphisms of algebras $\sigma : \hat{E} \rightarrow \hat{F}$ we write $\sigma=\tilde{\sigma}_{j,l}$ as we did earlier. The above result then follows from the following lemma.

\begin{lemma}\label{remedy2} The $\hat{E}\otimes_\Q\hat{F}$-module $V\otimes_\Q \hat{F}$ has a decomposition as an $\hat{E}\otimes_\Q\hat{F}$-module as 
\begin{center}
	$V\otimes_\Q \hat{F}= \Bigsum{\sigma:\hat{E}\rightarrow \hat{F}}{ }\hat{W}_{\sigma}$, 
\end{center}
where $\hat{W}_\sigma$ are $\hat{F}$-subspaces of $V\otimes_\Q \hat{F}$ on which $\hat{E}\otimes_{\Q}\hat{F}$ acts via multiplication by $\sigma$. Moreover, $\dim_{\hat{F}} W_\sigma = \frac{\dim_\Q V_{i(\sigma)}}{f_{i(\sigma)}}$ where $i(\sigma)\in\{1,\ldots,r\}$ is such that $\sigma =\tilde{\sigma}_{i(\sigma),l}$ for some $l$ and $\tilde{\sigma}$ with our previous notation.
\end{lemma}
\begin{proof}
	First of all, note that $\forall i$ we have $F_i\hookrightarrow \End_\Q V_i$ trivially. Therefore $V_i$ is isomorphic, as an $F_i$-module, to \begin{equation}\label{eq:isomfimodules}
		V_i\simeq F_i^{t_i}
	\end{equation}for some $t_i$. Counting dimensions of these as $\Q$-vector spaces we get that $t_i= \frac{\dim_\Q V_i}{f_i}$.

Tensoring both sides of \eqref{eq:isomfimodules} by $\otimes_\Q\hat{F}$ we get that $V_i\otimes_{\Q} \hat{F}\simeq (F_i\otimes_{\Q} \hat{F})^{t_i}$ as $F_i$-modules.  Now note that since $\hat{F}$ is a Galois extension that contains $F_i$ we have that $	F_i\otimes_{\Q} \hat{F}  \simeq\Bigsum{\tilde{\sigma}:F_i\rightarrow \hat{F}}{ }\hat{F}^{\tilde{\sigma}}$, where $\hat{F}^{\tilde{\sigma}}$ is just $\hat{F}$ viewed as an $F_i$-module via the action of the embedding $\tilde{\sigma}:F_i\hookrightarrow \hat{F}$. Combining the above we get\begin{center}
		$V_i\otimes_{\Q} \hat{F}\simeq \Bigsum{\tilde{\sigma}:F_i\rightarrow \hat{F}}{ }\big(\hat{F}^{\tilde{\sigma}}\big)^{t_i}$.
	\end{center}

The result now follows trivially.	
\end{proof}
	\section{Non-trivial relations at infinite places}

As in \Cref{section:notationspseudocm} we consider a fixed G-admissible variation of Hodge structures underlied by the morphism $f':X'\rightarrow S'$. We also let $F'_t:\mathcal{X}'_t\rightarrow C'_t$ be the family morphisms underlying the G-admissible variations associated to a good cover of the curve $S'$ as in \Cref{goodcovers}.

Consider now $s\in C(L)$, where $L/K$ is some finite extension, be a point that is $v$-adically close to $0$ for at least one archimedean place $v\in \Sigma_{L,\infty}$. Also, in this section we abandon the assumption that $L=\hat{L}$, where $\hat{L}$ is the field defined by \Cref{propendodr}, that was made in the previous section, mainly for reasons of notational simplicity. 

The relations created in \Cref{constpseudocm} were created after fixing a place $v\in \Sigma_{L,\infty}$, corresponding to an inclusion $i_v:L\rightarrow \C$. This is because we assumed that $s$ is archimedeanly close to $s_t$, with respect to this fixed embedding $L\hookrightarrow \C$.

We want to create relations among the values of the G-functions in the associated family $\mathcal{G}$, see \Cref{defnfamilygfuns}, at $\xi :=x(s)$ for all places $v\in \Sigma_{L,\infty}$ as above. 

In order to be able to create these we will need the following technical lemma, following the exposition in Ch.X, \S $3.1$ of \cite{andre1989g}. We fix a priori, the matrix \begin{equation}\label{eq:matrixofgfunctions}
	G^{(t)}:=(y_{i,j}^{(t)})\in M_{\mu\times h}(K [[x]]),
\end{equation}for the G-functions of the sub-family $\mathcal{G}^{(t)}$ of $\mathcal{G}$ for a fixed $1\leq t\leq t_0=|x^{-1}(0)|$.

For now we fix one of the $t\in \{1,\cdots,t_0\}$ and let, by abuse of our notation $G:=G^{(t)}$. Let us furthermore consider $\iota :K \hookrightarrow \C$ to be a random complex embedding of $K$. We then have the complex Taylor series $\iota (y^{(t)}_{i,j})$. We also let $G_{\iota}$ be the matrix defined analogously to $G$ with the $y^{(t)}_{i,j}$ replaced by the power series $\iota(y^{(t)}_{i,j})$.

\begin{lemma}\label{changeofplace}For any $\iota$ as above the matrix $G_\iota$ is again the matrix that consists of the entries in the first $h$ columns of a period matrix with respect to the same basis of local sections of $H^n_{DR}(\mathcal{X}/C)$ and to some local frame of the local system $R_n(F^{an}_{\iota,\C})_{*}(\Q_{\mathcal{X}^{an}_\C})$. 	
\end{lemma}

	\begin{remark} Here by $F^{an}_{\iota,\C}$ we denote the analytification of the morphism $F_{\iota,\C}$, where $F_{\iota,\C}$ is the morphism induced from $F:\mathcal{X}\rightarrow C$ via the base change given by $\iota :\spec\C\rightarrow \spec K$.\end{remark}
\begin{proof} This follows essentially from the proof of Theorem 2 in Ch. IX, \S 4.1 of \cite{andre1989g}, which constitutes \S 4.4 of the same chapter. We review the main parts we need.\\
	
	\textbf{Notation:} We introduce a bit of notation, following Andr\'e's exposition. We write $Y_i$ for the components of the divisor $Y={F'_t}^{-1}(s_t)$, see \Cref{section:notations} for our conventions on $Y$. Furthermore, we set
	\begin{center}
		$Y^{[d]}:= \displaystyle{\bigsqcup_{i_0<\ldots <i_d}   \bigcap_{j=0}^{d} Y_{i_j}}$.
	\end{center}
	
	We note that the proof in loc.cit. also shows that upon assuming $h>0$, as we do, we must have that there are $n+1$ such components $Y_i$. In this case, i.e. the case when $h>0$, the same proof shows that the intersection $Y^{[n]}$ of all the components $Y_i$ is non-empty.\\
	
	\textbf{Step 1: A short review of the construction.} For each point $Q\in Y^{[n]}$ we can find an affine open subset $U^Q$ of $\mathcal{X}'_t$ admitting algebraic coordinates $x_{Q,1},\ldots ,\ldots,x_{Q,n+1}$ such that\footnote{Here we write $Z(x_{Q,i})$ for the Zariski closed subset of $U^Q$ defined by the equation $x_{Q,i}=0$. In other words $Y_i$ is cut out by $x_{Q,i}=0$ affine locally near $Q$.} $Y_i\cap U^Q=Z(x_{Q,i})$ and the local parameter $x$ of $C'_t$ at $s_t$ lifts to $x=x_{Q,1}\cdots x_{Q,n+1}$. To ease our notation we write simply $x_i$ for $x_{Q,i}$. We also fix the inclusion $i_Q:U^Q\rightarrow \mathcal{X}'_t$. 
	Then loc.cit. describes a horizontal map $T_Q:H^n_{DR}(\mathcal{X}_t'/C'_t(\log Y)) \rightarrow K[[x]]$.
	
	This map $T_Q$ also has an analytic description. To define it one needs some cycles ${i_Q}_{*}\gamma_Q$. We briefly review the definition of these cycles. For each $z\in \Delta$ we have the cycle $\gamma_{Q,z}\in H_{n} ((U^{Q}_z)^{an},\Z)$ defined by the relations $|x_2|=\ldots =|x_{n+1} |=\epsilon$ and $x_1x_2\cdots x_{n+1}=x(z)$, where $\epsilon>0$ is small. These cycles glue together to define a section $\gamma_Q \in H^{0}( \Delta, R_{n}(F'_t|_{\Delta}\circ i_Q )_{*} \Q  )$ which we can push-forward to a cycle ${i_Q}_{*} \gamma_Q \in H^{0} (\Delta , R^n(F'_t|_{\Delta} )_{*} \Q)$. 
	
	We note that $({i_Q}_{*} \gamma_Q)_z\in H_{n} (\mathcal{X}^{an}_z,\Q)$ is also invariant by the action of $\pi_1(\Delta^{*},z)$. In fact from the exposition in loc. cit. we know that the cycles $({i_Q}_{*} \gamma_Q)_z$ span the fiber $M_0R_n(F^{an}_\C)_{*} (\Q)_z$ for $z\in\Delta^{*}$. 
	
	From the analytic description of $T_Q$ we get that for $1\leq i\leq h\mu $ there exists a point $Q\in Y^{[n]}(\bar{\Q})$ such that the entry $y^{(t)}_{i,j}|_{\Delta}$ is equal to \begin{center}
		$T_Q(\omega) =\frac{1}{(2\pi i)^n} \int_{{i_Q}_{*}\gamma_Q}^{} \omega$ 
	\end{center}for some $\omega\in H^n_{DR}(\mathcal{X}'_t/C'_t(\log Y))$, where $\Delta$ is a unit disk centered at $s_0$.\\
	
	To be able to work over $K$, instead of $\bar{\Q}$ as loc. cit. does, we assume without loss of generality that all of the above, i.e. the points $Q$, the algebraic coordinates, and the coefficients of the $y^{(t)}_{i,j}$ are actually defined over our original field $K$. To achieve this we might have to a priori base change everything, i.e. $f:X\rightarrow S$ and $f':X'\rightarrow S'$, by some fixed finite extension $\hat{K}$ of our original field $K$.\\
	
	\textbf{Step 2: Changing embeddings.} Implicit in the definition of the cycles ${i_Q}_{*}\gamma_Q$ is the fixed embedding $K\hookrightarrow\C$. Shifting our point of view  to the embedding $\iota:K\hookrightarrow \C$ we get a similar picture. Given a $K$-variety $Z$ we define $Z_{\iota}:=Z\times_{\spec K,\iota}\spec \C$ the base change of $Z$ via $\iota:\spec \C\rightarrow \spec K$ and similarly for the base change of a morphism $\phi:Z_1\rightarrow Z_2$ between $K$-varieties. In other words we suppress reference to the original embedding $K\hookrightarrow\C$ but keep track of the new embeddings.
	
	The algebraic coordinates $x_1,\ldots ,x_{n+1}$ on $U^Q$ pullback to algebraic coordinates $\iota^{*}x_1,\ldots,\iota^{*}x_{n+1}$ on $U^Q_\iota$. We write $x_{i,\iota}$ for $\iota^{*} x_i$ and also consider a unit disk $\Delta_\iota \subset (C'_{t,\iota})^{an}$ centered at $s_t$. 
	
	Once again we have that $x_{1,\iota}\cdots x_{n+1,\iota}=x_{\iota}$. We define the cycles $\gamma_{Q,\iota}$ similarly:\\
	
	for $z\in \Delta_{\iota}$ we let $\gamma_{Q,\iota,z}\in H_{n} ((U^Q_{\iota})^{an},\Z)$ be defined by $|x_{2,\iota}|=\ldots =| x_{n+1,\iota}|=\epsilon$ and $x_{1,\iota}x_{2,\iota}\cdots x_{n+1,\iota}=x_{\iota}(z)$. Once again these glue together to give cycles ${i_{Q,\iota}}_{*}\gamma_{Q,\iota}\in H^{0} (\Delta_\iota, R_{n}(F_{t,\iota}'|_{\Delta_\iota})^{an}_{*}\Q)$. \\
	
	The cycles $({i_Q}_{*}\gamma_{Q,\iota })_z$, for $Q$ varying in the set $Y^{[n]}$, will span the fiber of the local system $M_0 R_n(F^{an}_{\iota} )_{*} (\Q_{X^{an}_{\iota, \C}})_z $ for $z\in{\Delta_{\iota}^{*}} $. This follows from the exposition in loc.cit. since the proof does not depend on the embedding $K\hookrightarrow\C$.
	
	Among these we may choose a frame of $M_0R_n(F^{an}_{\iota} )_{*} (\Q_{\mathcal{X}^{an}_{\iota, \C}})|_{V}$ and then extend that to a frame of $R_n(F^{an}_{\iota} )_{*} (\Q_{\mathcal{X}^{an}_{\iota, \C}})|_{V}$, where $V\subset \Delta_\iota^{*}$ is some simply connected open subset of $\Delta_{\iota}^{*}$. We thus get a relative period matrix $P_1$ of the morphism $F$.  Finally, Deligne's trick, see Remark 1 page 21 of \cite{andre1989g} together with the exposition in the aforementioned proof show that in fact $G_1=G_\iota$, where $G_1$ is the matrix that consists of the first $h$ columns of $P_1$, and the result follows.
\end{proof}

	\subsubsection{Construction of the actual relations}

Let $s\in C(L)$ be a point of the variation satisfying either of the conditions of \Cref{constpseudocm}. We assume that $s$ is $v_0$-adically close for some fixed $v_0\in \Sigma_{\hat{L},\infty}$, with $\hat{L}$ as in \Cref{propendodr}. Considering the embedding $i_{v_0}:\hat{L}\rightarrow \C$, which we drop from notation from now on writing just $\hat{L}\hookrightarrow \C$, the construction of \Cref{constpseudocm} goes through.

We consider now $G^{(t)}$, as in \eqref{eq:matrixofgfunctions} above, to be the matrices of G-functions, one for each sub-family $\mathcal{G}^{(t)}$, created with respect to the embedding $i_{v_0}|K:K\hookrightarrow \C$ of $K$.

For any other place $v\in \Sigma_{\hat{L},\infty}$ such that $s$ is $v$-adically close to $0$ we can find $t=t(v)\in\{1,\cdots,t_0\}$, depending on the place $v$ as in the discussion in \Cref{section:goodcovers}, such that $s$ is $v$-adically close to $s_{t(v)}$. We can thus repeat the process of \Cref{constpseudocm} for the sub-family $\mathcal{G}^{(t(v))}$ of G-functions. This time we replace $K$ by $i_v(K)$, $\hat{L}$ by $i_v(\hat{L})$, and $\mathcal{X}_{s,\hat{L}}$ by $\mathcal{X}_s\times_{L}i_v(\hat{L})$. Thanks to \Cref{changeofplace} we may choose trivializations so that the corresponding $\mu\times h$ matrix of G-functions we are interested in is $G^{(t(v))}_{\iota_v}$.

As a result for any such archimedean place $v$ we get a polynomial $R^{t(v)}$ with coefficients in $\bar{\Q}$ such that $i_v(R^{t(v)})(i_v(y^{t(v)}_{i,j})(i_v(\xi)) )=0$. We let \begin{equation}\label{eq:actualrelation}R_{s,\infty}=
	\underset{\underset{|\xi|_v<\min\{1,R_v(\mathcal{G})\}}{v\in \Sigma_{\hat{L},\infty}}}{\Pi}R^{t(v)}.\end{equation}
The relation we are looking for, in view of \Cref{hasse}, is partly the one coming from this polynomial. Note that this relation holds $v$-adically for all archimedean places of $\hat{L}$ for which $|\xi|_v<\min\{1,R_v(\mathcal{G})\}$, by construction. 

Later on, we construct another polynomial $R_{s,fin}$ that defines a relation coming from any possible non-archimedean proximity of $s$ to $0$. Taking $R_s:=R_{s,\infty}\cdot R_{s,\fin}$ and considering the relation it induces at $\xi$, we will get a global  relation. 

Leaving the proof of globality for later, we note that the relation induced from \eqref{eq:actualrelation} satisfies the other key property we want. Namely we have the following lemma.
\begin{lemma}\label{proofnottrivial}Assume that the absolute Hodge conjecture holds for the Hodge endomorphisms of the fibers of $f$. Then, the polynomial $R_{s,\infty}$ is homogeneous of degree $\leq C_1(f)\cdot [L:K]$, where $C_1(f)$ is a positive constant that depends only on the morphism $f$.
	
	Furthermore, the relation induced on the values of the family $\mathcal{G}$ of G-functions at $\xi=x(s)$ from the polynomial $R_{s,\fin}$ above is non-trivial, assuming that the generic special Mumford-Tate group of our variation is $Sp(\mu,\Q)$.
\end{lemma}
\begin{proof}Since we have assumed the absolute Hodge conjecture holds for endomorphisms in our case, we get from \Cref{propendodr} that the extension $\hat{L}/\Q$ is finite. Furthermore from \Cref{propdegreebound} we know that $[\hat{L}:L]\leq C_0(\mu)$, where $\mu$ is the dimension of the Hodge structures of the fibers of $f$, and hence of $F$. 
	
	Thus we have by construction that the product defining $R_{s,\fin}$ is finite. Indeed, there are at most $[\hat{L}:\Q]$-many polynomials that appear in this product, one for each of the places $v\in \Sigma_{\hat{L},\infty}$ for which $|\xi|_v<\min\{1,R_v(\mathcal{G})\}$. From above we thus get that the number of these factors is at most $[\hat{L}:\Q]\leq C_0(\mu)[L:\Q]$.
	
	Now from \Cref{constpseudocm} we know that each of the factors that appear is homogeneous of degree $\leq 2$ and the first part of the lemma follows.
	
	To show non-triviality we need to show that the polynomial $R_{s,\infty}$ does not define relations that hold on the functional level among the elements of $\mathcal{G}$. Assume otherwise. 
	
	We would then have that $R^{(t)}(y^{(t)}_{i,j})=0$ for some $t$. In other words, we would have a functional relation among the members of the sub-family $\mathcal{G}^{(t)}$. The impossibility of this has been already ruled out in \Cref{constpseudocm}.\end{proof}

\section{Non-trivial relations at finite places}\label{section:nonarchi}

We follow the general notation of \Cref{section:notationspseudocm}. Namely we fix a morphism $f':X'\rightarrow S'$ defined over a number field $K$ that underlies a G-admissible variation of $\Q$-HS of weight $n$. We then consider the associated family of G-admissible variations coming from a good cover $C'$ of $S'$, together with the underlying morphisms $F'_t:\mathcal{X}'_t\rightarrow C'_t$. We also assume that the polarization of the variation induces a symplectic form on each of the fibers.

We fix for the remainder of this subsection a point $s\in C(L)$, and write $Y:=\mathcal{X}_s$. We write $V_{dR}:=H^n_{DR}(Y/L)$ and let $V_{\Q}:=H^n(Y^{an},\Q)$ for the analytification of $Y$ with respect to an embedding $i:L\rightarrow\C$. 

As in \eqref{eq:decomprepeat}, we write $V_{\Q}=V_1^{m_1}\oplus\ldots V_r^{m_r}$ for the decomposition of $V_{\Q}$ into a direct sum of irreducible polarized HS $V_i$. We also write $D=\End_{HS}(V_{\Q})$ and $D_i=\End_{HS}(V_i)$ so that we have as earlier the decomposition 
\begin{center}
	$D=M_{m_1} (D_1)\oplus\ldots\oplus M_{m_r}(D_r)$.
\end{center}As usual, we write $e_i:=[F_i:\Q]$ where $F_i:=Z(D_i)$, and let $\mu:=\dim_{\Q} V_{\Q}=: 2\nu$. We will also let $\nu_i:=\frac{\dim_{\Q}V_i}{2}$, which will be natural numbers since all these subspaces are symplectic subspaces, with symplectic form coming from the polarization.

\subsection{Points with algebras with atypically large centers}\label{section:largecenters}

The first non-archimedean relations we create are for points whose corresponding algebras have large centers. Due to the complexity of the computations of the general case we start first with a lemma whose techniques, as we will see, deal with the more general case as well.

\begin{lemma}\label{nonarchirel}Let $V_{\Q}$, $V_{dR}$, $D_i$, $F_i$, $e_i$ , and $L/\Q$ be as above. We assume that the Hodge Conjecture holds for Hodge Endomorphism of $V_{\Q}$ and, as usual, that $h\geq 2$. If $(e_i-1)\cdot h\geq \nu_i \geq h$ for some $i$, and $\{v\in \Sigma_{L,f}: s$ is $v$-adically close to $0\}\neq \emptyset$, then there exists a homogeneous polynomial \begin{center}
		$R_{s,\fin}\in \bar{\Q}[X_{i,j,t}; 1\leq i\leq \mu$, $1\leq j\leq h$, $1\leq t\leq t_0]$ such that
	\end{center}

\begin{enumerate}
	\item\label{condition1} $R_{fin}(i_v(y_{i,j}^{(t)}(x(s))))=0 $, for all $v\in \Sigma_{L,f}$ for which $s$ is $v$-adically close to $0$,
	
	\item\label{condition2} the relation induced by $R_{s,\fin}$ on the values of the G-functions in the family $\mathcal{G}$ at $x(s)$ is non-trivial, and 
	
	\item\label{condition3} $\deg R_{s,\fin}\leq C_2(t_0,f)$, where $C_2(t_0,f)$ is a constant depending only on the morphism $f$ and $t_0:= x^{-1}(0)$.

\end{enumerate}

\end{lemma}

\begin{proof}By \Cref{propendodr} and \Cref{propdegreebound} there exists a finite extension $\hat{L}/L$ with degree $[\hat{L}:L]\leq C_0(\mu)$ such that we have an injective morphism of algebras $D\hookrightarrow\End_{\hat{L}}((V_{dR})_{\hat{L}},\hat{L})$. By replacing $Y$ by $Y_{\hat{L}}=Y\times_{L}\hat{L}$ we may and do assume for on, for simplicity of notation, that $L=\hat{L}$. 
	
	We may and do assume from now on also that $i=1$, i.e. that $e_1=e_i$, and to simplify our notation we let $e:=e_1$. We also let $W\leq V_{\Q}$ be the first summand isomorphic to $V_1$ that appears in the decomposition of $V_{\Q}$ above. From \Cref{albert}, we know that $F_i=\Q(\tau)$ and by abuse of notation we identify the algebraic number $\tau$ with the element in $\End_{HS}(W)$ and $\End(V_{dR})$ it defines. 
	
	Let $v\in \Sigma_{L,f}$ with $v|p$, $p\in\Z$, be a finite place, let $V_{\et}:= H^n(Y^{ad}_{v,\proet},\Q_p)$, and consider the $p$-adic comparison isomorphism of \cite{scholzerigid1}\begin{center}
		$V_{dR}\otimes_{L} B_{dR}\rightarrow V_{\et}\otimes_{\Q_p}B_{dR}$. 
	\end{center}Since we have assumed the Hodge Conjecture for the endomorphisms at hand holds we get, by compatibility of the cycle class map with the comparison isomorphism, that $\tau$ has a cohomological interpretation as an element of $\End(V_{\et})$, such that the two interpretations, as an element of $\End(V_{dR})$ and as an element of $\End(V_{\et})$, are compatible with the above $p$-adic Hodge comparison isomorphism. In fact, since $\tau$ is defined over $L$ we have that $\tau\in\End(V_{\et})^{G_{L_v}}$.

By the compatibility of $\tau$ with Grothendieck's comparison isomorphism, by virtue of it being a Hodge endomorphism, we get that there exists a subspace $W_{dR}\subset V_{dR}$ with $\dim_{L}W_{dR}=\dim_{\Q} W=:2\nu_1$ such that $\tau|_{W_{dR}}\in \aut(W_{dR})$ and $W_{dR}$ is mapped to $W$ via this comparison isomorphism. 

Also by the compatibility of $\tau$ with the $p$-adic Hodge comparison isomorphism we get a subspace $W_{\et}\leq V_{\et}$ such that $\dim_{\Q_p} W_{\et}=2\nu_1$, $\tau\in \aut(W_{\et})$, and $W_{\et}$ is mapped to $W_{dR}$ by the $p$-adic Hodge comparison isomorphism. 

We note that the minimal polynomial of $\tau$, $P_\tau$, when viewed as an element of $\End_{HS}(W)$ is equal to the minimal polynomial of $\tau$ when viewed as an algebraic number, in other words we have that $P_{\tau}=\prod_{i=1}^{e}  (T-\alpha_i)$, where $\alpha_i=\iota(\tau)$, with $\iota$ varying over the embeddings $\Q(\tau)\hookrightarrow\C$. Furthermore, we have that $P_{\tau}$ is independent of the cohomological interpretation of $\tau$ that is chosen. 
    
Assume now that for the above $v\in \Sigma_{L,f}$ we have that $s$ is $v$-adically close to $0$, so that for $\xi:=x(s)$ we have that $|\xi|_v\leq\min\{1,R_v(\mathcal{G})\}$,. We then have by \Cref{section:goodcovers} that there exists some $1\leq t\leq t_0$ such that $s$ is $v$-adically close to $s_t\in x^{-1}(0)$. Associated to the point $s_t$ we have the $h$ linearly independent functionals $(\hat{\gamma}^{(t)})^{*}_j\in \homm(V_{\et}, \Q_p(n))$ of \Cref{padicrealization}. For now we fix the point $s_t$ above, in essence working with the set of finite places \begin{center}
	$\Sigma(s,t):=\{w\in \Sigma_{L,f}: s$ is $w$-adically close to $s_t\}$,
\end{center}which we assume to be non-empty.

For simplicity, since we have fixed the point $s_t$, we write $\fgj$ for these functionals and consider their interplay with $W_{et}$ to divide our proof into two cases.\\

\textbf{Case $1$:} $\fgj|_{W_{\et}}=0$ for some $1\leq j\leq h$.\\

In this case we may proceed as in the first case of the proof of Theorem $1.16$ of \cite{davidg} to get a linear polynomial $R_0^{(t)}(X_{i,j,t})\in \bar{\Q}[X_{i,j,t}]$ such that $R_0^{(t)}(i_v(y^{(t)}_{i,j}(\xi)))=0$. Note that this polynomial is independent of the choice of $v\in \Sigma(s,t)$ and the relations it defines among the values of the G-functions holds for all such places $v\in \Sigma(s,t)$. \\

\textbf{Case 2:} $\fgj|_{W_{\et}}\neq0$ for all $1\leq j\leq h$.\\

We write $v_1:=\dim_{\Q}W_{\Q}/2$ and consider the family of functionals \begin{equation}\label{eq:case2fun}
\drfj{j}\circ\tau^{q}\text{, where }1\leq j\leq h\text{, and }0\leq q\leq e-1. 
\end{equation}

By Lemma $5.1$ of \cite{davidg} these are given by extension of scalars from elements of $W^{*}_{\et,0}:=\homm(W_{\et},\Q_p(n))^{G_{L_v}}$, see also the proof of Proposition $5.2$ in loc. cit.. Now by Lemma $5.4$ of loc. cit. we know that $\dim_{\Q_p}W^{*}_{\et,0}\leq h^{n,0}(W_{\Q})$. On the other hand since $W_{\Q}$ is a weight $n$, with $n$ odd, polarized Hodge structure, and by assumption a symplectic vector space, we get that $\nu_1\geq h^{n,0}(W_{\Q})$. Following the ideas in loc. cit. we will choose $\nu_1+1$ of these functionals in a way that makes some computations convenient.\\

Now note that as in \eqref{eq:pseudocmdrsplitting} we have a splitting
\begin{equation}\label{eq:derhamsplit}
	W_{dR}\otimes_{L}\hat{F}=W_1\oplus\ldots\oplus W_e,
\end{equation}where $\hat{F}$ is the compositum of $L$ with a Galois closure of the field $F_1=Z(D_1)$, and each $W_i$ is such that $\tau$ acts on it by multiplication by $\alpha_i$. By \Cref{remedy2}, together with compatibility of Hodge endomorphisms with Grothendieck's comparison isomorphism, we know that $\dim_{\hat{F}}W_i=\frac{2\nu_1}{e}$ for all $1\leq i\leq e$. 

Consider now $w_1,\ldots, w_{\nu_1},w_{\nu+1},\ldots, w_{\nu+\nu_1}\in W_{dR,\hat{F}}$ such that the following hold: \begin{enumerate}
	\item $\Span_{\hat{F}}(w_1,\ldots, w_{\nu_1})$ and $\Span_{\hat{F}}(w_{\nu+1},\ldots, w_{\nu+\nu_1})$ are complementary Lagrangian subspaces for the symplectic form on $W_{dR,\hat{F}}$ induced by the polarization,
	
	\item the $w_i$ are ordered as follows: 
	
	\subitem if $D_1$ is of type $I-III$ in \Cref{albert}, for $h_0=\frac{\nu_1}{e}$, we have $w_{(k-1)\cdot h_0+i}$ for $1\leq i\leq h_0$ are a basis of the maximal Lagrangian of $W_k$ for $1\leq k\leq e$, and $w_{\nu+1}\in W_1$ is transversal to the $w_{i}$ for $1\leq i\leq h_0$.
	
	\subitem if $D_1$ is of type $IV$ in \Cref{albert}, for $e_0=\frac{e}{2}$ and $h_0=\frac{\nu_1}{e_0}$, we have $w_{(k-1)\cdot h_0+i}$ for $1\leq i\leq h_0$ are a basis of $W_k$ for $1\leq k\leq e_0$, where we assume that the eigenvalues of $\tau$ on these $W_k$ are such that $\lambda_k\neq \lambda_{k'},\bar{\lambda_{k'}}$ for $1\leq k\neq k'\leq e_0$, and $w_{\nu+(k-1)\cdot h_0+i}$ for $1\leq i\leq h_0$ are a basis of $\bar{W}_k$ for $1\leq k\leq e_0$, which will correspond to the eigenvalue $\lambda_k$.
\end{enumerate}The fact that this is possible follows from the analysis in \Cref{section:involutions} and the aforementioned claim in the proof of \Cref{constpseudocm}, see also the proof of Claim $2$ in the proof of \Cref{constpseudocm}.

Now we extend the above $w_i$ to a full symplectic basis of $V_{dR}\otimes_L \hat{F}$ and write them as $\hat{F}$-linear combinations of the $(\omega_i)_s$, where $\omega_i$ is the basis of sections of the canonical extension $\mathcal{H}$ of $H^n_{DR}(\mathcal{X}/C)$ over $C'$ over some dense open affine subset $U'\subset C'$ chosen as in the discussion in \Cref{section:archim} preceding \Cref{constpseudocm}. We let $\varpi_i$ be the same $\hat{F}$-linear combinations of the $\omega_i$ with the $\varpi_i$ viewed as sections in $\mathcal{H}(U')$ for the same Zariski neighborhood $U'$ of $s_t$ over which the basis $\omega_i$ was defined.

We introduce a family of ``new'' G-functions $\tilde{y}_{i,j}^{(t)}(x)$, which will be $\hat{F}$-linear combinations of our ``original'' G-functions $y_{i,j}^{(t)}(x)$. Writing $\varpi_i=\Sum{j}{}a_{i,j}\omega_j$, consider the power series \begin{center}
	$\tilde{y}_{i,k}(x):=\Sum{j}{}a_{i,j}y_{j,k}(x)$.
\end{center}These will be $\bar{\Q}$-linear combinations of G-functions, hence will also be G-functions themselves. By the construction in \Cref{andresexistence}, see also \Cref{changeofplace} for a summary of this, we can interpret these power series as entries in the relative period matrix, coming from Grothendieck's comparison isomorphism, where we choose the basis $\varpi_i$ for $H^n_{dR}(X/S)$, over some affine open $U$, and the same local frame as in \Cref{changeofplace} that defines the $y_{i,j}(x)$ as integrals. 

Note that the basis $\varpi_i$, and thus the elements $a_{i,j}\in\bar{\Q}$, was chosen independently of any place of the field $L$. In particular, evaluating the functionals $\drfj{j}$ at the basis $w_i$ we find:\begin{center}
	$\drfj{k}(w_i)=\Sum{j}{} i_v(a_{i,j})i_v(y_{j,k}(\xi))= i_v(\tilde{y}_{i,k}(\xi))$,
\end{center}with the last equality following by construction.

We want to reorder the basis $w_i$ of $W$ that was chosen. We record the properties in the following claim.

\begin{claim}[Claim 1]There exists a permutation of the set of indexes $\{1,\ldots,\nu_1,\nu+\nu_1\}$ of the above basis, which denote by $g$, so that if we write $\tilde{w}_i=w_{g(i)}$, the reordered basis $\tilde{w}_i$ will be such that:\\
	
	if $D_1$ is of type $I-III$ in Albert's classification we, for $h_0:=\frac{\nu_1}{e}$, \begin{enumerate}
		\item $\tilde{w}_{(k-1)e+i}=w_{(i-1)h_0+k}$, where $1\leq k\leq h_0$, and $1\leq i\leq e$
		
		\item $\tilde{w}_{\nu+1}=w_{\nu+1}$,
	\end{enumerate}

and if $D_1$ is of type $IV$ in Albert's classification, for $e_0=\frac{e}{2}$ and $h_0=\frac{\nu_1}{e_0}$, we set 
 \begin{enumerate}
	\item $\tilde{w}_{2(k-1)e_0+i}=w_{(i-1)h_0+k}$, where $1\leq k\leq h_0$, and $1\leq i\leq e_0$, and
	
	\item $\tilde{w}_{(2(k-1)+1)e_0+i}=w_{\nu+(i-1)h_0+k}$, where $1\leq k\leq h_0$, and $1\leq i\leq e_0$.
\end{enumerate}
 
\end{claim}
\begin{proof}[Proof of Claim 1] This is straightforward following from the definition of the order on the original basis $w_i$. 
\end{proof}

Now we choose the subfamily of \eqref{eq:case2fun} that we will use. Remember that by our assumption we have $(e-1)h\geq \nu_1\geq h$.\\

\textbf{Sub-case 2.1: $D_1$ is of type $I-III$.} Note that in this case we have $h_0<h$. Indeed $(e-1)h\geq eh_0$ and since $h\geq 2$ we cannot have $e=1$, otherwise $\nu_1=0\geq 2$.

In this case we choose the family \begin{equation}\label{eq:functcasei}
	\{\drfj{j}\circ\tau^q: 1\leq j\leq h_0, 0\leq q\leq e-1\}\cup \{ \drfj{h_0+1}\}. 
\end{equation}We thus have the $\nu_1+1$ functionals we wanted. We further order this reverse-lexicographically on the pair of indices $(j,q)$, in other words the elements of this set are ordered as $\drfj{1},\ldots\drfj{1}\circ\tau^{e-1}$ in this order followed by the $\drfj{2}\circ\tau^q$ in the same order.\\

\textbf{Sub-case 2.2: $D_1$ is of type $IV$.} Note that in this case we have $h_0<2h$ by the same argument. Write $h_0=2\cdot h_1 + \epsilon$, $\epsilon \in \{0,1\}$.

If $\epsilon =0$, we have $h_1<h$ and we choose the family \begin{equation}\label{eq:functcaseiv1}
	\{\drfj{j}\circ\tau^q: 1\leq j\leq h_1, 0\leq q\leq e-1\}\cup \{ \drfj{h_1+1}\}. 
\end{equation}We thus have the $\nu_1+1$ functionals we wanted. We order these as in the previous case.

If $\epsilon =1$ we have $h_1+1\leq h$ and we choose the family  \begin{equation}\label{eq:functcaseiv2}
	\{\drfj{j}\circ\tau^q: 1\leq j\leq h_1, 0\leq q\leq e-1\}\cup \{ \drfj{h_1+1}\circ\tau^q: 0\leq q\leq e_0\}. 
\end{equation}We thus have the $\nu_1+1$ functionals we wanted. Again we consider these ordered as earlier.\\

Letting $w_1,\ldots,w_{\mu}$ be the above symplectic basis of $V_{dR,\hat{F}}$ we now consider the column vectors \begin{equation}\label{eq:vectorscase2}
	\vec{v}_{j,i}:=[\drfj{j}\circ\tau^{i}(w_l)]\in B_{dR}^{\mu},
\end{equation}where $j$, $i$ are as in \eqref{eq:functcasei}, \eqref{eq:functcaseiv1}, and \eqref{eq:functcaseiv2} respectively depending on the case. Consider the matrix $M=[\vec{v}_{j,i}] \in M_{\mu\times (\nu_1+1)}(B_{dR})$ whose columns are the above vectors $\vec{v}_{j,i}$ ordered in the same way as the family of functionals chosen.

The $\nu_1+1$ are linearly dependent, as we saw and thus any $(\nu_1+1)\times(\nu_1+1)$-submatrix $M_0$ of $M$ is such that $\det(M_0)=0$. We will choose the matrix depending on the case we are in.

As noted in the proof of Proposition $5.2$ of \cite{davidg} each of the entries of the vectors $\vec{v}_{j,i}$ are $\hat{F}$-linear combinations of the $v$-adic values $i_v(\tilde{y}_{l,j}^{(t)}(\xi))$ of our new G-functions at $\xi=x(s)$ and thus also of the $v$-adic values of our original G-functions $y_{l,j}^{(t)}$ evaluated at the same $\xi$. Therefore $\det(M_0)=0$ corresponds to a homogeneous polynomial $R_1^{(t)}\in \bar{\Q}[X_{i,j,t}]$, of degree $\nu_1+1$, such that $R_1^{(t)}(i_v(y^{(t)}_{i,j}(\xi)))=0$. Furthermore, as in case $1$ above, this polynomial is independent of the choice of place $v\in \Sigma(s,t)$ and once again the relation it defines among the values $i_v(y^{(t)}_{i,j}(\xi))$ holds for all places in the set $\Sigma(s,t)$.

Now let us choose the rows of the matrix $M_0$ in each of the cases as follows:
\begin{enumerate}
	\item for the family \eqref{eq:functcasei} we choose the first $\nu_1+1$ rows of the matrix $M$, and we order these as in the claim above.
	
	\item for the family \eqref{eq:functcaseiv1} we choose the rows corresponding to the re-ordering $\tilde{w}_j$, ordered as in the claim, given by $\tilde{w}_{2(k-1)e_0+i}$ and $\tilde{w}_{(2(k-1)+1)e_0+i}$, where $1\leq k\leq h_1$, and the last row corresponding to $\tilde{w}_{2(h_1)e_0+1}$, and
	
	\item for the family \eqref{eq:functcaseiv2} we choose the rows corresponding to the re-ordering $\tilde{w}_j$, ordered as in the claim, given by $\tilde{w}_{2(k-1)e_0+i}$ and $\tilde{w}_{(2(k-1)+1)e_0+i}$, where $1\leq k\leq h_1$ and $1\leq i\leq e_0$,  followed by $\tilde{w}_{2(h_1)e_0+i}$, where $1\leq i\leq e_0$, and the last row is given by $\tilde{w}_{(2(h_1)+1)e_0+1}$.

\end{enumerate}Note that the re-ordering we did does not alter the relation $\det(M_0)=0$, since we are changing signs by permuting the rows.\\

Let us now abandon the assumption that $L=\hat{L}$. Then we have to consider the base change $V_{dR}\otimes_{L}\hat{L}$ and repeat the above process. 

In order to define $R_{fin}$ we combine the above two cases. Since by assumption we have that \begin{center}
	$\Sigma(s):=\{v\in \Sigma_{L,f}: s$ is $v$-adically close to $0\}\neq \emptyset$,
\end{center}we may write $\Sigma(s)$ as a disjoint union $\Sigma(s)=\Sigma_1\sqcup \Sigma_2$ where $\Sigma_1$ is the disjoint union of the sets $\Sigma(s,t)$, where $1\leq t\leq t_0$, for those $s_t$ for which for some $v\in \Sigma(s,t)$, and hence for all such $v$, case $1$ above holds, and similarly for $\Sigma_2$. 

We then define \begin{equation}\label{eq:finiteplacespoly}
	R_{s,\fin}=\prod_{\Sigma(s,t)\subset \Sigma_1}^{} R^{(t)}_{0}(X_{i,j,t})\cdot \prod_{\Sigma(s,t)\subset \Sigma_2}^{} R^{(t)}_{1}(X_{i,j,t}).
\end{equation}We note that this polynomial is indeed homogeneous of degree $\deg(R_{s,\fin})\leq t_0(\nu_1+1)$  and so by our exposition so far we get that  conclusions \ref{condition1} and \ref{condition3} of our proposition hold. 

We are thus left with checking that the relation defined by $R_{s,\fin}$ among the values of our G-functions at $\xi$ is non-trivial, i.e. it does not hold on the functional level.\\

\textbf{Non-Triviality of the relation:} Again for simplicity of notation, we return to the case $\hat{L}=L$. 

By definition of $R_{s,\fin}$ as a product of the polynomials $R_0^{(t)}$ and $R_1^{(t)}$, we get that if $R_{s,\fin}(y^{(t)}_{i,j}(x))=0$ on the functional level we must have that there exists a point $s_t\in x^{-1}(0)$, $1\leq t\leq t_0$, such that either $R_0^{(t)}(y^{(t)}_{i,j})=0$, in case $\Sigma(s,t)\subset \Sigma_1$ in the above notation, or $R_1^{(t)}(y^{(t)}_{i,j})=0$, in case $\Sigma(s,t)\subset \Sigma_2$ in the above notation.

Note that in both of those cases we know that these polynomials define relations among the G-functions in the sub-family $\mathcal{G}^{(t)}$. But in this case we know that the functional relations are those corresponding to polynomials in the ideal $I_t$ of \Cref{desidealtriv}. Thus the former of the two cases, i.e. when $R_0^{(t)}(y^{(t)})=0$, is easily discarded since the polynomial $R_0^{(t)}$ is linear, while the polynomials defining the $I$ are homogeneous of degree $2$. 

Therefore, we are reduced to showing that $R_1^{(t)}\notin I_t$ for the polynomial $R_1^{(t)}$ constructed in case $2$ above. By our earlier remarks we may and do fix the point $s_t\in x^{-1}(0)$. Therefore, for notational simplicity, from now on we drop any mention of $t$ and write simply $y_{i,j}$ for our G-functions and $R:=R_1^{(t)}$ for our polynomial.\\

Since the $y_{i,j}(x)$ are $\hat{F}$-linear combination of the $\tilde{y}_{i,j}(x))$, and vice versa, any relation among the former induces one for the latter and vice versa, whether that is on the functional level or among the values of these functions at a point. Note furthermore that by construction of the basis $\varpi_i$ and the general exposition of \Cref{section:trivialrelations}, the trivial relations among the $\tilde{y}_{i,j}(x)$ are also defined by the same ideal as in \Cref{desidealtriv}. 

Therefore, writing $R'\in \bar{\Q}[Y_{i,j}]_{\underset{1\leq i\leq h}{1\leq j\leq 2\nu}}$ for the polynomial induced by $\det(M_0)=0$, in case $2$ above, on the values $i_v(\tilde{y}_{i,j}(\xi))$ we are reduced to showing that \begin{equation}
	R'\notin \langle P_{j,j'}: 1\leq j\leq j'\leq h\rangle,
\end{equation}where $P_{j,j'}=\sum_{l=1}^{\nu} Y_{l,j'}Y_{\nu+l,j}-\sum_{l=1}^{\nu} Y_{l,j}Y_{\nu+l,j'}$. From now on we assume the contrary.

By choice of our bases $w_i$ and $\varpi_i$ we have that the $l$-th row of $M_0$ will be of the form $[\alpha_{g(l)}^i\tilde{y}_{l,j}(\xi)]$ where $g(l)\in \{1,\ldots, e\}$ is such that $w_l\in W_l$ in the decomposition of \eqref{eq:derhamsplit}. 

Replacing the $\tilde{y}_{i,j}(\xi)$ by indeterminates $Y_{i,j}$ we get a matrix, that we also denote, by abuse of notation, by $M_0\in M_{\nu_1+1} (\bar{\Q}[Y_{s,t}])$, so that $\det(M_0)=R'$. By our earlier assumption we have that there exist $Q_{j,j'}\in \bar{\Q}[Y_{s,t}]$ such that \begin{equation}\label{eq:member}
	R'=\sum_{1\leq j\leq j'\leq h}^{} Q_{j,j'} P_{j,j'}.
\end{equation}

Roughly put, in order to get a contradiction to \eqref{eq:member}, we forcefully ``block-upper-triangularize'' the matrix $M_0$ to more easily compute $R'$, all the while taking advantage of our choices so far. Note that after the re-ordering, corresponding to the basis $\tilde{w}$, the elements of the $i$-th row of $M_0$ will look like $\lambda_{t}^q \cdot Y_{l_i,k}$, in other words we write $l_i$ for the original $w_{l_i}$ corresponding to the vector $\tilde{w}_i$.

Now consider the reduction map \begin{center}
	$\phi:\bar{\Q}[Y_{s,t}]\rightarrow \bar{\Q}[Y_{s,t}:(s,t)\in \Sigma_0]$,
\end{center}where $\Sigma_0= \{(i,j):1\leq i\leq \nu_1$, $\nu+1\leq i\leq \nu+\nu_1$, $1\leq j\leq h\}\backslash(\Sigma_1)$, with 
$\Sigma_1= \Bigcup{j=0}{h_t}\{(l_i,j): i>j\cdot e  \}$, where $h_t=h_0$ or $h_1$ depending on the case we are in, so that $\ker \phi  = \langle Y_{s,t}:(s,t)\notin \Sigma_0\rangle$. 

In other words, we ``kill'' the $Y_{i,j}$ that do not appear in $M_0$ and enough of those that appear so that $\phi(M_0)$ is block-upper-triangular. 

Write $B_j$ for the $j$-th block that appears in the diagonal of the block-upper triangular matrix $\phi(M_0)$. By construction we have $B_j\in M_{e} (\bar{\Q}[Y_{l,j}: 1\leq l\leq 2\nu])$ for $1\leq j\leq h_t$, and $B_{h_t}\in M_{1} (\bar{\Q}[Y_{l,j}: 1\leq l\leq 2\nu])$ in case we have \eqref{eq:functcasei}, or \eqref{eq:functcaseiv1}, while $B_{h_t}\in M_{e_0+1} (\bar{\Q}[Y_{l,j}: 1\leq l\leq 2\nu])$ in the final case \eqref{eq:functcaseiv2}. Note in particular that the  entries of the matrix $B_j$ all correspond to the G-functions $\tilde{y}_{l,j}$ that appear as a single column of the relative period matrix. 
    Also by construction we have that $B_j=(\lambda_t^{q-1}Y_{l_i,j})_{\underset{1\leq q\leq \alpha'}{1\leq i\leq \alpha' }}$, with $\alpha'\in \{e,e_0+1,1\}$, depending on the case, where $i$ is running through the rows of the matrix and $q$ through its columns, and $\lambda_t$ is the eigenvalue by which $\tau$ acts on $W_{t}$ with no $\lambda_t$ appearing twice here by virtue of the construction of basis $\tilde{w}_j$ and the choice of our matrix.

In particular, note that $\phi(R')=\prod_{j=1}^{h} \det(B_j)$. But $B_j=\diag(Y_{l_i,j})\cdot V(\lambda_1,\ldots,\lambda_{\alpha'})$, where $\diag(Y_{l_i,j})$ is a diagonal matrix with entries $Y_{l_i,j}$ and $V(\lambda_1,\ldots, \lambda_{\alpha'})$ is some Vandermonde matrix, so that $\det V(\lambda_1,\ldots, \lambda_{\alpha'})= \prod_{}^{} (\lambda_i-\lambda_{i'})=V_0\in \bar{\Q}$. Note, furthermore that $V_0\neq 0$ by virtue of the definition of the basis $\tilde{w}$ and since $\alpha\leq e$ by construction in any case. 

Thus by acting by $\phi$ on \eqref{eq:member} we get \begin{equation}\label{eq:member2}
	\phi(R')= \sum_{1\leq j<j'\leq h}\phi(Q_{j,j'}) \cdot \phi(P_{j,j'}).
\end{equation}From the above remarks we know $\phi(R')=A\cdot \prod_{}^{} Y_{l_i,j}$ is a product of $\nu_1+1$ of the $Y_{s,t}$ with $(s,t)\in \Sigma_0$ with some $A\in \bar{\Q}$ with $A\neq 0$. 

The contradiction to \eqref{eq:member} thus follows from the following claim.

    \begin{claim}[Claim 2]The ideal $I_0:=\langle\phi(P_{j,j}): 1\leq j<j'\leq h\rangle$ is prime and $Y_{s,t}\notin I_0$ for all $s$, $t$. \end{claim}
\begin{proof}[Proof of Claim 2]By definition of $\phi$, or rather its kernel, we have that
	\[	\bar{P}_{j,j'}:= \phi(P_{j,j'}) =\sideset{}{'}\sum_{l=1}^{\nu_1}Y_{l,j'} Y_{\nu+l,j}-\sideset{}{'}\sum_{l=1}^{\nu_1} Y_{l,j} Y_{\nu+l,j'},\tag{67}\label{eq:phipjj}
	\]where $\Sigma'$ denotes the sums in the above range accompanied with the restriction that $(l,j)$, $(\nu+l,j')\in \Sigma_0$.
	\setcounter{equation}{67}

By assumption, we have that $h\geq 2$, and hence also by assumption $\nu_1\geq 2 $. Note also that $j<j'$ and, by construction, if $Y_{l,j}\notin \Sigma_1$ then we get $Y_{l,j'}\notin \Sigma_1$.

We now claim that there are at least two nonzero summands in the two sums in \eqref{eq:phipjj}. If this were not the case, then for all $l$ we would have, from the above remark, that either $Y_{l,j}$ or $Y_{\nu+l,j}\in \Sigma_1$, which implies that $|(\Sigma_1)\cap (\N\times\{j\})|\geq \nu_1$. But $|(\Sigma_1)\cap (\N\times\{j\})|\leq \nu_1+1-e\leq \nu_1-1$ by construction and since as noted earlier $e\geq 2$ here.

Consider the order on the set $\{Y_{i,j}:(i,j)\in \Sigma_0\}$ given by the reverse lexicographic order on the pairs of indices $(i,j)\in \Sigma_0$, so that $(i,j)>(i',j')$ if $i<i'$ or $i=i'$ and $j<j'$. Considering the obvious graded lexicographic order on $\bar{\Q}[Y_{i,j}:(i,j)\in \Sigma_0]$, see \cite{clo} $\S$ $2.2$, induced by the above order on the $Y_{i,j}$, we get that the leading term $\LT(\bar{P}_{j,j'})$ of $\bar{P}_{j,j'}$ is given by \begin{equation}
	F_{j,j'}:=\LT(\bar{P}_{j,j'})=Y_{l(j),j}Y_{l(j'),j'},
\end{equation}where $l(j):=\min\{l:(l,j)\in \Sigma_0\}$ and $l(j')$ is given by $|l(j)-l(j')|=\nu$.

Let now $P$, $Q\in \bar{\Q}[Y_{i,j}:(i,j)\in \Sigma_0]$ be such that $P\cdot Q\in I_0$.  By the division algorithm in $\bar{\Q}[Y_{i,j}:(i,j)\in \Sigma_0]$, see \cite{clo} $\S$ $2.3$, we get that $P=\Sigma Q^{(1)}_{j,j'} \bar{P}_{j,j'}+R_1$ and $Q=\Sigma Q^{(2)}_{j,j'} \bar{P}_{j,j'}+R_2$ with $R_i$ being either $0$ or $\bar{\Q}$-linear combinations of the $F_{j,j'}$.
	
	From the above we get $PQ-R_1\cdot R_2\in I_0$ and, since $PQ\in I_0$, we have $R_1\cdot R_2\in I_0$. Write $R_1=\Sigma a_{j,j'}F_{j,j'}$ and $R_2=\Sigma b_{j,j'}F_{j,j'}$. Then there exist $Q^{(3)}_{j,j'}\in \bar{\Q}[Y_{i,j}:(i,j)\in \Sigma_0]$ such that \begin{equation}\label{eq:member3}
	(	\Sigma a_{j,j'}F_{j,j'}	)(\Sigma a_{j,j'}F_{j,j'})=\Sigma Q^{(3)}_{j,j'} \bar{P}_{j,j'}.		
	\end{equation}Since both sides of \eqref{eq:member3} are homogeneous degree $2$ polynomials we have $Q^{(3)}_{j,j'}\in \bar{\Q}$ for all $j<j'$. Assume $Q^{(3)}_{j,j'}\neq 0$ for some pair $(j,j')$. 

By our earlier remark, there exists a monomial of the form $Y_{l,j}Y_{l',j'}\neq F_{j,j'}$ that appears in the expression of $\bar{P}_{j,j'}$ as in \eqref{eq:phipjj} and by definition will not appear in any other $\bar{P}_{i,i'}$. So this monomial will have a non-zero contribution to the sum on the right hand side of \eqref{eq:member3}. But this monomial will not appear on the left hand side of \eqref{eq:member3}. Therefore $Q^{(3)}_{j,j'}=0$ for all $j<j'$ and thus either $R_1=0$ or $R_2=0$. So $I_0$ is indeed a prime ideal.

Finally, it is obvious that $Y_{s,t}\notin I_0$ for all $s,t$ by a similar argument as above.\end{proof}

This concludes the proof of \Cref{nonarchirel}.
\end{proof}

The necessity of the condition $(e_i-1)\cdot h\geq \nu_i \geq h$  of the \Cref{nonarchirel}, is evident in the proof. It can be reread as saying that we have a ``large'' Hodge substructure, relative to the number of G-functions, whose algebra of Hodge endomorphisms also has a large center. We return to this in \Cref{section:examples}.

A natural generalization of \Cref{nonarchirel} is the following proposition.

\begin{prop}\label{nonarchilcenter}Let $V_{\Q}$, $V_{dR}$, $D_i$, $m_i$, $F_i$, $e_i$ , and $L/\Q$ be as  above. We assume that the Hodge Conjecture holds for Hodge Endomorphism of $V_{\Q}$ and, as usual, that $h\geq 2$. If for some subset $J\subset \{1,\ldots,r\}$, of the number of non-isomorphic components of the decomposition of $V_{\Q}$ as in  \eqref{eq:decomprepeat}, there exist positive integers $m_i'\leq m_i$, where $i\in J$, such that \begin{equation}\label{eq:nonarchcondtion}
		(h(\Sum{i\in J}{} (e_i-1)m'_i))\geq (\Sum{i\in J}{}m'_i \nu_i)\geq h
	\end{equation}and $\{v\in \Sigma_{L,f}: s$ is $v$-adically close to $0\}\neq \emptyset$, then there exists a homogeneous polynomial \begin{center}
			$R_{s,\fin}\in \bar{\Q}[X_{i,j,t}; 1\leq i\leq \mu$, $1\leq h\leq h$, $1\leq t\leq t_0]$ such that
		\end{center}satisfying the same conditions as those in \Cref{nonarchirel}.
\end{prop}

In other words, \Cref{nonarchirel} corresponds to the case $|J|=1$ in the above.

\begin{proof}The proof largely follows that of \Cref{nonarchirel}. Assume without loss of generality that $J=\{1,\ldots, r_0\}$ where $r_0\leq r$. We start, as we did in that proof by assuming that $L=\hat{L}$ for the field $\hat{L}$ defined in \Cref{propendodr} to simplify our exposition. 

Consider the summand $W:=V_1^{m'_1}\oplus\ldots\oplus V_{r_0}^{m'_{r_0}}$ of $V_{\Q}$. We then write, following the same arguments as those in the proof of \Cref{nonarchirel}, $W_{dR}:= W_{1,dR}^{m'_1}\oplus\ldots\oplus W_{r_0,dR}^{m'_{r_0}}$, and $W_{\et}:= W_{1,\et}^{m'_1}\oplus\ldots\oplus W_{r_0,\et}^{m'_{r_0}}$, for the subspaces of $V_{dR}$ and $V_{\et}$ respectively that are mapped to $W$ under the comparison isomorphisms, either those of Grothendieck or Faltings respectively.
	
	We further fix an order for the components of the summands $V_i^{m'_i}$, respectively of the $W_{i,dR}^{m'_i}$ and $W_{i,\et}^{m'_i}$, denoted by $V_i^{(j)}$, respectively $W_{i,dR}^{(j)}$ and $W_{i,\et}^{(j)}$, so that these also are mapped to one another by the above comparison isomorphisms. 
	
	For each of the centers $F_i$ that appear, we may write $F_i=\Q(\tau_i)$ by \Cref{albert}. We write $\tau_{i,(j)}$ for all of the cohomological interpretations of $\tau_i$ considered as an endomorphism of $V_i^{(j)}$, or $W_{i,dR}^{(j)}$, or $W_{i,\et}^{(j)}$.
	
	As in the proof of \Cref{nonarchirel} we fix the point $s_t$ which $s$ is $v$-adically close to, and work with the non-empty set of finite places \begin{center}
		$\Sigma(s,t):=\{w\in \Sigma_{L,f}: s$ is $w$-adically close to $s_t\}$.
	\end{center}
	
	Consider now the functionals $\fgj$ of \Cref{padicrealization} acting on $W_{\et}$. As in the proof of \Cref{nonarchirel} we have two possible cases.\\
	
\textbf{Case $1$:} $\fgj|_{W^{(k)}_{i,\et}}=0$ for some $j$, $i$, $k$ as above.\\

This is dealt with as before and gives rise to a a linear polynomial $R_0^{(t)}(X_{i,j,t})\in \bar{\Q}[X_{i,j,t}]$.  The same arguments as in the previous proof show that this will satisfy the same conditions as the polynomial constructed there.\\

\textbf{Case 2:} $\fgj|_{W^{(k)}_{i,\et}}\neq 0$ for all $j$, $i$, $k$ as above.\\

Let us write $\nu_0:= \Sum{i=1}{r_0}m'_i\nu_i$, so that $h\leq \nu_0$ by assumption. 

Without loss of generality assume that $e_1\geq 2$, note that at least one of the $e_i$ has to be $\geq 2$ by the assumption above.

We then consider the family of non-zero functionals\begin{equation}\label{eq:functionalshard}
\Gamma(W):=\{\drfj{j}\circ \tau_{i,(k_i)}^{q_i}:1\leq j\leq h, 0\leq q_i\leq e_i-1, 1\leq k_i\leq m'_i \},
\end{equation}note here that $\drfj{j}\tau_{i,k_i}^0$ will just be the restriction of $\drfj{j}$ to the component $W^{(k_i)}_{i,dR}$ and not $\drfj{j}$ on the whole of $W_{dR}$!

The idea remains the same as earlier. Choose $\nu_0+1$ such functionals, which by the same argument as in \Cref{nonarchirel} will be given by extensions of scalars from vectors in $W^{*}_{\et,0}:=\homm(W_{\et},\Q_p(n))^{G_{L_v}}$. As in the proof of the previous lemma we get that $\dim_{\Q_p}W^{*}_{\et,0}\leq h^{n,0}(W_{\Q})\leq \nu_0$, and thus any $\nu_0+1$ vectors in the family will be linearly dependent.

As before we have a splitting given by the action of $\tau_{i,(j)}$
\begin{equation}\label{eq:derhamsplitgen}
	W^{(j)}_{i,dR}\otimes_{L}\hat{F}=W^{(j)}_{i,1}\oplus\ldots\oplus W^{(j)}_{i,e_i},
\end{equation}where $\hat{F}$ is the compositum of $L$ with the Galois closure of the compositum of the fields $F_i=Z(D_i)$, and each $W_i$ is such that $\tau_{i,(j)}$ acts on it by multiplication by $\lambda_{i,t}$, where $\lambda_{i,t}$ for $1\leq t\leq e_i$ are the different embeddings of $\tau_i$ in $\C$. As before, by \Cref{remedy2}, together with compatibility of Hodge endomorphisms with Grothendieck's comparison isomorphism, we know that $\dim_{\hat{F}}W^{(j)}_{i,t}=\frac{2\nu_i}{e_i}$ for all $j$ and all $1\leq t\leq e_i$.

Let us write $h_i$ to be either $\frac{\nu_i}{e_i}$ if $D_i$ is of type $I-III$ or $\frac{\nu_i}{e_{i,0}}$ if $D_i$ is of type $IV$ in \Cref{albert}, where $e_{i,0}=e_i/2$ in this case. Consider $w_1,\ldots, w_{\nu_0},w_{\nu+1},\ldots, w_{\nu+\nu_0}\in W_{dR,\hat{F}}$ such that the following hold: \begin{enumerate}
	\item $\Span_{\hat{F}}(w_1,\ldots, w_{\nu_0})$ and $\Span_{\hat{F}}(w_{\nu+1},\ldots, w_{\nu+\nu_0})$ have the same property as before, and 
	
	\item we may write $w_l=w^{(j(i))}_{(k-1)h_i+i'}$ where $w^{(j(i))}_{(k-1)h_i+i'}$ and $w^{(j(i))}_{(k-1)h_i+i'}$ give bases of the respective subspaces of $W^{(j(i))}_{i,dR}$ with ordering as that chosen in the proof of \Cref{nonarchirel}. The ordering of each such block will depend on the type of the algebra $D_i$.
\end{enumerate}In other words the basis is subordinate to the splitting $W_{dR}=W_{1,dR}^{m'_1}\oplus\ldots\oplus W_{r_0,dR}^{m'_{r_0}}$, the splitting \eqref{eq:derhamsplitgen}, and the ordering of the $W^{(j)}_{i,dR}$. Again this is possible by \Cref{section:involutions} and the aforementioned claim in the proof of \Cref{constpseudocm}. 

As in the proof of the previous lemma, extend the above $w_i$ to a full symplectic basis of $V_{dR}\otimes_L \hat{F}$ and write them as $\hat{F}$-linear combinations of the $(\omega_i)_s$. Again we let $\varpi_i$ be the of sections $\mathcal{H}(U')$ with the same definition as in the previous proof. We also write $\tilde{y}_{i,j}^{(t)}(x)$ for the ``new'' G-functions we obtain in the same way as in the previous proof.

We now want to order a set of $\nu+0_1$ vectors of the basis in a convenient way as we did in Claim $1$ of the previous proof. Let us write $\mathcal{B}_{k,i,j}$, where $1\leq k\leq h_i+\epsilon_i$, where $\epsilon_i\in \{0,1\}$, for the sets of vectors in the basis of $W^{(j)}_{i,dR}$  implied by Claim $1$ in the previous proof. In particular if $D_1$ is of type $I-III$ or of type $IV$ with $\frac{\nu_i}{e_{i,0}}=2h_i$, i.e. the case $\epsilon_i=0$, these sets are of size $e_i$ for $k\leq h_i$, while for type $IV$ with $\frac{\nu_i}{e_{i,0}}=2h_i+1$ these sets will be of size $e_i$ for $k\leq h_i$ and size $e_{i,0}$ for $k=h_{i}+1$.

  In more detail, the sets of size $e_i$, for $1\leq k\leq h_i$, in the case of type $I-III$ or of type $IV$ with $\frac{\nu_i}{e_{i,0}}=2h_i$, will look as follows:
\begin{center}
	$\mathcal{B}_{k,i,j}=\{\tilde{w}^{(j)}_{i,(k-1)e_i+l}:1\leq l\leq e_i\}$ if $D_i$ is of type $I-III$ or 
	
$\mathcal{B}_{k,i,j}=\{\tilde{w}^{(j)}_{i,2(k-1)e_{i,0}+l}, \tilde{w}^{(j)}_{i,(2(k-1)+1)e_{i,0}+l} :1\leq l\leq e_{i,0} \}$ if $D_i$ is of type $IV$.
\end{center}Similarly for the sets $\mathcal{B}_{i,h_i+1,j }$ in the type $IV$ case when $\frac{\nu_i}{e_{i,0}}=2h_i+1$.

We have $|\Bigcup{k,i,j}{}\mathcal{B}_{k,i,j}|=\nu_0$ by construction. To choose the $(\nu_0+1)$-th vector we proceed as earlier. In particular we may assume without loss of generality that $e_1>1$ and then choose this vector from those of $W^{(1)}_{1,dR}$ from those not already chosen in accordance to the type of $D_1$ in Albert's classification and whether $\epsilon_1=0$ or $1$, exactly as we did in the previous proof. For this new vector we either create a new block $\mathcal{B}_{i,h_i+1,j }$ if $e_1=0$, or add it last to the existing block of the same name. Finally, we order these blocks of vectors lexicographically on the triple $(k,i,j)$ and write $\mathcal{B}$ for this choice of basis elements.
	
Now consider the following families of functionals:\begin{center}
	 $\mathcal{F}_{k,i,j}:= \{\drfj{k}\circ\tau_{i,j}^q: 1\leq q\leq e_i \}$, or $F_{h_i+1,i,j}= \{ \drfj{h_i+1}\circ\tau_{i,j}^q: 1\leq q\leq e_{i,0}-1\}$
\end{center}where $k$ varies in the same sets as above depending on the algebra $D_i$, as does the existence of $F_{h_i+1,i,j}$. Note that these are just the $\nu_i$ first functionals chosen in the previous proof in each case.
	
We keep the same order as before for the elements of each $\mathcal{F}_{k,i,j}$, i.e. we order them on $q$, and order the families themselves lexicographicaly on the triples $(k,i,j)$. Again there will be $\nu_0$ elements in these sets in total. This all makes sense since for all $h_i$ we have $h_i<h$ by our  hypothesis that $(h(\Sum{i\in J}{} (e_i-1)m'_i))\geq \nu_0$ and the definition of the $h_i$. Finally, by our assumption that $e_1\geq 2$ earlier we can take an $(n_0+1)$-th functional of the form $\gamma_{h_1+\epsilon_1}\circ\tau_{1,(1)}^q$. Same as the $(\nu+1)$-th vector chosen for the blocks $\mathcal{B}_{k,i,j}$ we either add this to the block for $\mathcal{F}_{h_1+1,1,1}$ or create such a block with just one element if did not exist already. Finally, write $\mathcal{F}$ for this ordered family of size $\nu_0+1$.\\

As in the previous proof, consider the column vectors \begin{equation}\label{eq:vectorsgencase2}
	\vec{v}_{j,i,q,k}:=[\drfj{j}\circ\tau_{i,(k)}^{q}(w_l)]\in B_{dR}^{\mu},
\end{equation}where $\drfj{j}\circ\tau_{i,(k)}^{q}\in \mathcal{F}$.  As in the previous proof we then get a matrix $M=[\vec{v}_{j,i}] \in M_{\mu\times (\nu_1+1)}(B_{dR})$ with columns the above vectors, of rank $\leq \nu_0$

Consider the $(\nu_0+1)\times(\nu_0+1)$-submatrix $M_0$ of $M$ whose rows correspond to the elements of $\mathcal{B}$. This, as we have seen, gives a homogeneous polynomial $R_1^{(t)}\in \bar{\Q}[X_{i,j,t}]$, of degree $\nu_0+1$, satisfying the exact same, initial, properties as the polynomial denoted by the same notation in the previous proof.\\

Again, abandoning the assumption $L=\hat{L}$ we may define $R_{s,\fin}$ by the same formula as in \eqref{eq:finiteplacespoly}. Once again, this will satisfy by construction all the properties we want other than non-triviality which we have to check by hand. Once again arguing exactly as in the proof of the previous lemma we are reduced to checking the non-triviality for the polynomials induced by those constructed in case $2$ for the values of the ``new'' G-functions $\tilde{y}^{(t)}_{i,j}$ for a single $t$. We write $R'$ for those and thus check the following:\\

\textbf{Non-triviality of case $2$:} Assume that \begin{equation}
	R'\in \langle P_{j,j'}: 1\leq j\leq j'\leq h\rangle,
\end{equation}where $P_{j,j'}=\sum_{l=1}^{\nu} Y_{l,j'}Y_{\nu+l,j}-\sum_{l=1}^{\nu} Y_{l,j}Y_{\nu+l,j'}$ as before.

The construction giving our families $\mathcal{B}$ and $\mathcal{F}$ shows that the $l$-th row of $M_0$ will be of the form $[\delta_{g(l)}^{q}\tilde{y}_{l,j}(\xi)]$ where $\delta_{g(l)}$ is either equal to $\lambda_{g(l)}$ for some $g(l)\in \{1,\ldots, e_i\}$ if $w_l\in W^{(j')}_{i,g(l)}$ in the decomposition \eqref{eq:derhamsplitgen} and the element in question is also in the column given by the functional $\drfj{j}\circ\tau_{i,(j')}^q$ or $\delta_{g(l)}=0$ otherwise. For the last comment note that for $j\neq j'$ by construction we have that $\tau_{i,(j)}$ restricts to $0$ in $W^{(j')}_{i}$, note also our earlier comment that $\drfj{k}\circ\tau_{i,(j)}^{0}$ is just the restriction of $\drfj{k}$ on $W^{(j)}_{i,dR}$. 

The idea of the rest of the proof is the same as that of the previous one. Take a carefully defined reduction of polynomial rings by killing coordinates to block-upper-triangularize the matrix $M_0$ and then compute the determinants of the blocks. The nice thing about our ordering of $\mathcal{F}$ and $\mathcal{B}$ is that now these blocks are themselves ``block-lower-triangular'' and thus their determinants are easy to compute!

Before seeing the proof let us illustrate this by an example. Let $\mathcal{F}_1$ be the union of the $\mathcal{F}_{1,i,j}$ and let $N_0:=\Sigma e_im'_i$ and $M_0=\Sigma_{m'_i}$. Consider the $(N_0)\times(N_0)$-matrix that consists of the first $(N_0)$ rows and the first $(N_0)$ columns of $M_0$ after this re-ordering of the  rows. This matrix, or rather its functional analog by replacing $\tilde{y}_{l,j}(\xi)$ by $Y_{l,j}$ as in the previous proof, will look as follows:

\begin{center}
	$\mathcal{A}_1=\begin{pmatrix}A_1 &0   &0 &\cdots &0\\
		&A_2&0&\cdots&0\\
		&  & \cdots &    &0\\
		&&&&A_{M_0}
	\end{pmatrix}$,
\end{center}where $A_1=\diag(Y_{l,1})\cdot V(\lambda_{1,1},\ldots, \lambda_{1,e_1})$, $A_2= \diag(Y_{l,1})V(\lambda_{1,1}\ldots,\lambda_{1,e_i})$, where $i=1$ if $m'_1\geq 2$ and $i=2$ otherwise, and \begin{center}
$A_{M_0}= \diag(Y_{l,1})V(\lambda_{r_0,1}\ldots,\lambda_{r_0,e_{r_0}})$.
\end{center} In particular the determinant of this will be of the form $\beta_1\cdot\Pi Y_{l,1}^{p_l}$ where $\beta_1\in \bar{\Q}^{*}$ and $p_l$ are some non-negative integers, not all of whom are zero. The fact that $\beta_1\neq0$ follows from the choice of the basis of the block $\mathcal{B}_{k,i,j}$ as in the previous proof! Indeed, each of the different eigenvalues appears since none of the vectors in this block come from the same $W^{j}_{i,k}$!\\

To reduce to this we just have to define the reduction homomorphism. Then a variant of Claim $2$ of the previous proof kicks in to take care of non-triviality. 

Let us denote by $\mathcal{F}_k:=\cup\mathcal{F}_{k,i,j}$ and let $f_k=|\mathcal{F}_k|$. Let us also write $Y_{l_i,j}$ for the elements that appear as the indeterminates in the entries of the $i$-th row of $M_0$. In other words $w_{l_i}$ is the $i$-th element of $\mathcal{B}$. 

With that in mind we define \begin{center}
	$\phi:\bar{\Q}[Y_{s,t}]\rightarrow \bar{\Q}[Y_{s,t}:(s,t)\in \Sigma_0]$,
\end{center}where $\Sigma_0= \{(i,j):1\leq i\leq \nu_0$, $\nu+1\leq i\leq \nu+\nu_0$, $1\leq j\leq h\}\backslash(\Sigma_1)$, with 
$\Sigma_1= \Bigcup{j=1}{h_f}\{(l_i,k): i>f_1+\cdots+ f_k  \}$, where $h_f\leq \max{h_i}+1\leq h$ the final ''column of G-functions'' that will appear. In other words $\ker \phi  = \langle Y_{s,t}:(s,t)\notin \Sigma_0\rangle$.

Note that the $(\nu_0+1)\times (f_1)$ matrix of the first $f_1$ columns of $M_0$ will be of the form \begin{center}
	$\begin{pmatrix}
		\mathcal{A}_1\\
		0\\
	\end{pmatrix}$,
\end{center}where $\mathcal{A}_1$ is the above matrix. Similarly for the other columns corresponding to the other $\mathcal{F}_k$. Notice also that these matrices will, as in the previous proof, only have indeterminates of the form $Y_{i,k}$ with $k$ fixed, i.e. coming from only one original column in the relative period matrix of the variation of Hodge structures.

We are thus reduced to checking, as before, the following:
\begin{claim}The ideal $I_0:=\langle\phi(P_{j,j}): 1\leq j<j'\leq h\rangle$ is prime and $Y_{s,t}\notin J$ for all $s$, $t$. \end{claim}
\begin{proof}
	As before, let us write
	\[	\bar{P}_{j,j'}:= \phi(P_{j,j'}) =\sideset{}{'}\sum_{l=1}^{\nu_0}Y_{l,j'} Y_{\nu+l,j}-\sideset{}{'}\sum_{l=1}^{\nu_0} Y_{l,j} Y_{\nu+l,j'},\tag{75}\label{eq:phipjj2}
	\]where $\Sigma'$ denotes the extra restriction $(l,j)$, $(\nu+l,j')\in \Sigma_0$.
		\setcounter{equation}{75}
	
    As before we note that $h\geq 2$, and $\nu_0\geq 2 $ by assumption. Once again if for some $j<j'$ we have $Y_{l,j}\notin \Sigma_1$ then we get $Y_{l,j'}\notin \Sigma_1$ by construction. This is by the fact that $\phi(M_0)$ was constructed to be block-upper-triangular.
	
	Again there will be at least two non-zero summands in the two sums in \eqref{eq:phipjj2}. Otherwise as before we would get $|(\Sigma_1)\cap (\N\times\{j\})|\geq \nu_0$. But $|(\Sigma_1)\cap (\N\times\{j\})|\leq \nu_0+1-f_1\leq \nu_0-1$, the last inequality following from the fact that $f_1=\Sum{i=1}{r_0}e_im'_i$ and  thus $f_1\geq 2$, for example because we are forced to have by our hypothesis that at least one of the $e_i\geq 2$.
	
	The rest of the proof is now identical to that of Claim $2$ in \Cref{nonarchirel}.\end{proof}

This concludes the proof of \Cref{nonarchilcenter}.
\end{proof}

\subsection{Points with algebras with many splittings}

The other case we treat is that when the algebra $D$, see the beginning of this section for this and other notation, is ``large'' by nature of there being too many components of the Hodge structure $V_{\Q}$, or in other words it splits into too many smaller sub-Hodge structures.  We return to this naming choice of ours in \Cref{section:examples}.

\begin{prop}\label{nonarchsplittings} Let $V_{\Q}$, $V_{dR}$, $D_i$, $m_i$, $F_i$, $e_i$ , and $L/\Q$ be as  above. We assume that the Hodge Conjecture holds for Hodge Endomorphism of $V_{\Q}$ and, as usual, that $h\geq 2$. If for some subset $J\subset \{1,\ldots,r\}$, of the number of non-isomorphic components of the decomposition of $V_{\Q}$ as in  \eqref{eq:decomprepeat}, there exist positive integers $m_i'\leq m_i$, where $i\in J$, such that \begin{equation}\label{eq:nonarchcondtion2}
	 h \geq (\Sum{i\in J}{}m'_i \nu_i)+1 
	\end{equation}and $\{v\in \Sigma_{L,f}: s$ is $v$-adically close to $0\}\neq \emptyset$, then there exists a homogeneous polynomial \begin{center}
		$R_{s,\fin}\in \bar{\Q}[X_{i,j,t}; 1\leq i\leq \mu$, $1\leq h\leq h$, $1\leq t\leq t_0]$ such that
	\end{center}satisfying the same conditions as those in \Cref{nonarchirel}.
	
\end{prop}

\begin{proof}Let us introduce some notation. As usual we start by assuming $\hat{L}=L$ for the field defined in \Cref{propendodr}. 
	
	We assume without loss of generality that $J=\{1,\ldots,r_0\}$. As we did in the proof of \Cref{nonarchilcenter} we write $W:=V_1^{m'_1}\oplus\ldots\oplus V_{r_0}^{m'_{r_0}}$, and set $W_{dR}:= W_{1,dR}^{m'_1}\oplus\ldots\oplus W_{r_0,dR}^{m'_{r_0}}$, and $W_{\et}:= W_{1,\et}^{m'_1}\oplus\ldots\oplus W_{r_0,\et}^{m'_{r_0}}$, for the subspaces of $V_{dR}$ and $V_{\et}$. 
	
	We then consider the family of functionals \begin{equation}\label{eq:functsplit}
		\drfj{j}, \text{ where } 1\leq j\leq \nu_0+1.
	\end{equation}

As before the case $\fgj|_{W_{\et}}=0$ for some $j$ gives linear relations among periods and respective polynomials  $R^{(t)}_0$. 

Assume from now on that  $\fgj|_{W_{\et}}\neq 0$ for all $j$. This shows as discussed in the previous proofs, that the functionals \eqref{eq:functsplit} are $(\nu_0+1)$ non-zero vectors in $\homm(W_{dR}, B_{dR})$ given by extension of scalars from vectors in $W^{*}_{\et,0}:=\homm(W_{\et},\Q_p(n))^{G_{L_v}}$, the latter being a space of dimension $\leq \nu_0$. 

Proceeding as in the previous proofs we then define the matrix $M$ subordinate to a symplectic basis of $V_{dR}$. We order this basis so that $\{w_1,\ldots,w_{n_0}\}$ is a maximal Lagrangian of $W_{dR}$ and choose $w_{\nu+1}$ to be the first vector in a transverse Lagrangian. 

We choose the matrix $M_0$ to consist of the rows that correspond to these vectors. As before $\det(M_0)=0$ and we denote by $R{(t)}_1$ the corresponding homogeneous polynomial induced on the values $i_v(y_{i,j}^{(t)}(\xi))$.\\

The definition of $R_{s,\fin}$ will be the same as before. The properties we need are then straightforwardly checked as earlier. The only property that needs to be checked, as in all other proofs in the section, is that $R^{(t)}_1\notin I_t$, where $I_t$ is the ideal described in \Cref{desidealtriv}. Again as we have been doing so far, we may reduce this to the ``new'' family of G-functions $\tilde{y}^{(t)}_{i,j}(x)$. 

In what follows we write $R'$ for the polynomial obtained by this process on the values of the family $\tilde{y}^{(t)}_{i,j}(x)$ at $x=\xi$. As before, for convenience of notation, we drop mention of the point $s_t$ from our notation and write simply $\tilde{y}_{i,j}$ for this family of G-functions.\\

\textbf{Non-triviality of $R'$:} As before write $M_0$ for the matrix of indeterminates, where $\tilde{y}_{i,j}$ is replaced by $Y_{i,j}$. This matrix of indeterminates will then be \begin{center}
	$M_0=\begin{pmatrix}
	Y_{1,1}&\cdots&Y_{1,\nu_0+1}\\
	 Y_{2,1}&\cdots& Y_{2,\nu_0+1}\\
	 \vdots& & \vdots\\
	 Y_{\nu_0,1}&\cdots &Y_{\nu_0,\nu_0+1}\\	 
	 Y_{\nu+1,1}&\cdots &Y_{\nu+1,\nu_0+1}
\end{pmatrix}$
\end{center}

Consider the reduction homomorphism $\phi: \bar{\Q}[Y_{s,t}]\rightarrow \bar{\Q}[Y_{s,t}:(s,t)\notin \Sigma_1]$ where $\Sigma_1:= \{(\nu+1,t):  1\leq t\leq \nu_0  \}$ with $\ker{\phi}=\langle Y_{s,t}: (s,t)\in \Sigma_1\rangle$. In particular we have that  \begin{center}
	$M_0=\begin{pmatrix}
		Y_{1,1}&\cdots&Y_{1,\nu_0} &Y_{1,\nu_0+1}\\
		Y_{2,1}&\cdots& Y_{2,\nu_0}& Y_{2,\nu_0+1}\\
		\vdots& &\vdots & \vdots\\
		Y_{\nu_0,1}&\cdots &Y_{\nu_0,\nu_0} &Y_{\nu_0,\nu_0+1}\\	 
		0&\cdots &0 &Y_{\nu+1,\nu_0+1}
	\end{pmatrix}$
\end{center}So in particular we have \begin{center}
$R:=\det(\phi(M_0))=\phi(\det(M_0))= Y_{\nu+1,\nu_0+1} \Sigma (-1)^{\sigma} Y_{1,\sigma(1)}\cdots Y_{\nu_0,\sigma(\nu_0)}$. 
\end{center}

If we assume that $R':=\det(M_0)\in I_t$ we would have, as in the previous proofs, that there exist $Q_{j,j'}\in \bar{\Q}[Y_{s,t}:(s,t)\notin \Sigma_1]$ with \begin{equation}\label{eq:splitemember}
	R= \Sigma Q_{j,j'} \bar{P}_{j,j'}, 
\end{equation}where $\bar{P}_{j,j'}=\phi(P_{j,j'})$.

\begin{claim}The ideal $I_0:=\langle \bar{P}_{j,j'}: 1\leq j<j'\leq h\rangle$ is prime and we have that \begin{enumerate}
		\item $Y_{\nu+1,\nu_0+1} \notin I_0$, and 
		
		\item$\Sigma (-1)^{\sigma} Y_{1,\sigma(1)}\cdots Y_{\nu_0,\sigma(\nu_0)}\notin I_0$.
	\end{enumerate}
\end{claim}

\begin{proof}[Proof of claim]By definition of $\phi$ we have 
	\[	\bar{P}_{j,j'}=\sideset{}{'}\sum_{l=1}^{\nu}Y_{l,j'} Y_{\nu+l,j}-\sideset{}{'}\sum_{l=1}^{\nu} Y_{l,j} Y_{\nu+l,j'},\tag{79}\label{eq:phipjj3}
	\]where $\Sigma'$ denotes the sums in this range together with the restriction $(\nu+1,j)$, $(\nu+1,j')\notin \Sigma_1$.
	\setcounter{equation}{79}
	
	By assumption, we have that $h\geq 2$, and hence also by \Cref{maxisotropic} $\nu \geq 2 $.
	
	As before, there are at least two nonzero summands in the two sums in \eqref{eq:phipjj3}. In this setting this is trivial since $Y_{2,j}Y_{\nu+2,j'}$ and $Y_{2,j'}Y_{\nu+2,j}$ appear in these sums, by virtue of $\nu\geq 2$, and are not zero. 
	
	Consider the order on the set $\{Y_{i,j}:(i,j)\notin \Sigma_1\}$ given by the reverse lexicographic order on the pairs of indices $(i,j)\notin \Sigma_1$, so that $(i,j)>(i',j')$ if $i<i'$ or $i=i'$ and $j<j'$. As in the proof of Claim $2$ in \Cref{nonarchirel}, considering the graded lexicographic order on $\bar{\Q}[Y_{i,j}:(i,j)\notin \Sigma_1]$ induced by the above order on the $Y_{i,j}$, the leading term $\LT(\bar{P}_{j,j'})$ of $\bar{P}_{j,j'}$ is given by \begin{equation}
		F_{j,j'}:=\begin{cases}\LT(\bar{P}_{j,j'})=Y_{1,j}Y_{\nu+1,j'}&\text{, if } j'>\nu_0+1 \\
		\LT(\bar{P}_{j,j'})=Y_{2,j}Y_{\nu+2,j'}&\text{, if } j'\leq \nu_0+1.\end{cases}	\end{equation}
	
	Let now $P$, $Q\in \bar{\Q}[Y_{i,j}:(i,j)\in \Sigma_0]$ with $P\cdot Q\in I_0$.  The division algorithm in $\bar{\Q}[Y_{i,j}:(i,j)\in \Sigma_0]$, see \cite{clo} $\S$ $2.3$, gives $P=\Sigma Q^{(1)}_{j,j'} \bar{P}_{j,j'}+R_1$ and $Q=\Sigma Q^{(2)}_{j,j'} \bar{P}_{j,j'}+R_2$ with $R_i$ being either $0$ or $\bar{\Q}$-linear combinations of the $F_{j,j'}$.
	
	We then have $R_1\cdot R_2\in I_0$, so that upon writing $R_1=\Sigma a_{j,j'}F_{j,j'}$ and $R_2=\Sigma b_{j,j'}F_{j,j'}$, there exist $Q^{(3)}_{j,j'}\in \bar{\Q}[Y_{i,j}:(i,j)\in \Sigma_0]$ such that \begin{equation}\label{eq:member5}
		(	\Sigma a_{j,j'}F_{j,j'}	)(\Sigma a_{j,j'}F_{j,j'})=\Sigma Q^{(3)}_{j,j'} \bar{P}_{j,j'}.		
	\end{equation}Since both sides of \eqref{eq:member5} are homogeneous degree $2$ polynomials we have $Q^{(3)}_{j,j'}\in \bar{\Q}$ for all $j<j'$. Assume $Q^{(3)}_{j,j'}\neq 0$ for some pair $(j,j')$. 
	
	By our earlier remark, there exists a monomial of the form $Y_{l,j}Y_{l',j'}\neq F_{j,j'}$ that appears in the expression of $\bar{P}_{j,j'}$ as in \eqref{eq:phipjj3} and by definition will not appear in any other $\bar{P}_{i,i'}$. So $I_0$ is indeed a prime ideal by the same argument as in the proof of Claim $2$ of \Cref{nonarchirel}.
	
It is obvious that $Y_{s,t}\notin I_0$ for all $s,t$ by a similar argument as above. For the fact that $\Sigma (-1)^{\sigma} Y_{1,\sigma(1)}\cdots Y_{\nu_0,\sigma(\nu_0)}\notin I_0$, note that $I_0\leq J_0:=\langle Y_{\nu+l, j} : (\nu+l,j\notin \Sigma_1)\rangle$, this can be trivially seen by \eqref{eq:phipjj3}, while $\Sigma (-1)^{\sigma} Y_{1,\sigma(1)}\cdots Y_{\nu_0,\sigma(\nu_0)}\notin J_0$, again for trivial reasons.
\end{proof}

This concludes the proof of \Cref{nonarchsplittings}.
\end{proof}

   \begin{quest}It is the author's expectation that the weaker assumption that there exists $J$ as in \Cref{nonarchsplittings} for which \begin{equation}\label{eq:nonarchcondtion3}
		h\cdot (\Sum{i\in J}{}m'_i)\geq (\Sum{i\in J}{}m'_i \nu_i)+1 
	\end{equation}also leads to polynomials $R_{s,\fin}$. 

We expect that a similar strategy of proof will work in this case as well. One would want to order the family $\drfj{j}\circ \tau_{i,(j)}$, with $\tau_{i,(j)}$ denoting projections into various summands, and finding a convenient $(\nu_0+1)$-dimensional subspace of $W_{dR}$ as in the previous proof. 

The hard part, as in the earlier proofs, is doing this so that the polynomials we denote by $R^{(t)}_1$ in all previous proofs are not in the ideal $I_t$ of trivial relations of the family $\mathcal{G}^{(t)}$, or to more precise establishing this fact in some straightforward way.
\end{quest}
	
\part{The height bound and applications to unlikely intersections}

We are finally able to put everything together and prove \Cref{maintheorem}. After the proof we discuss the case $n=1$, i.e. the case of a one parameter family of curves of genus $g$ degenerating at one point.

\section{The height bound}\label{section:proofoftheresult}
Consider a morphism $f':X'\rightarrow S'$ satisfying the hypotheses of \Cref{maintheorem}, a good cover $C'$ of $S'$, as in \Cref{section:goodcovers}, and the family of G-admissible variations associated to this good cover.

Note that, by \Cref{goodcoverslemma} and the proof of \Cref{desidealtriv}, we know that if the original morphism $f':X'\rightarrow S'$ of our G-admissible variation of $\Q$-HS satisfies the hypotheses of \Cref{maintheorem} so will every G-admissible variation of the family of such variations associated to the good cover $C'$. From standard properties of the Weil height we are thus reduced to establishing \Cref{maintheorem} where we have replaced the curve $S'$ by its good cover $C'$.\\

Following the above remarks we fix throughout the proof a pair $(C',x)$, with $C'$ a curve over a number field and $x\in K(C')$ with only simple zeroes. 

We write $s_t$, $1\leq t\leq t_0$ for the roots of $x$. We assume that we have a  projective morphism $F':\mathcal{X}'\rightarrow C'$, whose only singular values are the ones in the set $x^{-1}(0)$, and let $F'_t:X'_t\rightarrow C'_t$ be defined as in \Cref{goodcovers}. We also write $F:\mathcal{X}\rightarrow C$, where $C:=C'\backslash x^{-1}(0)$, and assume that this morphism, paired with any of the morphisms $F'_t$, satisfies the hypotheses of \Cref{maintheorem}.

\begin{proof}[Proof of \Cref{maintheorem}:] We start with a short lemma that we write down in the form of a claim.
	\begin{claim}The points $P\in C(\bar{\Q})$ such that the set \begin{center}
			$\Sigma(P):=\{v\in \Sigma_{K(P)}:P\text{ is } v\text{-adically close to }0\}$
		\end{center}is nonempty have height bounded uniformly by a constant $C_4(F)$ depending only on $F$.
		
	\end{claim} 
	\begin{proof}[Proof of Claim] Let $P\in C(\bar{\Q})$ be such that $\Sigma(P)=\emptyset$, and let $\xi:=x(P)$. By \Cref{vadicprox} we will then have $|\xi|_v\geq \min\{1,R_v(\mathcal{G})\}$. This implies		
		 \begin{center}
			$h(\xi^{-1})  \leq \rho(\mathcal{G})$,
		\end{center}
		where $\rho(\mathcal{G})$ is the global radius of the collection of all of the power series $y^{(t)}_{i,j}$. From Lemma $2$, $a)$ of Chapter I.$\S 2.2$ of \cite{andre1989g} we have \begin{center}
			$\rho(\mathcal{G})=\max\{\rho(y^{(t)}_{i,h}): 1\leq i\leq \mu, 1\leq j\leq h,1\leq t\leq t_0 \}$.
		\end{center}The Corollary of Chapter VI.$\S 5$ of loc.cit., gives us that $\rho(y^{(t)}_{i,h})<\infty$. Putting these together we get that $h(\xi)=h(\xi^{-1}) \leq \rho(\mathcal{G})<\infty$. 
	
	The result follows by setting $C_4(F):=\rho(\mathcal{G})$.
	\end{proof}
	
	Let $s\in C(L)$ be a point satisfying the conditions in \Cref{maintheorem} for the variation $\V=R^nF^{an}_{*}\Q_{X^{an}_{\C}}$ where $L/K$ is some finite extension. 
	From now on we assume that the set of places $\Sigma(s)$ is non-empty. We let $\xi:=x(s)$, where $x\in K(C')$ is as above the rational function, that only vanishes at $s_t$, $1\leq t\leq t_0$, with respect to which the $y_i\in\mathcal{G}$ are written as power series.
	
	 By \Cref{propendodr} there exists a finite extension $\hat{L}/L$ such that $D_s$ acts on $H^n_{DR}(\mathcal{X}_{s}\times_{L}\hat{L}/\hat{L})$. From \Cref{propdegreebound} we also know that $\hat{L}$ may be chosen so that $[\hat{L}:L]$ is bounded only in terms of $\mu:=\dim_{\Q} H^n(X^{an}_{s},\Q)$. 
	
	Let $y^{(t)}_{i,j}$ be the G-functions that comprise the sub-family $\mathcal{G}^{(t)}$ of $\mathcal{G}$ that encodes the first $h$ columns of the relative $n$-period matrix associated to the morphism $F$ and the singular value $s_t$ of the morphism $F'$. 
	
	By our assumption that $\Sigma(s)\neq \emptyset$, we will then have that at least one of the products defining the polynomials $R_{s,\infty}$ of \eqref{eq:actualrelation} and $R_{s,\fin}$ of \Cref{nonarchirel}, \Cref{nonarchilcenter}, or \Cref{nonarchsplittings} respectively, is non-vacuous and thus defines a non-constant relation among the values of the G-functions of the family $\mathcal{G}$ at $\xi$. Note that both of these are homogeneous with coefficients in $\bar{\Q}$ and their product $R_s:=R_{s,\infty}\cdot R_{s,\fin}$ will thus be homogeneous with coefficients in $\bar{\Q}$ and, by \Cref{proofnottrivial} and \Cref{nonarchirel}, \Cref{nonarchilcenter}, or \Cref{nonarchsplittings} respectively, will have degree $\leq C_3(F)[L:\Q]$, where $C_3(F)$ is a positive constant depending on the morphism $F$.
	
Combining \Cref{proofnottrivial} and \Cref{nonarchirel}, \Cref{nonarchilcenter}, or \Cref{nonarchsplittings} respectively depending on the case, we know that these polynomials define relations among the values of the G-functions in question at $\xi$ that are non-trivial, i.e. they do not hold on the functional level. Therefore the relation on these values defined from $R_s$ is also non-trivial. 

Finally, we note that this relation is also global. Indeed, we have that for all $v\in \Sigma(s)$ the power series $i_v(y_{i,j}^{(t)})(x)$ converge at $i_v(\xi)$, by definition and $i_v(R_s(i_v(y_{i,j}^{(t)})(i_v(\xi))))=0$ by construction. 
	
	Since we know that this relation is both non-trivial and global we get from \Cref{hasse} that \begin{equation}\label{eq:heightboundprefinal}
		h(\xi)\leq c_1(\vec{y}) \delta^{3\mu h-1}(\log \delta +1),
	\end{equation}where $\delta$ is the degree of the polynomial \eqref{eq:actualrelation} in $\bar{\Q}[x_1,\ldots x_{h\mu}]$. 
	
	By our remarks above we know that $\delta\leq C_3(F)\cdot[L:\Q]$. Combining this with \eqref{eq:heightboundprefinal} we get that if the point satisfies the hypotheses of our theorem and $\Sigma(s)\neq\emptyset$ there exist positive constants $C_5$, $C_6$, independent of the point $s$, such that \begin{equation}\label{eq:heightboundfinal}
		h(\xi)\leq C_5 ([L:\Q]+1)^{C_6},
	\end{equation}as we wanted.\\
	
	By replacing $C_5$ in \eqref{eq:heightboundfinal} by $\max\{C_5, C_4(F)\}$, we find constants as we wanted such that \eqref{eq:heightboundfinal} holds irrespective of whether $\Sigma(s)$ is empty or not. Our result then follows by standard arguments\footnote{See \cite{hindrysilverman} Theorem B.$2.5$.} in the theory of heights comparing $h(s)$ with $h(x(s))=h(\xi)$.
\end{proof}
	
\section{Applications to the case n=1}\label{section:applications}

Here we present an application of \Cref{maintheorem} to the moduli space $\mathcal{M}_g$ of curves of genus $g$. We also compare this result with what was known using previous height bounds. We start with a short review of material we will need to pass from families of curves to families of Jacobians. In particular we will need some light input from variations of mixed $1$-motives and the role of local monodromy in determining the number $h$ of columns of G-functions in the relative period matrix.\\

Throughout this section we consider $S\hookrightarrow \mathcal{M}_g$ smooth irreducible Hodge generic curves defined over $k\subset \C$ algebraically closed, unless otherwise stated. We let $f:X\rightarrow S$ be the associated family of genus $g$ smooth projective curves obtained as the pullback of the universal family. Let us assume that the completion of $S$ in $\bar{\mathcal{M}}_g$ intersects the boundary $\bar{\mathcal{M}}_g\backslash \mathcal{M}_g$ at some point $s_0$ and let $S':=S\cup \{s_0\}$. With respect to \Cref{maintheorem} we are thus in the case where $n=1$, if $k=\bar{\Q}$.

We will write $\V:=R^1f^{an}_{*}\Q$ for the associated variation of polarized weight $1$ $\Q$-HS. Note that in this case the Hodge conjecture is in fact known classically from results of Lefschetz and that the Hodge genericity of the curve $S$ implies that the generic special Mumford-Tate group $G_{smt}(\V)=Sp(\mu,\Q)$. 

\subsubsection*{Some notes about families of 1-motives}

Associated to this family of curves we get an associated family of Jacobians which we denote by $F:\mathcal{A}\rightarrow S$. These Jacobians will be $g$-dimensional principally polarized abelian varieties of dimension $g$. Furthermore it is well known that the variations of weight $1$ polarized Hodge structures associated to $R^1F^{an}_{*}\Q$ and $\V$ coincide. 

The number $h$ in \Cref{maintheorem} then has a natural interpretation in terms of the degeneration of the family of Jacobians. Indeed, by \Cref{andresexistence} we also know that $h=\rank(N)$, where $N$ is the nilpotent endomorphism associated to the local monodromy around the degeneration of the family $\mathcal{A}$ at the point $s_0$. This number is also described by the toric rank of the connected component of the fiber $\tilde{\mathcal{A}}_0$ at $s_0$ of the N\'eron model $F':\tilde{\mathcal{A}}\rightarrow S'$ of the generic fiber $\mathcal{A}_{\eta}$ of $F$ over $S'$, see for example \cite{delignehodge3, jumpsmono}.

\subsection{The height bounds when g=h} 

Let us assume from now on that $g=h$, in other words the family of Jacobians defined has completely multiplicative degeneration at the point $s_0$. When studying this case Y. Andr\'e in \cite{andre1989g} gave the following definition:
\begin{definition}\label{exceptional}A point $s\in S(\C)$ will be called \textbf{exceptional} if the algebra $D_s$ is not equal to the generic algebra of Hodge endomorphisms of the variation.
\end{definition}

We note that Y. Andr\'e studied families where the generic endomorphism algebra could be a totally real field and not just $\Q$ as in our case. So for us ``exceptional points'' are just those for which $D_s\neq \Q$.

The archimedean relations known in the case $h=g$, for Hodge generic one-parameter families of abelian varieties with $g\geq 2$ have already been studied, as noted in the introduction, and are stronger than those in \Cref{constpseudocm}, at least when $g$ is odd. Indeed, there it is known by work of Y. Andr\'e, in the case $g$ is odd, and C. Daw and M. Orr, in the case $g$ is even, that the following holds:
\begin{theorem}[\cite{andre1989g}, \cite{daworr}]\label{andredaworr}Let $s\in S(\bar{\Q})$ be exceptional. Furthermore, if $g$ is even assume that $D_s\not\hookrightarrow M_g(\Q)$.
	
	Then there exists a polynomial $R_{s,\infty}$ satisfying the same properties as in \Cref{proofnottrivial}, assuming $S$ is defined over $\bar{\Q}$ here.
\end{theorem}
\begin{proof}\footnote{The author thanks C.Daw for noticing that the author had misquoted here the conclusion of his Theorem with M. Orr in an earlier version of this paper.}In the case $g$ is odd this is the original construction of Andr\'e, see $X.2.4.1$ of \cite{andre1989g}.
	
	In the case $g\geq2$ then this is either the construction in \Cref{proofnottrivial}, when $s\in\mathcal{E}$ in the notation of \Cref{stronglyexceptional}, or the analogue of \Cref{proofnottrivial} with \Cref{constpseudocm} replaced by the construction of Theorem $8.1$ of \cite{daworr} which assumes that $D_s\not\hookrightarrow M_g(\Q)$ to create said relations at archimedean places. 
\end{proof}
	
We are thus in position to prove \Cref{application1}.

\begin{proof}[Proof of \Cref{application1}] Let $s$ be a point in the set $\mathcal{E}$ in question and note that by the above remarks we are in the case $h=g\geq 2$. Let us write $L=K(s)$, and $\Sigma(s)=\Sigma(s,\infty)\cup \Sigma(s,\fin)$, where  \begin{center}
		$\Sigma(s, \infty)=\{ v\in \Sigma_{L,\infty} : s\text{ is } v\text{-adically close to }0\}$, and 
	
	$\Sigma(s, \fin)=\{ v\in \Sigma_{L,f} : s\text{ is } v\text{-adically close to }0\}$.
	\end{center}
	
	If $\Sigma(s)=\emptyset$ we proceed to bound $h(s)$ by an absolute constant, say $c_0$, as in the proof of \Cref{maintheorem}. If $\Sigma(s,\fin)=\emptyset$ using \Cref{andredaworr} we may apply \Cref{hasse} to find  constants $c_3$, $c_4$ in the place of $c_1$, $c_2$ above. Thus, from now on assume that $\Sigma(s,\fin)\neq\emptyset$.
	
	Let us write $\mathcal{A}_s\sim A_1^{m_1}\times \cdots \times A_r^{m_r}$ for the decomposition of $\mathcal{A}_s$ into powers of simple abelian varieties, with $A_i$ not isomorphic to $A_j$ for $i\neq j$. Let us also write $\nu_i:=\dim A_i$ and $D_i =\End^{0}_{\C}(A_i)$, $F_i$ for its center, and $e_i=[F_i:\Q]$.  
	
        For now assume that $g$ is odd. Then, by definition $g=h=\Sum{i=1}{r} \nu_i m_i$. If $r\geq 2$ we can thus find we must have that $h=g\geq \nu_1+1$. Here we proceed as in the proof of \Cref{maintheorem} using \Cref{andredaworr} and \Cref{nonarchsplittings}. This gives us $c_5$, $c_6$ in place of $c_1$ and $c_2$, independent of the point $s$ with these properties.
    
    If $r=1$ and $m_1\geq 2$ we proceed exactly as in the previous case. We thus get $c_7$, $c_8$ in place of $c_1$ and $c_2$, again independent of the point $s$ with these properties.
    
    Finally, assume $r=1$, $m_1=1$. Then $\mathcal{A}_s$ is simple with algebra $D\neq \Q$. If $e=[Z(D):\Q]\geq 2$ we then proceed as in \Cref{maintheorem} using \Cref{nonarchirel}, since $(e-1)h\geq g=\nu_1\geq h$ in this case. Again we obtain $c_9$, $c_10$ in place of $c_1$ and $c_2$, which will be independent of the point $s$ with these properties. 
    
    If $e=1$ we know that we must be in the case where $D$ is of either type $II$ or $III$, a quaternion algebra over $\Q$. This forces $2|g$, see for example the table in page $187$ of \cite{mumfordabelian}. Thus we have exhausted the case $g$ is odd. 
    
   Now assume that $g$ is even.  If $s\in \mathcal{E}$, as in \Cref{stronglyexceptional} we are done by \Cref{maintheorem}, getting constants $c_5$ and $c_6$ as above. On the other hand, by definition of $\mathcal{E}_g$ in this case, if $D_s\not\hookrightarrow M_g(\Q)$ we have $\Sigma(s,\fin)=\emptyset$ by Lemma $3.4$ of Chapter $X$ of \cite{andre1989g}, thus we are in the case at the beginning of the proof.
    
    Taking $c_1$ to be the maximum of $c_0$, and the $c_i$ with $i$ odd and $c_2$ to be the maximum of the above $c_i$ with $i$ even we are done.\end{proof}

\subsection{Large Galois orbits}

The most important application of the height bounds of \Cref{maintheorem} that we had in mind, in the context of the Pila-Zannier strategy towards the Zilber-Pink Conjecture, is their usage in establishing so the largeness of the Galois orbits of atypical points. 

Following \Cref{maintheorem}, or \Cref{application1} when $g=h$, and the techniques of \cite{daworr,daworr2,daworr3} one can prove: 

\begin{cor}\label{lgoforcurvesincurves}Let $f:X\rightarrow S$, defined over some number field $K$, be the morphism underlying a G-admissible variation of Hodge structures with fibers smooth irreducible curves of genus $g\geq 2$. 
	
	Consider the set of points $s\in S(\bar{\Q})\cap \mathcal{E}_h$, where $\mathcal{E}_h$ is defined to be the set $\mathcal{E}$ of \Cref{stronglyexceptional} if $h<g$ and $\mathcal{E}_h=\mathcal{E}_g$ as defined in \Cref{application1} when $h=g$. 
	
	Then there exist positive constants $c$, $c'$ such that for all of the above $s$ we have\begin{equation}
		c\cdot |\disc(R_s)|^{c'}\leq \aut (\C/K)\cdot s,	\end{equation}
where $R_s:=\End_{\bar{\Q}}(\mathcal{A}_s)$ is the ring of endomorphisms of the Jacobian $\mathcal{A}_s$ of the fiber $X_s$. 
\end{cor}

Note here that $R_s$ is an order of the semisimple algebra $D_s$ of endomorphisms of Hodge structures at the point.

In other words, the Galois orbit of any point with the above properties is bounded from below. In the same spirit we can also prove the following slight variation of the above corollary.

\begin{cor}\label{lgogeneral} Let $f:X\rightarrow S$ be as above. For $s\in \mathcal{E}_h$, defined as in \Cref{lgoforcurvesincurves}, let us write \begin{center}
		$\mathcal{A}_s\sim A_1^{m_1}\times \cdots \times A_r^{m_r}$
	\end{center} for the usual decomposition of the corresponding Jacobian $\mathcal{A}_s$.
	
	Let $R_i:=\End_{\bar{\Q}}(A_i)$ be the endomorphism ring of $A_i$. Then there exist positive constants $c$, $c'$, such that for all points in $\mathcal{E}_h\cap S(\bar{\Q})$ we have \begin{equation}
		c\cdot (\max|\disc(R_i)| )^{c'} \leq \aut (\C/K)\cdot s.
	\end{equation}
	
	\end{cor}

\begin{proof}Base changing the decomposition of $\mathcal{A}_s$ by the field $\widehat{K(s)}$ of \Cref{propendodr} we may assume from now on, without loss of generality, by using the analogue of \Cref{propdegreebound} for abelian varieties proven in \cite{silverberg}, $\widehat{K(s)}=K(s)$. In particular, we now have that $\End_{\bar{\Q}}\mathcal{A}_s=\End_{K(s)}\mathcal{A}_s$. From now on we let $d:=[K(s):K]$.
		
	By \cite{mwihes}, see the discussion on page $6$ and page $23$, we get that there exist positive constants $c_1$ and $\kappa_1$ depending only on $g$ such that 
	\begin{equation}\label{eq:mw1}\max\{h(A_i)\}\leq c_1 \max\{ d, h_F(\mathcal{A}_s)\}^{\kappa_1}.
			\end{equation}
		
		On the other hand, from \cite{mwendoesti} one knows that for all $i$ there exist constants $\alpha_i$, $\lambda_i$ that depend only on the dimension $g_i:=\dim(A_i)$ such that \begin{equation}\label{eq:mw2}
|\disc(R_i)|\leq \alpha_i \max\{d,h_F(A_i)\}^{\lambda_i}
		\end{equation}Taking $\kappa_2=\max\{\lambda_i\}$ and $c_2=\max\{\alpha_i\}$ we then get $(\max|\disc(R_i)| )\leq c_2 \max\{d,h_F(A_i)\}^{\kappa_2}$, which combined with \eqref{eq:mw1} gives \begin{equation}
(\max|\disc(R_i)| )\leq c_3 \max\{ d, h_F(\mathcal{A}_s)\}^{\kappa_3},	
\end{equation}the $c_3$ and $\kappa_3$ depending only on $g$. To see this last dependence we note that $g_i\leq g$ and all the constants that we use from \cite{mwendoesti} are increasing as functions in their dependence on $g_i$, see the bottom of page $650$ of \cite{mwendoesti} for this.

The conclusion follows by using the height bounds of \Cref{maintheorem}, or \Cref{application1} depending on whether $h=g$ or not, the exact same way as in the proof of Theorem $6.5$ in \cite{daworr2}.
\end{proof}

\subsection{Zilber-Pink type statements}\label{section:zpstatements}

Here we use results of Urbanik as well as C. Daw and M. Orr that reduce the Zilber-Pink conjecture, in the context of atypical points who owe their atypicallity to having large endomorphism algebras, to height bounds as those established here.\\

We start with the following:
\begin{theorem}\label{zpforcurvesincurvessplit}Let $S\hookrightarrow \mathcal{M}_g$, with $g\geq 3$, be a Hodge generic smooth irreducible curve defined over $\bar{\Q}$. Assume that the compactification $\bar{S}$ of $S$ intersects the boundary $\bar{\mathcal{M}}_g\backslash \mathcal{M}_g$ at a point $s_0$ and let $S':=S\cup \{s_0\}$. 
	
	Let $f:X\rightarrow S$ be the associated $1$-parameter family of smooth irreducible projective genus $g$ curves. Let  $h:=t_{\rank}\tilde{\mathcal{A}}'_{s_0}$ be the toric rank of $\tilde{\mathcal{A}}'_{s_0}$, i.e. the connected component of the identity of the fiber of the N\'eron model of the family of Jacobians at the degeneration $s_0$ as before.
	
Let $\mathcal{E}_{h}$ be the points in $\mathcal{M}_g(\C)$ defined by the same conditions as those for the set $\mathcal{E}$ in \Cref{stronglyexceptional} if $h<g$, and $\mathcal{E}_h$ be the set of points in $\mathcal{M}_g(\C)$ defined by the same conditions as those for the set $\mathcal{E}_g$ in \Cref{application1} when $h=g$.

 Consider the set $\mathcal{I}_h\subset\mathcal{E}_{h}$ whose points $s$ are either CM or such that the corresponding Jacobian $\mathcal{A}_s$ is isogenous to a non-simple abelian variety.
	
	Then $\mathcal{I}_g\cap S(\C)$ is finite.	
\end{theorem}

\begin{proof}The finiteness of CM-points follows from the Andr\'e-Oort Conjecture in $\mathcal{A}_g$, which is a theorem of J. Tsimerman \cite{tsimermanag}. So from now on we assume the points in question are not CM.
	
	
	Since the points of interest are atypical and are in the intersection of a special subvariety of $\mathcal{A}_g$, which is thus defined over $\bar{\Q}$, with our curve, which is also defined over $\bar{\Q}$, we get $\mathcal{I}_g\cap S(\C)\subset S(\bar{\Q})$. 
	
	Since $g\geq 3$ the result then follows by combining \Cref{maintheorem} with Corollary $7.15$ of \cite{davidg}, noting that here $H_S=G_S$ in the notation of loc. cit. by \Cref{monoatgeneric}, and the discussion at the end of $\S$ $8.1$ of loc.cit.. 
\end{proof}

Focusing on points $s\in \mathcal{E}_h$, the latter defined as in \Cref{zpforcurvesincurvessplit}, for which the algebra $D_s$ is simple of type $I$ or $II$ we can in fact prove similar results.

\begin{theorem}\label{zpsimplecurves}Let $S\hookrightarrow \mathcal{M}_g$ and $\mathcal{E}_{h}$ be as in \Cref{zpforcurvesincurvessplit} with $h\geq 2$.
	
	Consider the set $\mathcal{S}_{I,II}\subset \mathcal{E}_{\mathcal{M}_g}$ corresponding to points in $\mathcal{M}_g(\C)$ for which $D_s$ is simple of type $I$ or $II$.
	
	Then $\mathcal{S}_{I,II}\cap S(\C)$ is finite.
\end{theorem}

\begin{proof}As in the proof of \Cref{zpforcurvesincurvessplit}, we see that  $\mathcal{S}_{I,II}\cap S(\C)\subset S(\bar{\Q})$.
	
	Now we use Theorem $1.3$ of \cite{daworr3} for the associated family of Jacobians. Note that Conjecture $1.5$ of loc. cit. is then satisfied in this case by \Cref{lgoforcurvesincurves}. 
	

\end{proof}

Combining the above theorems together in the case $g=h$ we get \Cref{zpgish}.

\begin{remarks}1. For the case $g=2$, see page $50$ of \cite{davidg}, it is known that if the curve $S$ intersects the boundary of $\mathcal{M}_2$ in $\bar{\mathcal{M}}_2$ in the locus $\mathcal{B}$ described on page $1$ of loc. cit., then we will have $h=2$.
	
	As noted in the introduction, all that is needed to establish the Zilber-Pink conjecture in full, following \cite{daworr,daworr2}, for such curves is an analogue of \Cref{constpseudocm} for points with algebra of endomorphisms isomorphic to $M_2(\Q)$, i.e. points for which the Jacobian of $X_s$ is a product of isogenous non-CM elliptic curves.\\

	2. We expect that one can prove the analogous statement of Theorem $1.2$ of \cite{daworr3} for $D$ a division algebra of type $III$ or $IV$ in \Cref{albert}. The techniques of \cite{daworr3} should then be able to reduce the Zilber-Pink conjecture for intersections of our curve $S\subset \mathcal{M}_g$ with points whose algebras are simple and of the above types to a Large Galois orbits hypothesis. Our \Cref{lgoforcurvesincurves} should then kick in to deal with this hypothesis.
	
	This would for example prove the Zilber-Pink conjecture for all Hodge generic smooth curves completely for $g=h=3$. It would also give if $h=g=$odd, the following Zilber-Pink-type statement:
	
	\begin{center}
		The points of intersection of a smooth irreducible Hodge generic curve in $\mathcal{M}_g$ for which $h=g$ with the set of all special subvarieties $\mathcal{S}\subset \mathcal{M}_g$ which are such that $\End(H^1(X^{an}_s,\Q))^{G_{\mathcal{S}}}\neq \Q$ for any Hodge generic point $s\in \mathcal{S}(\C)$, where $G_{\mathcal{S}}$ is the generic Mumford-Tate group of $\mathcal{S}$, is finite.
	\end{center}
	
	In other words, we expect that the above results will establish the finiteness, in any such curve in $\mathcal{M}_g$ with $h=g=$odd, of points whose Hodge structure has non-trivial endomorphisms, i.e. all exceptional points in the original \Cref{exceptional} of Y. Andr\'e. 
	\end{remarks}

Another possible avenue to establish the Zilber-Pink Conjecture in this setting follows naturally from Corollary $7.14$ in \cite{davidg}, the discussion preceding it, and \Cref{lgogeneral}.

\begin{conj}\label{conjectureondet}Let $X$ be a simple $g$-dimensional abelian variety defined over $\bar{\Q}$ with ring of endomorphisms $R:=\End_{\bar{\Q}}(X)$.
	
	Let $\phi_1,\ldots,\phi_d$ be some integral basis of $R$ and let $D:=R\otimes_{\Z}\Q$ be the algebra of endomorphisms of $X$. Then there exist positive constants $c_1$ and $ c_2$ that depend only on $g$ and the type of $D$ in \Cref{albert}, such that

	\begin{equation}
c_1|\disc(R)|^{c_2}\geq \max\{\det(\phi_1),\ldots,\det(\phi_d)\},
	\end{equation}where the determinants of the $\phi_i$ are with respect to the natural embedding $D\hookrightarrow \GL_{2g}(\Q)$.
\end{conj}

We note that \Cref{conjectureondet} can also be naturally phrased without mention of $X$, but rather as a conjecture about bounds of the determinants of an integral basis of a maximal order of a division algebra $D$ that has the same properties as the algebras in \Cref{albert}. The determinants considered here will be those coming from an embedding $D\hookrightarrow \GL_{2g}(\Q)$ associated to some representation of the algebra $D$.

Combining the aforementioned corollary of Urbanik with \Cref{lgogeneral} it is easy to see that establishing \Cref{conjectureondet} would lead to Zilber-Pink type statements in our setting.	
	
\subsection{Examples}\label{section:examples}

We have already seen several examples of points in the sets $\mathcal{E}$ defined in \Cref{stronglyexceptional} in the case $n=1$ and $g=h$. Let us move away from the condition $h=g$, thus abandoning \Cref{andredaworr} for \Cref{proofnottrivial}.

The conditions imposed on $\mathcal{E}$ by the need to have archimedean relations are definitely stronger than those coming from the respective need for relations at finite places. We deal with some extremal cases that we think highlight some of the restrictions.\\

Let us write $V_s=V_1^{m_1}\oplus\ldots\oplus V_r^{m_r}$ for the decomposition of the Hodge structure at some point $s\in \mathcal{E}$, associated to some weight $1$ G-admissible variation of Hodge structures. Let us set $D_i:=\End_{HS}(V_i)$, $F_i=Z(D_i)$, $e_i=[F_i:\Q]$ and $2\nu_i=\dim V_i$. \\

\textbf{Case 1: $h=g-1\geq 2$} In this case, it is easy to see that every such point for which $e_i\geq 3$ for some $i$ will satisfy at least one of the first two conditions defining $\mathcal{E}$. Assume this holds for $i=1$ and consider the last two conditions of \Cref{stronglyexceptional}. If the last of those two fails we must have that $r=1$ and $\nu_1=g$. Therefore, in the previous step we would have that $i=1$ and $e_1\geq 3$ which obviously implies the third condition on \Cref{stronglyexceptional} in this case.

In conclusion, as long as there exists some $i$ for which $e_i\geq 3$ we get $s\in \mathcal{E}$!\\

\textbf{Case 2: $h=2$} If $r=m_1=1$ then the only points that satisfy the first two conditions are those for which $[F:\Q]\geq g$ and $D$ is of type $IV$.  Note that in this case the third condition is satisfied automatically.  CM-points are definitely contained in the above, but this set contains potentially more ''pseudo-CM points''! 

If we are allowed to work with $r\geq 2$ we can then get an abundance of examples. Take for instance $g\geq 3$ and assume that $V_s$ has a two-dimensional summand $V_1$ of CM-type. Then regardless of the rest of the summands $s\in \mathcal{E}$ since the second and fourth condition are satisfied! 

Similarly, assume that $V_s$ has a summand that is of type $IV$ with $D_1$ such that $[F_1:\Q]=\frac{(\dim_{\Q} V_1)}{2}\geq 2$, in particular not a CM-HS. Then again $s\in \mathcal{E}$ since the first and third conditions are satisfied in \Cref{stronglyexceptional}. Again this happens regardless of what the algebras of the rest of the summands look like!

Thus in the case $h=2$ for large enough $g$ our results go slightly beyond Theorem $1.1$ of \cite{davidg}, where we can now deal with points that have a proper summand that is pseudo-CM, i.e. of type $IV$ in Albert's classification with large center but not necessarily CM.

		\part*{Appendix}	
	\appendix
	\bookmarksetupnext{level=-1}	\addappheadtotoc

\section{Some notes on polarizations}\label{appendixpolarizations}

\subsection{The non-relative case}

\textbf{Notation:} Let $X/k$ be a smooth projective variety over a subfield $k$ of $\C$ and let $n=\dim_k X$.\\

\subsubsection*{Short review on polarizing forms}
For all $d\in \N$ there exist non-degenerate bilinear polarizing forms
\begin{center}$\langle ,\rangle_{DR} : H^d_{DR} (X/k)\otimes_k H^d_{DR} (X/k)\rightarrow k$, and 
	
\end{center}\begin{center}
$\langle , \rangle_{B} : H^d(X^{an}_{\C},\Q)\otimes_{\Q} H^d(X^{an}_{\C},\Q)\rightarrow (2\pi i)^{-d}\Q=\Q(-d)$,
\end{center}on de Rham cohomology and Betti cohomology respectively. We also write $\langle ,\rangle_{B}=(2\pi i)^{-d} \langle ,\rangle $, where $\langle ,\rangle $ has values in $\Q$ and is of the same type, i.e. symmetric or skew-symmetric, as $\langle ,\rangle_{B}$.

These two are the polarizing forms of the corresponding cohomology group. Their existence follows from the fact that $X$ is projective and smooth and they are constructed via a very ample line bundle \cite{delhodge2}. 

We also have that, via the two embeddings $k\hookrightarrow \C$ and $(2\pi i)^{-d} \Q\hookrightarrow \C$, and the comparison isomorphism 
\begin{center}
	$P^{d}_{X}: H^d(X/k)\otimes_{k} \C \rightarrow H^{d} (X^{an}_{\C} ,\Q)\otimes_{\Q} \C$,
\end{center}the two bilinear forms $\langle , \rangle_{DR}$ and $\langle , \rangle_{B}$ are compatible under $P^d_{X}$, meaning that 
\begin{equation}\label{eq:polarscomp}\langle v,w \rangle_{DR}=\langle P^d_{X}(v) ,P^d_{X}(w) \rangle_{B}, \forall v,w\in H^d_{DR} (X/k)\otimes_k\C.	
\end{equation}

\subsection*{Relations on periods-Notation}

From now on we assume that $d=n=\dim_kX$.  We can and do consider from now on the above polarizing forms $\langle , \rangle_{DR}$, $\langle , \rangle_{B}$, and the form $\langle , \rangle$ as vectors in the spaces $H^n_{DR} (X/k)^{*} \otimes_kH^n_{DR} (X/k)^{*}$, $(H^n(X^{an}_{\C},\Q)^{*} \otimes_{\Q} H^n(X^{an}_{\C},\Q)^{*})(-n)$, and $H^n(X^{an}_{\C},\Q)^{*} \otimes_{\Q} H^n(X^{an}_{\C},\Q)^{*}$ respectively. 

In this case, i.e. $d=n$, via Poincar\'e duality, these forms will correspond to elements $t_{DR} \in H^n_{DR}(X/k)\otimes_{k}H^n_{DR}(X/k)$, $t_B\in (H^n(X^{an}_{\C},\Q)\otimes_{\Q} H^n(X^{an}_{\C},\Q))(n)$, and $t\in H^n(X^{an}_{\C},\Q)\otimes_{\Q} H^n(X^{an}_{\C},\Q)$, respectively.

The compatibility of the comparison isomorphism $P^n_{X}$ with Poincar\'e duality implies that $P^n_X\otimes P^n_X (t_{DR}) =t_B=(2\pi i)^{n}t$. In particular $t_{DR}$ is a Hodge class defined over the field $k$. For cycles such as this it is known\footnote{See page $169$ of \cite{andre1989g}.} that they impose polynomial relations among the $n$-periods with coefficients in the field $k((2\pi i)^n)$.

In what follows we show that the aforementioned relations imposed by $t_{DR}$ are in fact the Riemann-relations, i.e. they are the equations imposed on the $n$-periods by \eqref{eq:polarscomp}. This is used without proof by Andr\'e in, essentially, the case where $n=1$. The author is sure that this part is known to experts in the field and includes it only for the sake of completeness of the exposition.\\

\textbf{Notation:} We consider from now on a fixed basis $\{ \gamma_i: 1\leq i\leq \mu:=\dim_\Q H^n(X^{an}_{\C},\Q)\}$ of $H_n(X^{an}_{\C},\Q)$ and we let $\gamma_i^{*}$ be the elements of its dual basis, which constitutes a basis of $H^n(X^{an}_{\C},\Q)$. We also consider $\omega_i$, $1\leq i\leq \mu$, a fixed $k$-basis of $H^n_{DR}(X/k)$.

With respect to these choices the isomorphism $P^n_X$ corresponds to the matrix $( \int_{\gamma_j} \omega_i )$. We denote this matrix also by $P^n_X$ so that the isomorphism is nothing but $P^n_X(v)=\prescript{t}{}{v} P^n_X$, where on the right we have the matrix acting on the right. Vectors in the various spaces will be considered as column vectors in the various bases. Finally, we denote the matrix of the $n$-periods by $P:= (2\pi i)^{-n} P^n_X$.\\

With the above notation fixed we let $\langle \omega_i,\omega_j\rangle_{DR}=d_{i,j}$ and let $M_{DR}=(d_{i,j})\in \GL_{\mu}(k)$, which will be the matrix corresponding to the form $\langle,\rangle_{DR}$, i.e. \begin{center}
	$\langle v,w\rangle_{DR}=\prescript{t}{}{v}M_{DR} w$.
\end{center}Considering, alternatively as above, $\langle,\rangle_{DR}$ as an element of the space $H^n_{DR} (X/k)^{*}\otimes_{k} H^n_{DR} (X/k)^{*}$, the above are equivalent to
\begin{center}
	$\langle ,\rangle_{DR} =\Sum{i,j=1}{\mu} d_{i,j} \omega^{*}_i\otimes \omega^{*}_j$.
\end{center}

Similarly we let $q_{i,j}=\langle \gamma^{*}_i,\gamma^{*}_j\rangle \in \Q$ and set $M_B=(q_{i,j})\in \GL_\mu(\Q)$. This implies that $\langle \gamma_i^{*} ,\gamma_{j}^{*}\rangle_B =(2\pi i)^{-n} q_{i,j}$. Same as above these relations can be rewritten as 
\begin{center}
	$\langle v,w\rangle=\prescript{t}{}{v} M_B w$ and $\langle v,w\rangle_{B}=\prescript{t}{}{v} ((2\pi i)^{-n} M_B)w$,
\end{center}for all $v,w \in H^n(X^{an}_{\C} ,\C)$. Alternatively, if we were to consider $\langle,\rangle $ and $\langle ,\rangle_{B}$ as elements of $H^n(X^{an}_{\C},\C)^{*}\otimes_{\C} H^n(X^{an}_{\C},\C)^{*}$ we can write these as $\langle ,\rangle =\sum_{i,j=1}^{\mu} q_{i,j} \gamma_i\otimes\gamma_j$, and $\langle ,\rangle_{B} =\sum_{i,j=1}^{\mu}(2\pi i)^{-n} q_{i,j} \gamma_i\otimes\gamma_j $ respectively.

We now consider the Poincar\'e duality isomorphisms $\Pi_{DR}:H^n_{DR} (X/k)\rightarrow H^n_{DR} (X/k)^{*}$ and $\Pi_B:H^n(X^{an}_{\C},\Q)\rightarrow H^n(X^{an}_{\C},\Q)^{*}$, which we have since $\dim_k X=n$. With respect to the bases $\{ \omega_i \}$ and $\{\omega^{*}_i \}$ the isomorphism $\Pi_{DR}$ corresponds to an invertible matrix which we denote by $A_{DR}\in\GL_{\mu}(k)$. Similarly, with respect to the bases $\{ \gamma_i\}$ and $\{\gamma^{*}_i\}$ we get the invertible matrix $A_B$ corresponding to $\Pi_B$.

Finally, let us write $t_{DR} =\sum_{i,j=1}^{\mu}\lambda_{i,j} \omega_i\otimes \omega_j$, $t=\sum_{i,j=1}^{\mu}\tau_{i,j}\gamma_i^{*} \otimes \gamma_{j}^{*}$ and $t_B=\sum_{i,j=1}^{\mu}(2\pi i)^{-n}\tau_{i,j}\gamma_i^{*} \otimes \gamma_{j}^{*}$, where $\lambda_{i,j}\in k$ and $\tau_{i,j}\in\Q$. We also define $\Lambda_{DR}=(\lambda_{i,j})$ and $\Lambda_{\Q}=(\tau_{i,j})$.
\subsection*{Classes and forms}

With the above notation fixed from now on we turn to describing the relation between the classes $t_{DR}$, and $t$ and the respective forms.\\

By definition we have $\Pi_{DR}^{\otimes 2}(t_{DR})=\langle ,\rangle_{DR}$. This implies that \begin{equation}\label{eq:basiscyclesdr}
	\sum_{i,j=1}^{\mu} \lambda_{i,j} \Pi_{DR}(\omega_i) \otimes \Pi_{DR}(\omega_j)=\sum_{i,j=1}^{\mu} d_{i,j} \omega_i^{*}\otimes \omega_j^{*}.
\end{equation}We know that $\Pi_{DR}(\omega_i)=\Sigma a_{i,j} \omega_j^{*} $, with $A_{DR} =(a_{i,j})$ . Applying this to \eqref{eq:basiscyclesdr} it is easy to see, with a few trivial computations, that $\prescript{t}{}{A}_{DR} \Lambda_{DR}A_{DR} =M_{DR}$, or equivalently we get the equality\begin{equation}\label{eq:matricesdr}
\Lambda_{DR}=\prescript{t}{}{A}^{-1}_{DR} M_{DR} A^{-1}_{DR}.
\end{equation}

Similarly for the pair $t$ and $\langle ,\rangle $ we find that\begin{equation}\label{eq:matricesbetti}
	\Lambda_{\Q}=\prescript{t}{}{A}_{B}^{-1} M_B A_B^{-1},
\end{equation}coming from the equality $\Pi_B^{\otimes 2} (t)=\langle ,\rangle$.\\

\begin{flushleft}
	\textbf{The relation given by $t_{DR}$.}
\end{flushleft}
We review how a relation on the $n$-periods is constructed from $t_{DR}$. We start from the equality $(2\pi i)^{-n} (P_X^n)^{\otimes 2} (t_{DR} )=t$. This in turn implies that for all $l,m$, with the notation as above, we have \begin{equation}
	\sum_{i,j=1}^{\mu} \lambda_{i,j} ((2\pi i)^{-n} \int_{\gamma_l}^{} \omega_i)   (   (2\pi i)^{-n} \int_{\gamma_m}^{} \omega_j ) =(2\pi i)^{-n} \tau_{l,m}.	
\end{equation}

These equations are the relations between $n$-periods that we eluded to earlier. Putting them altogether the previous relation is equivalent to the equality 
\begin{equation}\label{eq:relationsmatrix}
	\prescript{t}{}{P} \Lambda_{DR} P = (2\pi i )^{-n} \Lambda_{\Q}.
\end{equation}

\begin{flushleft}
	\textbf{Comparing the matrices $A_B$ and $A_{DR}$.}
\end{flushleft}

Earlier on we had the matrices $A_{DR}$ and $A_B$ corresponding to the respective Poincar\'e duality isomorphisms. We saw in \eqref{eq:matricesdr} and \eqref{eq:matricesbetti} how these matrices relate the ``$\Lambda$-matrices'' and ``$M$-matrices''. We would like to replace the ``$\Lambda$-matrices in \eqref{eq:relationsmatrix} by the corresponding ``$M$-matrices'', showing thus that the relations created are nothing but the Riemann-relations\footnote{See \Cref{section:subsectionnotationsnontrivial} for a definition and the reason of why we needed these.} . The first step is to describe how the matrices $A_{DR}$ and $A_B$ relate to one another.\\

We had the isomorphisms $\Pi_{DR}$ and $\Pi_B$ and the matrices $A_{DR}$ and $A_B$ that represented these with respect to the bases we have chosen. We know that the comparison isomorphism $P^n_X$ respects Poincar\'e duality, meaning that the following diagram commutes:\begin{center}
	$\begin{tikzcd}
	H^n_{DR} (X/k)\otimes_k \C\arrow[d, "\Pi_{DR}\otimes_{k}\C"'] \arrow[r, "P^n_X"] & H^n(X^{an}_{\C}, \Q)\otimes_{\Q} \C \arrow[d, "\Pi_B\otimes_{\Q}\C"] \\
	H^n_{DR} (X/k)^{*}\otimes_k\C                               & H^n(X^{an}_\C,\Q)\otimes_{\Q}\C \arrow[l, "(P^n_X)^{\vee}"]
\end{tikzcd}$\end{center}
where $(P^n_X)^{\vee} (f) =f\circ P^n_X$ for all $f\in H^n(X^{an}_\C, \Q)^{*}\otimes_{\Q}\C$.

Looking at what the relation of the above diagram, i.e. $\Pi_{DR}\otimes_{k}\C=(P_X^n)^{\vee}\circ (\Pi_B\otimes_{\Q}\C)\circ P^n_X$, does to the basis $\{\omega_i\}$, and using the fact that with respect to the bases $\{\gamma_j\}$ and $\{ \omega_i^{*}\}$ the matrix representing $(P^n_X)^{\vee}$ will be the matrix $\prescript{t}{}{(\int_{\gamma_j}^{}\omega_i )}$, i.e. the transpose of $P^n_X$, we conclude that \begin{equation}\label{eq:theamatrices}
	A_{DR} = (\int_{\gamma_j}^{}\omega_i )\cdot A_B \cdot \prescript{t}{}{(\int_{\gamma_j}^{}\omega_i )}.
\end{equation}

\begin{flushleft}
	\textbf{Conclusions}
\end{flushleft}

Combining \eqref{eq:relationsmatrix} with \eqref{eq:matricesdr} and \eqref{eq:matricesbetti} we get \begin{eqnarray}\label{eq:almosttheend}
	\prescript{t}{}{P} (\prescript{t}{}{A}^{-1}_{DR} M_{DR} A^{-1}_{DR}) P= (2\pi i)^{-n} (\prescript{t}{}{A}_B^{-1} M_B A_B^{-1}).
\end{eqnarray}
From \eqref{eq:theamatrices} we get\begin{equation}\label{eq:thematrices1}
	\prescript{t}{}{P} \prescript{t}{}{A}^{-1}_{DR} =\frac{1}{(2\pi i)^{n}}   \prescript{t}{}{A}_B^{-1} P, \text{ and}
\end{equation}
\begin{equation}\label{eq:theamatrices2}
	A^{-1}_{DR} P =\frac{1}{(2\pi i)^{n}}    \prescript{t}{}{P}^{-1} A_B^{-1}.
\end{equation}

Using \eqref{eq:thematrices1} and \eqref{eq:theamatrices2} together with \eqref{eq:almosttheend} we get\begin{equation}\label{eq:riemannrelationsmatrixorm}
	PM_B   \prescript{t}{}{P} =(2\pi i)^{-n} M_{DR}.
\end{equation}But, this is the relation we get between the above matrices by looking at the equation \eqref{eq:polarscomp} and translating it in terms of matrices. Indeed, \eqref{eq:polarscomp} translates to\begin{center}
	$\prescript{t}{}{v} P^n_X ((2\pi i)^{-n} M_B) { } \prescript{t}{}{(\prescript{t}{}{w} P^n_X)}=\prescript{t}{}{v}M_{DR}w$ for all $v,w\in H^n_{DR} (X/k)\otimes_k\C$.
\end{center}From this we recover \eqref{eq:riemannrelationsmatrixorm} by multiplying on both sides by $(2\pi i)^{-n}$ and noting that $P=(2\pi i)^{-n} P^n_X$.\\

What is actually of use to us is not exactly \eqref{eq:riemannrelationsmatrixorm} but rather the same relation for the transpose of $P$. To obtain this, first from \eqref{eq:riemannrelationsmatrixorm} we get trivially\begin{center}$(2\pi i)^{n} M_B  = P^{-1} M_{DR}\prescript{t}{}{P}^{-1}$,
\end{center}then taking inverses on both sides we get \begin{equation}\label{eq:riemannwewant}
\prescript{t}{}{P} M_{DR} ^{-1}  P =(2\pi i)^{-n} M_{B}^{-1}.
\end{equation}
\subsection{The relative case}

\textbf{Setting:} We consider $f:X\rightarrow S$ a smooth projective morphism of $k$-varieties itself defined over the same subfield $k$ of $\C$. We also assume that $S$ is a smooth connected curve which is not necessarily complete over $k$ and the dimension of the fibers of $f$ is $n$.

We then have, for all $d\in\N$, the relative version of the comparison isomorphism between the algebraic de Rham and the Betti cohomology\begin{equation}\label{eq:relativeisomorphism} P^d_{X/S}:H^d_{DR}(X/S)\otimes_{\mathcal{O}_S} \mathcal{O}_{S^{an}_{\C}}\rightarrow R^df^{an}_{*}\Q_{X^{an}_\C}\otimes_{\Q_{S^{an}_\C}}\mathcal{O}_{S^{an}_{\C}}.
\end{equation}Once again we let, in parallel to the non-relative case we studied earlier, $\mu$ denote the rank of these sheaves.

We once again have the same picture, as far as polarizing forms are concerned, as in the non-relative case. In other words we have $\langle,\rangle_{DR}$ a polarizing form of the de Rham cohomology sheaves $H^d_{DR}(X/S)$ which is defined over the field $k$, and a polarizing form $\langle,\rangle_B=(2\pi i)^{-n}\langle ,\rangle $ of the sheaves on the right of \eqref{eq:relativeisomorphism}. These two forms will be compatible with the relative isomorphism \eqref{eq:relativeisomorphism}, meaning that we have\begin{equation}\label{eq:relativecompatibility}
\langle	P^d_{X/S}(v),P^d_{X/S}(w)\rangle_{B}=\langle v,w\rangle_{DR},
\end{equation}holds for all sections $v,w$ of the sheaf on the right of \eqref{eq:relativeisomorphism}.\\

From now on we focus on the case $d=n$.
We choose $U\subset S$ a non-empty affine open subset. Then the form $\langle,\rangle_{DR}|_{U}$ will map, via the relative version of the Poincar\'e duality isomorphism, to a class $t_{DR} \in H^n_{DR}(X/S)\otimes_{\mathcal{O}_S}H^n_{DR}(X/S)|_{U}$. 

Similarly we repeat this process for the forms $\langle,\rangle$ and $\langle,\rangle_B$, over the analytification $U^{an}_{\C}$, to get elements $t\in (R^nf_{*}\Q\otimes_{\Q_{S^{an}_{\C}}}R^nf_{*}\Q)|_{U^{an}_{\C}}$ and $t_B\in(R^nf_{*}\Q\otimes_{\Q_{S^{an}_{\C}}}R^nf_{*}\Q)(n)|_{U^{an}_{\C}}$ with $t_B=(2\pi i)^{n}t$. 

Compatibility of Poincar\'e duality with the relative comparison isomorphism shows that $P^n_{X/S}\otimes P^n_{X/S}|_{U}(t_{DR})=t_B$. In other words the class $t_{DR}$ is a relative Hodge class thus defining polynomial relations among the relative $n$-periods.

Now we can repeat the arguments we made in the non-relative case. First, we choose $\{\omega_i\}$ a basis of section of $H^n_{DR}(X/S)$ over the affine open subset $U\subset S$ and $\{\gamma_j^{*}\}$ a frame of $R^nf_{*}^{an}\Q_{X^{an}_{\C}}|_{V}$, or equivalently a frame $\{\gamma_j\}$ of the relative homology $R_nf^{an}_{*}\Q_{X^{an}_{\C}}|_{V}$, where $V\subset U^{an}_\C$ is some open analytic subset. We get that the matrix $P_{X/S}=((2\pi i)^{-n} \int_{\gamma_j}\omega_i)$ satisfies \begin{equation}
	\label{eq:relativematrix1} P_{X/S} M_B \prescript{t}{}{P}_{X/S}=(2\pi i)^{-n} M_{DR},
\end{equation}where $M_B$ and $M_{DR}$ are the matrices of the forms $\langle,\rangle$ and $\langle,\rangle_{DR}$ with respect to the basis of section and the frame chosen above. 

The same process as before shows us that \eqref{eq:relativematrix1} is equivalent to the validity of the polynomial relations on the relative $n$-periods defined by the relative Hodge class $t_{DR}$. Finally, the same elementary argument as before shows the validity of the relative analogue of relation \eqref{eq:riemannwewant}, i.e. the Riemann relations that we use in \Cref{section:subsectionnotationsnontrivial}.

	\bibliographystyle{alpha}
	
	\bibliography{biblio}

\end{document}